\documentclass[a4paper,10pt, english, reqno]{smfart}

\usepackage{preamble-joaquin}

\author{Joaqu\'{i}n Rodrigues Jacinto}
\address{Universit\'e Aix-Marseille}
\email{joaquin.rodrigues-jacinto@univ-amu.fr}
\author{Chris Williams}
\address{University of Nottingham}
\email{chris.williams1@nottingham.ac.uk}

\title{An introduction to $p$-adic $L$-functions}
\date{}
\usepackage{appendix}

\newcommand{\lsem}{[\! [}
\newcommand{\rsem}{]\! ]}

\newcommand{\sD}{\mathscr{D}}
\newcommand{\sar}[2]{\ar@{}[#1]|-*[@]{#2}}

\newcommand{\cF}{\mathcal{F}}

\usepackage[all]{xy}
\usepackage{tikz-cd}
\usepackage{comment}
\usepackage{verbatim}
\newcommand{\GG}{\Gamma}
\newcommand{\Homc}{\mathrm{Hom}_{\mathrm{cts}}}
\newcommand{\GL}{\mathrm{GL}}

\newcommand{\sX}{\mathscr{X}}
\newcommand{\ccD}{\mathscr{D}}
\newcommand{\cD}{\mathcal{D}}
\newcommand{\sM}{\mathscr{M}}
\newcommand{\sC}{\mathscr{C}}
\newcommand{\cW}{\mathcal{W}}

\numberwithin{equation}{section}

\makeatletter
\def\@tocline#1#2#3#4#5#6#7{\relax
  \ifnum #1>\c@tocdepth % then omit
  \else
    \par \addpenalty\@secpenalty\addvspace{#2}%
    \begingroup \hyphenpenalty\@M
    \@ifempty{#4}{%
      \@tempdima\csname r@tocindent\number#1\endcsname\relax
    }{%
      \@tempdima#4\relax
    }%
    \parindent\z@ \leftskip#3\relax \advance\leftskip\@tempdima\relax
    \rightskip\@pnumwidth plus4em \parfillskip-\@pnumwidth
    #5\leavevmode\hskip-\@tempdima
      \ifcase #1
       \or\or \hskip 1em \or \hskip 2em \else \hskip 3em \fi%
      #6\nobreak\relax
    \dotfill\hbox to\@pnumwidth{\@tocpagenum{#7}}\par
    \nobreak
    \endgroup
  \fi}
\makeatother

\begin{document}
\thispagestyle{empty}
\maketitle

\renewcommand{\baselinestretch}{1.0}

\begin{abstract}
These expository notes introduce $p$-adic $L$-functions and the foundations of Iwasawa theory. We focus on Kubota--Leopoldt's $p$-adic analogue of the Riemann zeta function, which we describe in three different ways. We first present a measure-theoretic (analytic) $p$-adic interpolation of special values of the Riemann zeta function. Next, we describe Coleman's (arithmetic) construction via cyclotomic units. Finally, we examine Iwasawa's (algebraic) construction via Galois modules over the Iwasawa algebra.

The \emph{Iwasawa Main conjecture}, now a theorem due to Mazur and Wiles, says that these constructions agree. We will state the conjecture precisely, and give a proof when $p$ is a Vandiver prime (which conjecturally covers every prime). Throughout, we discuss generalisations of these constructions and their connections to modern research directions in number theory.
	%the $p$-adic $L$-functions encode beautiful congruences between the special values of $L$-functions, and have deep connections with the arithmetic of cyclotomic fields.
	%We will focus mainly on the construction and study of Kubota and Leopoldt's $p$-adic interpolation of the Riemann zeta function and on the ideas surrounding the \emph{Iwasawa Main conjecture}, now a theorem due to Mazur and Wiles. We will describe some classical results linking $L$-values and arithmetic data that led to the study of $p$-adic $L$-functions, and give several constructions of the $p$-adic zeta function. In particular, a construction due to Coleman using cyclotomic units will naturally lead to the statement of the Main conjecture, which we will prove when $p$ is a Vandiver prime (which conjecturally, at least, covers every prime). 
\end{abstract}

% \begin{quote}
% \small{\emph{These are notes from a graduate course at the London Taught Course Centre, November-December 2017. The course is to consist of five lectures, each of which is two hours long.\\
% \\
% In the course, we give an introduction to $p$-adic $L$-functions via the most classical results in the field. In particular, we give a construction of Kubota and Leopoldt's $p$-adic analogue of the Riemann zeta function and prove that it also interpolates the special values of Dirichlet $L$-functions. We then describe one of the most celebrated recent results in number theory, the \emph{Main conjecture of Iwasawa theory}, which gives a deep link between the arithmetic of cyclotomic fields and the analytic theory of $L$-functions. We conclude with an overview of some further topics in Iwasawa theory, including the $\mu$- and $\lambda$-invariants of a $\Zp$-extension.}}
% \end{quote}

\stepcounter{tocdepth}
\footnotesize
\tableofcontents
\normalsize

%\cite{CS06} gives a largely self-contained description of Rubin's (\cite[Appendix]{Lan90}) proof of the Main conjecture. The reader is urged to consult these references for further details.

%Specially including a general treatment of Iwasawa's theory of $\Zp$-extensions (Chapter 13); it also contains (yet more) alternative constructions of the Kubota--Leopoldt $p$-adic $L$-function.

\section{Introduction}\label{sec:intro}
The theory of $L$-functions, and their special values, has been central in number theory for 200 years. Their study spans from classical results, such as Gauss's class number formula and the proof of Dirichlet's theorem on primes in arithmetic progressions, to two major problems in mathematics: the Riemann hypothesis and Birch and Swinnerton-Dyer conjecture. They are also central in the Langlands program, a vast project connecting the fields of number theory, geometry and representation theory.

The Birch and Swinnerton-Dyer conjecture is one example of a huge network of conjectures on the special values of $L$-functions, including the Beilinson, Deligne, and Bloch--Kato conjectures. At their heart, these problems relate complex analytic information -- such as the order of vanishing and special values of meromorphic functions -- to arithmetic data, such as invariants attached to algebraic varieties and Galois representations. A fruitful approach to these problems has been the use of $p$-adic methods, for $p$ a prime number. Naively, one might consider the complex world a `bad' place to do arithmetic, as the integers are discrete in $\C$. This is not the case when one instead considers finite extensions of $\Qp$. The $p$-adic setting brings extra flexibility and methods with which to attack these open problems, including $p$-adic $L$-functions, Euler systems, and Hida and Coleman families of modular/automorphic forms. %Current state-of-the-art results towards Birch--Swinnerton-Dyer and Bloch--Kato rely on these $p$-adic tools. 

% The study of $p$-adic properties of special values of $L$-functions is generally known as \emph{Iwasawa theory}. At the heart of Iwasawa theory is \emph{Iwasawa's Main Conjecture}, a beautiful and surprising connection between analysis and arithmetic. This field is a fascinating, beautiful and highly active research area. However, it does have a reputation for being very technical. To even \emph{state} Iwasawa's Main Conjecture, for example, requires significant background, whilst $p$-adic $L$-functions can be intimidating for those learning the subject. 

The study of $p$-adic properties of special values of $L$-functions is generally known as \emph{Iwasawa theory}. In these notes, we give an introduction to this subject, focusing on perhaps the most fundamental of all $L$-functions: the Riemann zeta function $\zeta(s)$. We describe what a $p$-adic $L$-function is, construct it in this setting, and then describe Iwasawa's Main Conjecture\footnote{Iwasawa's original conjecture  was proved in full by Mazur and Wiles in \cite{MW84}. However analogous conjectures, relating Selmer groups and $p$-adic $L$-functions, have been formulated in a large generality, for example for elliptic curves, modular forms, and beyond. These are also (somewhat confusingly) referred to as `Iwasawa Main Conjectures', even in the special cases (such as the one we consider) where they have been proved. We discuss such generalisations in \S\ref{sec:greenberg selmer} and Appendix \ref{sec:modular forms}.} in this case. We try to anchor the theory in the context of current research activity, indicating how the various concepts we discuss have been generalised, and where the reader should turn next to learn more.

\subsection{What do we cover in these notes?} We now summarise the main results we cover. In \S \ref{sec:riemann to kubota}, we give a broad introduction to $p$-adic $L$-functions, with an emphasis on how one can naturally move from complex to $p$-adic $L$-functions. We make this precise in our case of interest by stating some of the main results of Part I. 

Our focus for the rest of the notes is on the Kubota--Leopoldt $p$-adic $L$-function (or $p$-adic zeta function), which is the $p$-adic analogue of the Riemann zeta function. We will see three constructions of this object, each of a different flavour, and describe the connections between them.

Part I is devoted to the construction and study of an \emph{analytic} version of the Kubota--Leopoldt $p$-adic $L$-function. This has the clearest connection to the classical complex Riemann zeta function; it is a pseudo-measure $\zeta_p^{\mathrm{an}}$ on $\Zp^\times$ that interpolates the (rational numbers) $\zeta(1-k)$ for all positive integers $k$ \footnote{ The precise meaning of this will be clear later. Here the word interpolation is as in Lagrange's interpolation formula, i.e.\ a single object that hits certain specific values when evaluated at various points.}.

\begin{itemize}

\item[$\bullet$]  
 In \S\ref{sec:measures}, we describe some basic tools and results from $p$-adic analysis needed to parse this statement, including measures/pseudo-measures, Iwasawa algebras, and their connections to power series. 

\item[$\bullet$] In \S \ref{kl} we use the techniques developed in \S\ref{sec:measures} to prove Theorem \ref{thm:kubota-leopoldt} (see also Theorem \ref{thm:kubota leopoldt theorem}) on the existence of the Kubota--Leopoldt $p$-adic $L$-function $\zeta_p^{\mathrm{an}}$. 

\item[$\bullet$] In \S \ref{interpolation} we prove that $\zeta_p^{\mathrm{an}}$ also interpolates special values of $L$-functions of Dirichlet characters of $p$-power conductor, and construct analogues for arbitrary Dirichlet characters.

\item[$\bullet$] In \S\ref{sec:s=1}, we describe a result of Leopoldt, showing that the values of the $p$-adic $L$-function of a non-trivial Dirichlet character $1$ are related to logarithms of cyclotomic units, establishing one of the first instances of the $p$-adic Beilinson conjectures. The (untwisted) $p$-adic zeta function has a simple pole at $s=1$; in \S\ref{sec:residue}, we prove an analogous result describing its residue.

\item[$\bullet$] Finally in \S \ref{sec:eisenstein} we discuss an approach to $p$-adic $L$-functions based on the theory of families of Eisenstein series.
\end{itemize}

In Part II, we give two more constructions of the Kubota--Leopoldt $p$-adic $L$-function: \emph{arithmetic} and \emph{algebraic} versions.
\begin{itemize}
	\item[$\bullet$] In \S\ref{sec:coleman map}, we give an arithmetic construction. The \emph{cyclotomic units} are special elements in cyclotomic fields. As one considers the tower $\Qp(\mu_{p^n})$ of cyclotomic extensions of $\Qp$, the cyclotomic units fit together into a norm-compatible tower/system. The \emph{Coleman map} is a map that attaches a $p$-adic measure to any such tower of units. Via this process, we show that to the cyclotomic units, one can naturally attach a pseudo-measure $\zeta_p^{\mathrm{arith}}$ on $\Zp^\times$. One connection between arithmetic and analysis, fully explained in \S\ref{sec:coleman map}, is that $\zeta_p^{\mathrm{an}} = \zeta_p^{\mathrm{arith}}$.
 
    \item[$\bullet$] In \S\ref{sec:iwasawa zeros} and \S\ref{sec:proof Iwasawa}, we deepen the arithmetic picture, respectively stating and proving \emph{Iwasawa's theorem} describing the zeros of the $p$-adic zeta function via modules of cyclotomic units.
 
	\item[$\bullet$] In \S\ref{sec:IMC}, we build on Iwasawa's theorem to give an algebraic construction of the Kubota--Leopoldt $p$-adic $L$-function, as an ideal $\zeta_p^{\mathrm{alg}}$ in the Iwasawa algebra. This ideal arises from the structure of a Galois module over the Iwasawa algebra. We also  describe this module in terms of Selmer groups and discuss generalisations of the Main Conjecture.
\end{itemize}
	
	The \emph{Iwasawa Main Conjecture} says this ideal is generated by the analytic/arithmetic Kubota--Leopoldt $p$-adic $L$-function, connecting the analytic, arithmetic and algebraic constructions, and ultimately connecting special complex $L$-values and Selmer groups. We state this precisely in \S\ref{sec:IMC}, and prove it in the special case where $p$ is a Vandiver prime.

  The reader interested in taking a minimal path to the Iwasawa Main Conjecture can do so by reading the following sections: \S\ref{subsec:classicalL}, \S\ref{sec:riemann}--\ref{sec:p-adic l-functions}, \S\ref{sec:p-adic measures 1}--\ref{sec:pseudo-measures}, \S\ref{kl}, \S\ref{sec:notation coleman}--\ref{sec:coleman map definition}, \S\ref{sec:iwasawa zeros}, \S\ref{sec:proof Iwasawa}, \S\ref{sec:structure theorems}--\ref{sec:vandiver proof}.

 In Appendix \ref{sec:modular forms}, we conclude with remarks on how the analytic, arithmetic and algebraic constructions above have been generalised to other situations, each spawning a field of study in their own right. We illustrate this by giving a sketch in the case of modular forms.

\subsection*{Further reading}
For more information and detail on Part I of these notes, the reader could consult \cite{Lan90}. The construction of the $p$-adic zeta function we give here is based on Colmez's beautiful lecture notes \cite{ColCourse} (in French).

Part II can serve as a prelude to a number of more advanced treatments, such as Rubin's (complete) proof of the Main Conjecture using the theory of Euler systems. We must mention the book of Coates and Sujatha \cite{CS06}, which inspired our original course, and whose aim was to present Rubin's proof. A canonical book in the field is \cite{Washington2}, which introduces further topics in classical Iwasawa theory that there was not space to treat here. We give a flavour of such topics in Appendix \ref{sec:mu invariant}.

%We summarise generalisations of this theory to $\mathrm{GL}(2)$ (the case of modular forms) in Appendix \ref{sec:modular forms}. Since this is of fundamental interest in modern research, we intend to make the $\mathrm{GL}(2)$ case the subject of a sequel set of notes. In this sequel, we will describe in detail the analytic construction of the $p$-adic $L$-function of a modular form via overconvergent modular symbols, due to Pollack and Stevens \cite{PS11}, and sketch the work of Kato \cite{Kat04} on the arithmetic and algebraic side. Further topics of interest in this direction include the close connection between $p$-adic $L$-functions and $p$-adic Hida/Coleman families/the Coleman--Mazur eigencurve, and for these topics, the reader is urged to consult Bella\"iche's \emph{The Eigenbook} \cite{Eigenbook}. %For further reading in this direction, the reader could consult Bella\"iche's \emph{The Eigenbook} \cite{Eigenbook}, which describes the close connection between $p$-adic $L$-functions and $p$-adic Hida/Coleman families and the Coleman--Mazur eigencurve, and Kato's original paper \cite{Kat04}, where he constructs an Euler system for any modular form, shows the explicit reciprocity law, and deduces one divisibility of Iwasawa's Main conjecture in this case.

\subsection*{Acknowledgements}
These notes started life as the lecture notes for a course at the London Taught Course Centre in 2017. We thank the organisers of the LTCC, and the participants of that course, for their attention and  enthusiasm. We would also like to thank Martin Baric, Keith Conrad, David Corwin and Luis Santiago for their comments and corrections on earlier drafts of these notes.  We first learnt of this construction of the Kubota--Leopoldt $p$-adic $L$-function from Pierre Colmez's notes \cite{ColCourse}, and we are grateful to him for allowing us to reproduce them here. We are also grateful to the three anonymous referees, whose  careful reading and insightful suggestions and corrections greatly improved this text.

\section{What is a $p$-adic $L$-function, and what should it do?}\label{sec:riemann to kubota}

This introductory section aims to motivate the definition and study of $p$-adic $L$-functions. We start with a general discussion on complex $L$-functions and then lean slowly towards the $p$-adic world, focusing on the example of most importance to us in these lectures: the Riemann zeta function.

\subsection{Classical $L$-functions} \label{subsec:classicalL} We first give some important examples of $L$-functions.

\begin{itemize}
	\item The \emph{Riemann zeta function}, the most famous and fundamental of all $L$-functions, is defined by
	\[
	\zeta(s) = \sum_{n\geq 1}n^{-s} = \prod_{p} (1 - p^{-s})^{-1}, 
	\] 
	where the last product -- an \emph{Euler product} -- runs over all prime numbers $p$ and the second equality is a consequence of the unique prime factorisation of integers. The sum converges absolutely whenever $s$ is a complex variable with real part greater than 1, making $\zeta$ a well-defined holomorphic function in a right half-plane $\{ s \in \C \, : \, \mathrm{Re}(s) > 1\}$. It can be meromorphically continued to the whole complex plane, and satisfies a \emph{functional equation}, a symmetry relating the values $\zeta(s)$ and $\zeta(1-s)$. 
	
	\item Let $F$ be a number field. The \emph{zeta function of $F$} is
	\[
	\zeta_F(s) \defeq \sum_{0\neq I \subset \roi_F}N(I)^{-s} = \prod_{\mathfrak{p}} (1 - N(\mathfrak{p})^{-s})^{-1}, 
	\]
	where the sum is over all non-zero ideals in the ring of integers, and the product is over all non-zero prime ideals of $K$. Again, this sum converges absolutely for $\mathrm{Re}(s) > 1$, can be meromorphically continued to $\C$, and satisfies a functional equation relating $\zeta_F(s)$ and $\zeta_F(1-s)$. The existence of the Euler product again follows from unique factorisation.

	\item Let $\chi : (\Z/N\Z)^\times \rightarrow \C^\times$ be a Dirichlet character. We extend $\chi$ to a function $\chi : \Z \rightarrow \C$ by setting $\chi(m) = \chi(m \newmod{N})$ if $m$ is prime to $N$, and $\chi(m) = 0$ otherwise. The $L$-function of $\chi$ is
	\[ 
	L(\chi,s) \defeq \sum_{n\geq 1}\chi(n) n^{-s} = \prod_p (1 - \chi(p)p^{-s})^{-1}.
	\]
	Yet again, the above sum defining $L(\chi, s)$ converges for $\mathrm{Re}(s)>1$, admits meromorphic continuation to $\C$ (analytic when $\chi$ is non-trivial), and satisfies a functional equation relating the values at $s$ and $1-s$.
	
	\item Let $E/\Q$ be an elliptic curve of conductor $N$. One can define an $L$-function
	\[
	L(E,s) \defeq \sum_{n\geq 1}a_n(E)n^{-s} = \prod_{p \nmid N} (1 - a_p(E) p^{-s} + p^{1 - 2s})^{-1} \prod_{p | N} L_p(s) ,
	\]
	where $a_p(E) = p + 1 - \# E(\mathbf{F}_p)$, and the $a_n(E)$ are defined recursively from the $a_p(E)$. The factors $L_p(s)$ at bad primes $p|N$ are defined as $L_p(s) = 1$ (resp. $(1 - p^{-s})^{-1}$, resp. $ (1 + p^{-s})^{-1}$) if $E$ has bad additive (resp. split multiplicative, resp. non-split multiplicative) reduction at $p$. The above sum defining the function $L(E,s)$ converges for $\mathrm{Re}(s) > 3/2$, admits analytic continuation to $\C$, and satisfies a functional equation relating the values at $s$ and $2-s$.
	
	\item Let 
	\[
	f = \sum_{n \geq 1} a_n(f) q^n \in S_k(\Gamma_0(N), \omega_f)
	\]
	be a modular newform of weight $k$, level $N$ and character $\omega_f$. The $L$-function associated to $f$ is given by
	\[ 
	L(f, s) \defeq \sum_{n \geq 1} a_n(f) n^{-s} = \prod_{p \nmid N} (1 - a_p(f) p^{-s} + \omega_f(p) p^{k - 1 - 2s})^{-1} \prod_{p | N} (1 - a_p(f) p^{-s})^{-1}. 
	\]
	This sum converges for $\mathrm{Re}(s) > (k+1)/2$, admits analytic continuation to $\C$, and satisfies a functional equation relating the values at $s$ and $k-s$. Such objects are introduced, and these results proved, in \cite[\S5]{DS05}.
\end{itemize}

The above examples share common features. The `arithmetic' $L$-functions of most interest to us should have the following basic properties (which can, nevertheless, be extremely deep): 
\begin{itemize}
	\item[(1)] An Euler product converging absolutely in a right-half plane;
	\item[(2)] A meromorphic continuation to the whole complex plane;
	\item[(3)] A functional equation relating $s$ and $k-s$ for some $k \in \R$.
\end{itemize}

\begin{remark}
More generally, let $\mathscr{G}_{\Q} = \mathrm{Gal}(\overline{\Q} / \Q)$ denote the absolute Galois group of $\Q$ and let $V \in \mathrm{Rep}_L \, \mathscr{G}_{\Q}$ be a $p$-adic Galois representation (i.e.\ a finite dimensional vector space over a finite extension $L$ of $\qp$ equipped with a continuous linear action of $\mathscr{G}_{\Q}$). One defines the global $L$-function of $V$ as a formal Euler product
	\[ 
	L(V, s) = \prod_{\ell} L_\ell(V, s)
	\] 
 of local factors. For $\ell \neq p$ a rational prime, the local factor at $\ell$ is defined as
	\[ 
	L_\ell(V, s) \defeq \det(\mathrm{Id} - \mathrm{Frob}_\ell^{-1} \ell^{-s} | V^{I_\ell})^{-1}, 
	\]
	where $\mathrm{Frob}_\ell$ denotes the arithmetic Frobenius at $\ell$, and $I_\ell$ denotes the inertia group at $\ell$ (all described in \cite[\S I]{Serre89}). At $p$, defining the local factor is considerably more complicated, requiring $p$-adic Hodge theory, as described in \cite{Ber02,Ben22}. In this case, one defines 
	\[ 
	L_p(V, s) \defeq \det(\mathrm{Id} - \varphi^{-1} p^{-s} | \mathbf{D}_{\mathrm{cris}}(V))^{-1},
	\]
	where $\mathbf{D}_{\mathrm{cris}}(V)$ denotes the crystalline module of $V|_{\mathscr{G}_{\Q_p}}$, equipped with a crystalline Frobenius denoted by $\varphi$. 
 
	When $V$ is the representation attached to an arithmetic object\footnote{For example, a number field, a Dirichlet character, an elliptic curve, a modular form, or much more generally -- in the spirit of the Langlands program -- an automorphic representation of a reductive group.}, the $L$-function of the representation is typically equal to the $L$-function attached to that object; for example, taking $V = \qp(\chi)$ (that is, $V$ is 1-dimensional, with $\mathscr{G}_{\Q}$ acting through the character $\chi$ via class field theory), one recovers the Dirichlet $L$-functions described above. See  \cite{Bel09} for further introductions to these topics.
\end{remark}

\subsection{Motivating questions for Iwasawa theory}

\subsubsection{Special values and arithmetic data}
%There is much interest in the \emph{special values} of $L$-functions.
There are deep results and conjectures relating special values of $L$-functions to important arithmetic information, of which a prototypical example is the following (cf., for example, \cite[\S 5]{NeuirchANT}):

\begin{theorem}[Class number formula]\label{thm:class number formula}
	Let $F$ be a number field with $r_1$ real embeddings, $r_2$ pairs of complex embeddings, $w$ roots of unity, discriminant $D$, and regulator $R$. The zeta function $\zeta_F$ has a simple pole at $s=1$ with residue
	\[\mathrm{res}_{s=1}\zeta_F(s) = \frac{2^{r_1}(2\pi)^{r_2} R}{w\sqrt{|D|}}h_F,\]
	where $h_F$ is the class number of $F$.
\end{theorem}
On the left-hand side, we have a special value of a complex
meromorphic function, from the world of analysis. On the right-hand side, we have invariants attached to a number field, from the world of arithmetic. The class number formula thus provides a deep connection between two different fields of mathematics.

A second famous example of such a connection comes in the form of the \emph{Birch and Swinnerton--Dyer (BSD) conjecture}. Let $E/\Q$ be an elliptic curve. The set of rational points $E(\Q)$ forms a finitely generated abelian group and, based on computer computations, Birch and Swinnerton--Dyer predicted that 
\[
\ord_{s=1}L(E,s) = \mathrm{rank}_{\Z}E(\Q).
\]
They also predicted a closer analogue of the class number formula: that the leading term of the $L$-function can be described in terms of arithmetic invariants attached to $E$.

Again, the left-hand side is from the world of analysis, the right-hand side is from the world of arithmetic, and this prediction is inherently surprising. The worlds are so different that the analytic $L$-function defies easy study using arithmetic properties of the elliptic curve. For example, when the conjecture was formulated, the left-hand side was not even known to exist: nobody had proved that $L(E,s)$ was defined at the value $s=1$. This relies on analytic continuation of the $L$-function; such a proof would not follow for several decades, and even now the only proof we have goes through another deep connection between arithmetic and analysis, namely Wiles' modularity theorem. 

\subsubsection{Iwasawa Main Conjectures}

%The full BSD conjecture remains open.
One of the goals of Iwasawa theory is to seek and prove $p$-adic analogues of BSD and its generalisations, replacing complex analysis (which is poorly suited to arithmetic) with $p$-adic analysis (where arithmetic arises naturally). For each prime $p$, there is a $p$-adic \emph{Iwasawa Main Conjecture (IMC)} for the elliptic curve $E$, relating a $p$-adic analytic $L$-function to certain $p$-adic arithmetic invariants of $E$:
\[
\xymatrix@C=25mm{
	\fbox{complex analytic $L$-function} \ar@{<-->}[r]^-{\text{BSD}}\ar@{<-->}[d] & \fbox{arithmetic invariants of $E$}\ar@{<-->}[d]\\
	\fbox{$p$-adic analytic $L$-function} \ar@{<->}[r]^-{\text{IMC}} & \fbox{$p$-adic invariants of $E$}
}
\]
One has many more tools available to attack the bottom row than the top, including Euler systems, $p$-adic families and eigenvarieties, $p$-adic Hodge theory and $(\varphi,\Gamma)$-modules, and more.  As a result, the $p$-adic conjectures are much more tractable than their complex counterparts. Indeed, whilst BSD remains open beyond low rank special cases, the IMC for elliptic curves has been proved much more widely by Skinner--Urban (see \cite{SU14}), following work of Kato (see \cite{Kat04}). See also \cite{FouquetWan} for a recent summary of known results on the IMC for elliptic curves.

\subsubsection{Applications of $p$-adic methods to classical BSD}
Each new $p$-adic Iwasawa Main Conjecture that is proved brings the worlds of analysis and arithmetic a little closer together. They can also bring us closer to our original goal of, for example, BSD. Indeed, the current state-of-the-art results towards BSD  have arisen as consequences of Iwasawa theory. 

To elaborate, let us summarise some fundamental work ($p$-adic and otherwise) on BSD. There are two natural subquestions.
\begin{itemize}
	\item[(a)] We could try to prove that $\ord_{s=1} L(E,s) \leq \mathrm{rank}_{\Z} E(\Q).$ A natural approach is to try to construct enough independent rational points on the elliptic curve. 
  The theory of \emph{Heegner points} is based on such an idea. A Heegner point in $E(\Q)$ has infinite order if and only if $\mathrm{ord}_{s=1}L(E,s) = 1$. This yields the above inequality in this case.
	
	\item[(b)] Conversely, we could try and prove that $\ord_{s=1} L(E,s) \geq \rank_{\Z} E(\Q).$ In this case we want to bound the number of rational points. One method for trying to do this uses \emph{Euler systems} attached to $p$-adic Galois representations (see \cite{Rub00} for a comprehensive introduction). The main application of Euler systems is in bounding certain Galois cohomology groups, known as \emph{Selmer groups} (cf.\ \S \ref{sec:greenberg selmer}), which are defined using local behaviour and can be viewed as a cohomological interpretation of the group of rational points on $E$. Indeed, 
%The difference between the Selmer group and $E(\Q)$ is captured in the \emph{Tate--Shafarevich group} $\Sha(E/\Q)$, a torsion abelian group that is conjecturally finite.
let $\Sha(E/\Q)$ denote the \emph{Tate--Shafarevich group} of $E$, a torsion abelian group that is conjecturally finite.
  If the subgroup $\Sha(E/\Q)[p^\infty] \leq \Sha(E/\Q)$ of elements with $p$-power order is finite, then the $\Z$-rank of $E(\Q)$ is equal to the $\Zp$-corank\footnote{That is, the rank of the Pontryagin dual of the $p$-Selmer group; see \cite[\S2.1.4]{Ski18}.} of the $p$-Selmer group. In this case, bounding the $p$-Selmer group is equivalent to bounding $E(\Q)$. 
\end{itemize}

The ideas above have led to special cases of the conjecture; in particular, we now know it to be true (under some assumptions) when $\ord_{s=1}L(E,s) \leq 1$ due to work of Kolyvagin, Gross--Zagier and Murty--Murty (see \cite{Kol88}, \cite{GZ86} and \cite{MM91}). More recent Iwasawa-theoretic research building on the above has led to results towards the converse \cite{Ski20}, as well as towards the leading term formula \cite{JSW15}. 

We emphasise that whilst these methods have yielded important progress towards BSD, to date such results have been fundamentally limited to elliptic curves defined over $\Q$ and of rank $\leq 1$. It is natural to try to execute such strategies in more general settings. More recently, the $p$-adic theory of \emph{Stark--Heegner points}, a $p$-adic analogue of (a) initiated in \cite{Dar01}, has been used with some success for elliptic curves over totally real fields. Heegner and Stark--Heegner points are beautifully summarised in \cite{DarmonBook} and \cite{DarmonICM}. A more recent overview of work on Heegner and Stark--Heegner points, and the relationships between them, is given in \cite{DiagonalAsterisque}.  There has also been encouraging (and fundamentally $p$-adic) work on analogous questions in rank 2 and beyond; for example, see \cite{DR16,CastellaHsieh}.

\begin{remark}
The study of $p$-adic $L$-functions is intrinsic to (b), and they also feature prominently in the $p$-adic analogues of (a). Mazur, Tate and Teitelbaum formulated a $p$-adic  BSD conjecture (see \cite{MTT86} and also \cite{ColmezBSDpadique}), which relates the order of vanishing of a $p$-adic $L$-function at $s=1$ to the rank of the rational points of the elliptic curve, and expresses its principal coefficient in terms of arithmetic data in a manner analogous to the classical BSD conjecture (replacing the complex regulator by a $p$-adic regulator).
 
  For an elliptic curve over $\Q$ of analytic rank $0$, we know that the order of vanishing of its attached $p$-adic $L$-function is always 0 or 1. The possible extra zero, discussed in \cite{MTT86, GS93}, is known as a \textit{trivial zero} of the $p$-adic $L$-function and is well understood in terms of local data attached to $E$ at $p$. If the Tate--Shafarevich group $\Sha(E/\Q)$ is finite, the $p$-adic and classical BSD conjectures are equivalent in this case.
 
 Under the assumption of the non-degeneration of the $p$-adic height pairing, the $p$-adic IMC for elliptic curves implies the $p$-adic BSD conjecture (see \cite{Schneiderpadicheight}). 
\end{remark}

\subsubsection{The IMC for the Riemann zeta function}
We mention the elliptic curve case only to motivate the study of $p$-adic $L$-functions and Iwasawa theory. In these notes, we will focus on a simpler example of the above picture, namely the Main Conjecture for the $p$-adic Riemann zeta function, as formulated by Iwasawa himself. Here the picture above is essentially complete; the full IMC is known for any prime $p$ (thanks to \cite{MW84} for $p$ odd, and \cite{Wil90} for $p=2$). We will (for odd $p$) construct the $p$-adic analogue of the zeta function on the way to stating the Main Conjecture, which we will prove for a special case.

\subsection{The Riemann zeta function}\label{sec:riemann}
Since the Riemann zeta function will be a central player in the rest of these notes, we take a brief detour to describe some of the classical theory surrounding it. We start with the following general result. 

\begin{theorem}\label{thm:l-function}
Let $f : \R_{\geq 0} \longrightarrow \R$ be a rapidly decreasing $\mathscr{C}^\infty$-function (i.e.\ such that $f$ and all of its derivatives $f', f'', \ldots$ decay exponentially at infinity). Let 
\begin{equation} \label{EqGamma}
\Gamma(s) = \int_0^\infty e^{-t}t^{s-1} dt.
\end{equation}
	be the usual Gamma function. The function
	\[L(f,s) \defeq \frac{1}{\Gamma(s)}\int_0^\infty f(t)t^{s-1}dt, \hspace{12pt} s \in \C,\]
	which converges to a holomorphic function for Re$(s) > 0$, has an analytic continuation to the whole complex plane, and
	\[L(f,-n) = (-1)^n \frac{d^n}{dt^n}f(0).\]
	We call $L(f,s)$ the \emph{Mellin transform} of $f$.
\end{theorem}
\begin{proof}
	To show analytic continuation, we claim that when Re$(s) > 1$, we have
	\[
	L(f,s) = -L(f',s+1),
	\]
	where $f' = df/dt$. This is an exercise in integration by parts, using the identity $\Gamma(s) = (s-1)\Gamma(s-1)$, and gives the analytic continuation to all of $\C$ by iteration. Finally, iterating the same identity $n+1$ times shows that
	\begin{align*}
		L(f,-n) &= (-1)^{n+1}L(f^{(n+1)} ,1)\\
		& = (-1)^{n+1}\int_0^\infty f^{(n+1)}(t)dt = (-1)^nf^{(n)}(0)
	\end{align*}
	from the fundamental theorem of calculus, giving the result.
\end{proof}

We would like to use the Mellin transform to recover the Riemann zeta function and its properties. For this, we pick a specific choice of $f$, namely, we let
\[f(t) = \frac{t}{e^t - 1} = \sum_{n \geq 0}B_n \frac{t^n}{n!},\]
the generating function for the Bernoulli numbers $B_n$. 
\begin{remark}
The Bernoulli numbers are highly combinatorial, and satisfy  recurrence relations that ensure they are \emph{rational} numbers; for example, the first few are
\[
B_0 = 1, B_1 = -\frac{1}{2}, B_2 = \frac{1}{6}, B_3 = 0, B_4 =  -\frac{1}{30}, ...
\]
For $k \geq 3$ odd, $B_k = 0$.
\end{remark}

We want to plug this function into Theorem \ref{thm:l-function}, and for this, we require\footnote{We thank Keith Conrad for pointing out this proof.}:
\begin{lemma}
	The function $f(t)$ and all of its derivatives decay exponentially at infinity. 
\end{lemma}

\begin{proof}
	For $t >0$, we may expand $f(t)$ as a geometric series
	\begin{align*}
		f(t) &= t(e^{-t} + e^{-2t} + e^{-3t} + \cdots) =: t F(t).
	\end{align*}
	Note that $f'(t) = F(t) + tF'(t)$, and $f''(t) = 2F'(t) + tF''(t)$; arguing inductively we see
	\begin{align*}
		f^{(n)}(t) &= nF^{(n-1)}(t) + tF^{(n)}(t) \\
		&= n (-1)^{n-1} (e^{-t} + 2^{n-1} e^{-2t} + 3^{n-1} e^{-3t} + \cdots) + (-1)^{n}t(e^{-t} + 2^ne^{-2t} + 3^n e^{-3t} + \cdots)\\
		&\sim (-1)^nte^{-t} \qquad \text{ as }t \to \infty.
	\end{align*}
	This decays exponentially.
\end{proof}

\begin{lemma} \label{lem:FormulaZeta}
	For the choice of $f$ as above, we have
	\[(s-1)\zeta(s) = L(f,s-1).\]
\end{lemma}
\begin{proof}
Substituting $t$ for $nt$ and rearranging in Equation \eqref{EqGamma} defining $\Gamma(s)$, we obtain 
	\[n^{-s} = \frac{1}{\Gamma(s)}\int_0^\infty e^{-nt}t^{s-1}dt.\]
	Now, when $\mathrm{Re}(s)$ is sufficiently large, we can write
	\[ \zeta(s) = \sum_{n\geq 1} n^{-s} = \frac{1}{\Gamma(s)}\sum_{n\geq 1}\int_0^\infty e^{-nt}t^{s-1}dt = \frac{1}{\Gamma(s)}\int_0^\infty \bigg( \sum_{n\geq 1} e^{-nt} \bigg) t\cdot t^{s-2}dt, \]
	%\begin{align*}\zeta(s) &= \sum_{n\geq 1} n^{-s}\\
	%&= \frac{1}{\Gamma(s)}\sum_{n\geq 1}\int_0^\infty e^{-nt}t^{s-1}dt\\
	%&= \frac{1}{\Gamma(s)}\int_0^\infty \left[\sum_{n\geq 1} e^{-nt}\right]t\cdot t^{s-2}dt,
	%\end{align*}
	and the result now follows from the identity
	\[\frac{1}{e^t - 1} = \sum_{n \geq 1} e^{-nt}. \qedhere \]
\end{proof}

From the theorem above, we immediately obtain:

\begin{corollary}\label{cor:rational}
	For $n\geq 0$, we have
	\[\zeta(-n) = -\frac{B_{n+1}}{n+1}.\]
	In particular, $\zeta(-n) \in \Q$ for $n\geq 0$, and $\zeta(-n) = 0$ if $n \geq 2$ is even.
\end{corollary}

\subsection{$p$-adic $L$-functions}\label{sec:p-adic l-functions}
As explained in the introduction, $p$-adic $L$-functions are excellent tools to study special values of $L$-functions. 
%We have already seen two examples where special values of $L$-functions should be able tell us information about arithmetic objects. In fact, there are very general conjectures (for example, the Bloch--Kato and Beilinson conjectures) predicting that these links exist on a much wider scale, and despite some partial results in special cases these conjectures remain deep open problems. Much of what we \emph{do} know about these conjectures comes through the theory of \emph{$p$-adic $L$-functions.} 
In this section, we explain what a $p$-adic $L$-function is and the properties it should satisfy.

%The general idea behind $p$-adic $L$-functions is to bring $p$-adic analysis into the study of special values of $L$-functions. The values at odd negative integers are rational, and so can be viewed as $p$-adic numbers. A flavour of the theory can be found in \emph{Kummer's congruences}:
%\begin{theorem}[Kummer]
%Let $m$ and $n$ be odd positive integers with $n \equiv m \not\equiv -1 \newmod{p-1}$. Then 
%\[\zeta(-n) \equiv \zeta(-m) \newmod{p}.\]
%\end{theorem}
%So not only are the values $p$-adic, they also satisfy nice properties $\newmod{p}$. The classical zeta function is a function $\zeta : \C \rightarrow \C$; it's now natural to ask whether we can `$p$-adicise' this picture. How might such a thing work?

\subsubsection{$p$-adic $L$-functions, a first idea}

The complex $\zeta$-function is a complex analytic function
\[\zeta : \C \longrightarrow \C\]
which is rational at negative integers. Since $\Z$ is a common subset of both $\C$ and $\zp \subseteq \cp$, it is natural to ask if there exists a function
\[
\zeta_p : \Zp \longrightarrow \Cp
\]
that is `$p$-adic analytic' (in some sense to be defined) and which agrees with the complex $L$-function at negative integers in the sense that
\begin{equation}\label{eq:naive interpolation}
\zeta_p(1-n) = (*) \cdot \zeta(1-n),
\end{equation}
for some explicit factor $(*)$. We would say that such a function `$p$-adically interpolates the special values of $\zeta(s)$'. Ideally, one would like these properties to uniquely characterise $\zeta_p$.

\subsubsection{Ideles, measures and Tate's thesis} \label{sec:dirichlet ideles}

In practice, there is no \emph{single} analytic function on $\Zp$ that interpolates all of the special values\footnote{Rather, there are $p-1$ different analytic functions $\zeta_{p,1}, \hspace{1pt}..., \hspace{1pt}\zeta_{p,p-1}$ on $\Zp$, and $\zeta_{p,i}$ interpolates only the values $\zeta(1-k)$ for which $k \equiv i \newmod{p-1}$.}, as we shall explain in Section \ref{mellin}. Instead, a better way of thinking about $L$-functions is to use a viewpoint initiated by Tate in his thesis \cite{Tat50} (and later independently by Iwasawa; see \cite{Iwa52}). This viewpoint sees $L$-functions as \emph{measures on ideles}, and allows one to package together \emph{all} Dirichlet $L$-functions, including the Riemann zeta function, into a single object. We will give a brief account of the classical theory here, but for fuller accounts, one should consult the references above.

We begin with some observations on characters.

\begin{proposition} \label{prop:dirichlet ideles}  The following assertions hold.
\begin{itemize}
\item[(i)] There is an identification between Dirichlet characters $\chi$ and continuous characters\[ \chi : \prod_{\ell\hspace{1pt}\mathrm{prime}} \Z_\ell^\times \longrightarrow \C^\times, \]
where the source is equipped with the product of the $\ell$-adic topologies.
\item[(ii)]
There is an identification of $\C$ with the space $\Homc(\R_{>0},\C^\times)$ of  continuous multiplicative characters by sending $s$ to $x \mapsto x^s$.
\end{itemize}

In particular, each pair $(\chi,s)$, where $\chi$ is a Dirichlet character and $s\in\C$, corresponds to a (unique) continuous character
	\begin{align*}\kappa_{\chi,s} : \R_{>0} \times \prod_{\ell\hspace{1pt}\mathrm{prime}} \Z_\ell^\times &\longrightarrow \C^\times\\
		(x,y) &\longmapsto x^s\chi(y),
	\end{align*}
where we equip the source with the product topology, and all continuous characters on this group are of this form.
\end{proposition}

\begin{proof}
First, observe that any Dirichlet character $\chi : (\Z/N\Z)^\times \longrightarrow \C^\times$  induces naturally a character
		\[\chi : \prod_{\ell\hspace{1pt}\mathrm{prime}} \Z_\ell^\times \longrightarrow \C^\times.\]
Indeed, suppose first that $N = \ell^n$ is a power of some prime $\ell$, with $n \geq 1$. As $(\Z/\ell^n\Z)^\times \cong \Z_\ell^\times/(1+\ell^n\Z_\ell)$, we can lift $\chi$ from $(\Z/\ell^n\Z)^\times$ to a function on $\Z_\ell^\times$. 
The case for general $N$ follows from the Chinese remainder theorem.  Conversely, any continuous character $\chi: \prod_{\ell\hspace{1pt}\mathrm{prime}} \Z_\ell^\times \to \C^\times$ must factor through a finite quotient $(\Z / N \Z)^\times$ of $\prod_{\ell\hspace{1pt}\mathrm{prime}} \Z_\ell^\times$ for some large enough $N$. Indeed, the image of a sufficiently small open neighbourhood 
\[
    U_N\defeq \Big\{ x \in \prod_{\ell} \Z_\ell^\times \; : \; x \equiv 1 \text{ (mod $N$)} \Big\}
\]
of $1$ is contained in $\{ z \in \C : |z - 1| < 1\}$. This image is a compact subgroup, but the only compact subgroup of the latter set is $\{1\}$. Hence $\chi$ is trivial on $U_N$, and factors through $(\prod_\ell \Z_\ell^\times) / U_N = (\Z / N \Z)^\times$, inducing a Dirichlet character. These two procedures are inverse to each other, showing the first point.

We now prove the second point. For $s \in \C$, the function $x \mapsto x^s$ is visibly a continuous character on $\R_{>0}$. We want to show that all such characters are of this form. After taking a logarithm, this is equivalent to showing that all continuous homomorphisms (of additive groups) $g : \R \rightarrow \C$ are of the form $g(x) = xg(1)$, which is shown by directly computing the values of $g$ on $\Q$ and extending by continuity.
\end{proof}

The product space appearing in Proposition \ref{prop:dirichlet ideles} is more usually written using ideles.

\begin{definition}
	Define the \emph{ideles} $\A^\times$ of $\Q$ to be
	\begin{align*}\A^\times &= \A_{\Q}^\times \defeq \R^\times \times \sideset{}{'}\prod_{\ell\hspace{1pt}\mathrm{prime}} \Q_\ell^\times\\
		&= \big\{(x_{\R},x_2,x_3,x_5,...) : x_\ell \in \Z_\ell^\times \text{ for all but finitely many }\ell\big\}.
	\end{align*}

(The prime on the product denotes \emph{restricted product}, which indicates the almost everywhere integral property in the definition.)  The ideles form a topological ring equipped with the restricted product topology, namely the topology with a basis of open neighbourhoods given by subsets of the form $U \times \prod_\ell U_\ell$, with $U \subseteq \R^\times$, $U_\ell \subseteq \Q_\ell^\times$ open subsets such that $U_\ell = \Z_\ell^\times$ for almost all primes $\ell$. The units $\Q^\times$ embed diagonally in $\A^\times$ (that is, via $x \mapsto (x,x,x,...)$) and we have:

\end{definition}

\begin{proposition}[Strong approximation] There is a topological isomorphism
	\[ \Q^\times \backslash \A^\times \cong \R_{>0} \times \prod_{\ell\hspace{1pt}\mathrm{prime}}\Z_\ell^\times.\]
	Hence all continuous characters
	\[\Q^\times \backslash \A^\times \longrightarrow \C^\times\]
	are of the form $\kappa_{\chi,s}$ as above, where $\chi$ is a Dirichlet character and $s \in \C$.
\end{proposition}

\begin{proof}
See \cite[Proposition I.4.6]{GoldfeldHundley}.
\end{proof}

\begin{remark}The space $\Q^\times\backslash\A^\times$ is the \emph{idele class group} of $\Q$, and features prominently in the idelic formulation of Class Field Theory.
\end{remark}

By the identification of $\C$ with $\Homc(\R_{>0},\C^\times)$ one can view $\zeta$ as a function
\begin{align*}\zeta : \Homc(\R_{>0},\C^\times) &\longrightarrow \C\\
	\big[x \mapsto x^s \big] &\longmapsto \zeta(s).
\end{align*}
But now we can consider \emph{all} complex Dirichlet $L$-functions \emph{at once} via the function
\begin{align} L : \Homc(\Q^\times\backslash\A^\times, \C^\times) &\longrightarrow \C \label{eq:dirichlet L-function}\\
	\kappa_{\chi,s} &\longmapsto L(\chi,s)\notag.
\end{align}
In the framework of Tate, this function $L$ can be viewed as integrating $\kappa_{\chi,s}$ against the \emph{Haar measure} on $\Q^\times\backslash\A^\times$. In Tate's thesis, he showed properties such as the analytic continuation and functional equations of Dirichlet $L$-functions using harmonic analysis on measures. Indeed, the idelic formulation gives a beautiful conceptual explanation for the appearance of the $\Gamma$-functions and powers of $2\pi i$ in the functional equation of the zeta function; these factors form the `Euler factor at the archimedean place'. The measure-theoretic perspective has proven to be a powerful method of defining and studying automorphic $L$-functions in wide generality.

\subsubsection{$p$-adic $L$-functions via measures}

To obtain a $p$-adic version of this picture, by analogy with \eqref{eq:dirichlet L-function}, a natural thing to do is to consider $\Homc(\Q^\times\backslash\A^\times,\Cp^\times)$ (that is, replacing $\C$ with $\Cp$).  Again, by strong approximation, an element of this space corresponds to a $\Cp$-valued character on $\R_{>0} \times \prod \Z_\ell^\times$. Since $\R_{>0}$ is connected and $\Cp$ is totally disconnected, the restriction of any such character to $\R_{>0}$ is trivial. Also using topological arguments we find that the restriction to $\prod_{\ell \neq p}\Z_{\ell}^\times$ factors through a finite quotient, so gives rise to some Dirichlet character of conductor prime to $p$. This leaves the restriction to $\Zp^\times$, i.e.\ $\Homc(\Zp^\times,\Cp^\times)$, which is by far the most interesting part.

In particular, in the spirit of \eqref{eq:naive interpolation}, we look for a `$p$-adic analytic' function
\[ 
\zeta_p : \Homc(\Zp^\times,\Cp^\times) \longrightarrow \Cp 
\]
which `sees' the special values of $\zeta(s)$ in the sense that that
\[\zeta_p(x \mapsto x^k) = (*) \cdot \zeta(1-k), \hspace{12pt} k \geq 1 \]
for an explicit factor $(*)$.

 In \S\ref{sec:measures}, we will develop the appropriate notion of `$p$-adic analytic' object in this setting: $p$-adic measures\footnote{It is not immediately obvious why we describe these as analytic, but in the background such objects can be described in terms of \emph{rigid analysis}, a $p$-adic analogue of complex analysis. Whilst we will not explicitly use rigid analysis, the connection is described precisely in Remark \ref{rem:weight space rigid analytic}.

\hspace{8pt} In this language, measures correspond to analytic functions, and pseudo-measures to meromorphic functions with at worst simple poles.}  (and pseudo-measures) on $\Zp^\times$. Then in \S\ref{kl} we will prove:

\begin{theorem} [Kubota-Leopoldt, Iwasawa]\label{thm:kubota-leopoldt}
	 There exists a unique pseudo-measure %\footnote{\emph{Pseudo-measures} will be defined in \S \ref{sec:pseudo-measures}. Roughly speaking, such an object is a measure that is allowed to have simple poles.} 
  $\zeta_p$ on $\Zp^\times$ such that, for all $k > 0$,
	\[\int_{\Zp^\times} x^k \cdot \zeta_p \defeq \zeta_p(x\mapsto x^k) = \big(1-p^{k-1}\big)\zeta(1-k).\]
\end{theorem}

\begin{remark}
	 Note that the factor $(1 - p^{k-1})$ is  the inverse of the factor at the prime $p$ of the product formula $\zeta(s) = \prod_{\ell} (1-\ell^{-s})^{-1}$, evaluated at $s = 1-k$. So, even though the Euler product does not converge at $s = 1 - k$, Theorem \ref{thm:kubota-leopoldt} morally says  that, removing the factor at $p$ from the Euler product formula, one can $p$-adically interpolate the Riemann zeta function. This is a general phenomenon appearing in the theory of $p$-adic $L$-functions.
\end{remark}

From the pseudo-measure $\zeta_p$, we can build $p-1$ (meromorphic) functions on $\Zp$, each satisfying a partial version of \eqref{eq:naive interpolation}. If we stick with the measure-theoretic approach, however, we have much more. The following result illustrates this power. In constructing $\zeta_p$, we use only values of $\zeta(s)$, without referring to Dirichlet characters at all. However, we also have:

\begin{theorem} \label{thm:introinterpdir}
	Let $\chi$ be a Dirichlet character of conductor $p^n$, $n\geq 0$, viewed as a locally constant character on $\Zp^\times$ \footnote{ That is, a character on $\Z_p^\times$ factoring through $\Z_p^\times / (1 + p^n \Z_p)$ for some large enough $n$.}. Then, for all $k > 0$,
	\[\int_{\Zp^\times} \chi(x) x^k \cdot \zeta_p = \big(1- \chi(p)p^{k-1}\big)L(\chi,1-k).\]
\end{theorem}

Thus $\zeta_p$ also interpolates all the negative integer values $L(\chi,-k)$ for \emph{all} Dirichlet $L$-functions of $p$-power conductor. This is very surprising, since a priori one constructs $\zeta_p$ using only information about the untwisted special $L$-values.

To complete the picture given in \S\ref{sec:dirichlet ideles}, one also considers Dirichlet characters of conductor prime to $p$. Similar ideas can also be used to show:

\begin{theorem} \label{thm:introdir}
	Let $D > 1$ be any integer coprime to $p$, and let $\eta$ denote a (primitive) Dirichlet character of conductor $D$. There exists a unique measure $\zeta_\eta$ on $\Zp^\times$ with the following interpolation property: for all primitive Dirichlet characters $\chi$ with conductor $p^n$ for some $n \geq 0$, we have, for all $k > 0$,
	\[\int_{\Zp^\times} \chi(x) x^k \cdot\zeta_{\eta} = \big(1 - \chi\eta(p)p^{k-1} \big) L(\chi\eta, 1-k). \]
\end{theorem}

\begin{remark} Let $(\Z/D\Z)^{\times\wedge}$ denote the space of characters on $(\Z/D\Z)^\times$. The measures given by Theorem \ref{thm:introdir} can be seen as functions on $\Homc(\zpe, \cp^\times) \times (\Z / D \Z)^{\times\wedge}$ and they are compatible with respect to the natural maps $(\Z / E \Z)^{\times\wedge} \to (\Z / D \Z)^{\times\wedge}$ for $E | D$. This shows that they define a function on
	\begin{align*}
		\Homc(\zpe, \cp^\times) \times \varinjlim_{(D, p) = 1} \big( \Z / D \Z\big)^{\times\wedge} &= \Homc(\zpe, \cp^\times) \times \big( \prod_{\ell \neq p} \Z_\ell^\times \big)^\wedge \\
		&= \Homc(\Q^\times \backslash \A^\times, \cp^\times).
	\end{align*}
	In other words, they give a measure on the idele class group of $\Q$.
\end{remark}

\begin{remark}\label{rem:kummer congruences}
Note that if $k \equiv \ell \newmod{p^{m-1}(p-1)}$, then $x^k \equiv x^\ell \newmod{p^m}$ for any $x \in \Zp^\times$. In particular,  for any Dirichlet character $\eta$ of conductor prime to $p$, these theorems tell us that the special values of $L$-functions satisfy congruences
\[
    (1-\eta(p)p^{k-1})L(\eta, 1-k) \equiv (1-\eta(p)p^{\ell-1})L(\eta,1-\ell) \newmod{p^m}.
\]
For the Riemann zeta function, these are the \emph{(generalised) Kummer congruences}, which played a role in his classification of irregular primes, and which provided significant motivation for Theorem \ref{thm:kubota-leopoldt}. This gives an alternative way of viewing $p$-adic $L$-functions: as $p$-adic analytic objects that package together systematic congruences between $L$-values.
\end{remark}

%\vspace{40pt}

%%======================================================================
%%======================================================================
%%======================================================================
%%======================================================================
\newpage
\begin{center}
	\scshape{{\LARGE Part I: The Kubota--Leopoldt $p$-adic $L$-function}}\\[20pt]
\end{center}
\addcontentsline{toc}{part}{Part I: The Kubota--Leopoldt $p$-adic $L$-function}

%{\small\emph{
In this part, we give constructions of the Kubota--Leopoldt $p$-adic $L$-function and the $p$-adic $L$-functions of Dirichlet characters. In Section \ref{sec:measures}, we introduce the necessary formalism of $p$-adic (pseudo-)measures and Iwasawa algebras, and -- via Mahler transforms -- their relationship with spaces of power series. This section sets up language and correspondences we will use throughout the rest of these notes. 

In Section \ref{kl}, we construct a pseudo-measure on $\Zp^\times$ that interpolates the values of the Riemann zeta function at negative integers. In Section \ref{interpolation}, we show moreover that this pseudo-measure interpolates the values $L(\chi,-k)$ for $k > 0$ and $\chi$ any Dirichlet character of $p$-power conductor. Further, if $\eta$ is a Dirichlet character of conductor prime to $p$, we construct a measure on $\Zp$ that interpolates the values $L(\chi\eta,-k)$ for the same range of $k$ and $\chi$. In Section \ref{mellin} we rephrase the construction in terms of analytic functions on $\Zp$ via the Mellin transform. In \S\ref{sec:s=1} and \S\ref{sec:residue} we describe the behaviour at $s=1$ of these analytic functions, a point outside the region of interpolation. Finally, in \S\ref{sec:eisenstein} we discuss how these results can be used to construct the $p$-adic family of Eisenstein series, a prototype for Hida and Coleman families.

%}}

%%======================================================================
%%======================================================================

\section{Measures and Iwasawa algebras}\label{sec:measures}
In this section, we formally develop the theory of $p$-adic analysis that we will be using in the sequel. Whilst some of the results may appear a little dry in isolation, fluency in the measure-theoretic language will greatly help us simplify later calculations that would otherwise be very technical.

We start in a general setting, letting $G$ be a profinite abelian group, and introducing $p$-adic measures on $G$. We then show that the space of $p$-adic measures is isomorphic to the Iwasawa algebra of $G$. Additionally, in the special case where $G = \Zp$, we give a third perspective, showing that the Iwasawa algebra is also isomorphic to a space of power series via the Mahler transform. After developing a measure-theoretical toolkit for later use, we introduce pseudo-measures. We then conclude by discussing generalisations, including locally analytic distributions and rigid analytic functions.

Throughout, we fix a finite extension $L$ of $\qp$, with the $p$-adic valuation normalised so that $v_p(p) = 1$; this will be the coefficient field. We write $\mathscr{O}_L$ for its ring of integers.

\subsection{Preliminaries on $p$-adic Banach spaces}
We first collect some technical general definitions to anchor our discussions. This is intended only to make precise some of the notions we subsequently use, and the reader comfortable with $p$-adic Banach spaces and orthonormal bases may skip to \S\ref{sec:p-adic measures 1}. For more details, see \cite[\S I.1]{Colm10}.

\begin{definition}\label{def:valuation}
Let $B$ be an $L$-vector space. A \emph{valuation} on $B$ is a function $v : B \to \R\cup\{+\infty\}$ such that:
\begin{itemize}
\item[(i)] $v(x) = +\infty$ if and only if $x = 0$; 
\item[(ii)] $v(x+y) \geq \mathrm{min}(v(x), \ v(y))$ for all $x,y \in B$; 
\item[(iii)]and $v(\lambda x) = v_p(\lambda) + v(x)$ for all $\lambda \in L, x \in B$.
\end{itemize}
Such a valuation induces a norm (hence a topology) on $B$.
\end{definition}

\begin{definition}\label{def:Banach}
An \emph{$L$-Banach space} is a complete topological $L$-vector space $B$ whose topology is induced from a valuation $v$.
\end{definition}

\begin{definition}\label{def:orthonormal}
\begin{enumerate}
\item	Let $I$ be a set, and $\ell^0_\infty(I,L)$ the set of sequences $(a_i)_{i\in I}$ in $L$ that tend to 0 in the sense that for all $\epsilon > 0$, we have $|a_i|_L < \epsilon$ for all but finitely many $i$. This is naturally an $L$-Banach space with valuation $v((a_i)_{i}) = \mathrm{inf}_{i \in I} v_p(a_i).$
			
\item If $B$ is an $L$-Banach space, an \emph{orthonormal basis} of $B$ is a collection $(e_i)_{i \in I}$, for some set $I$, such that we have an isometry
\begin{align*}
\ell^0_\infty(I, L) &\longrightarrow B,\\
(a_i)_{i \in I} &\longmapsto \sum_{i \in I} a_i e_i,
\end{align*}
 \end{enumerate}
 \end{definition}

\begin{remark}
By \cite[Prop.\ I.1.5]{Colm10}, if $B$ is an $L$-Banach space with valuation $v_B$, and $v_B(B) = v_p(L)$, then $B$ admits an orthonormal basis.
\end{remark}

We shall also be concerned with dual spaces. If $B$ is a topological $L$-vector space, let
\[
B^* \defeq \mathrm{Hom}_{\mathrm{cts}}(B, L)
\]
be its continuous linear dual. If $B$ is an $L$-Banach space, there are two natural topologies on $B^*$.

\begin{definition}\label{def:strong topology} Let $B$ be an $L$-Banach space and $B^*$ its continuous dual.
\begin{itemize}
\item The \emph{strong topology} is the topology induced by the natural dual valuation $v^*$, where
\[
v^*(\mu) \defeq \mathrm{inf}_{x \in B} \left(v_p(\mu(x)) - v(x)\right).
\]   
This is the coarsest topology such that a sequence $(\mu_j)_j \subset B^*$ converges if and only if it converges uniformly (in the usual sense of continuous functions on $B$).

\item The \emph{weak topology} is induced by the family of semivaluations\footnote{That is, functions $v$ that satisfy (ii) and (iii) of Definition \ref{def:valuation}, but not necessarily (i).} $\{v_x: x \in B\}$, where
\[
v_x(\mu) \defeq v_p(\mu(x)).
\]
This is the topology of pointwise convergence, the coarsest such that a sequence $(\mu_j)_j \subset B^*$ converges if and only if $\mu_j(x)$ converges for all $x \in B$.
\end{itemize}
\end{definition}

\begin{remark}\label{rem:strong topology}
The dual $B^*$ is complete with both of these topologies. However generally it is an $L$-Banach space only for the strong topology, whilst $B$ is reflexive (canonically isomorphic to its double continuous dual) only when $B^*$ is equipped with the weak topology.
\end{remark}

\subsection{$p$-adic measures}\label{sec:p-adic measures 1}
We now return to our specific setting, and introduce the $p$-adic measures fundamental to our story. Let $G$ be a profinite abelian group; the examples $G = \Zp$ or $G = \Zp^\times$ are of most interest to us.

\begin{definition}\label{def:cts functions}
We denote by $\mathscr{C}(G, L)$ the space of continuous functions $\phi : G \to L$. We equip this space with a valuation 
\[
v_{\mathscr{C}}(\phi) = \inf_{x \in G} v_p(\phi(x)), \qquad \phi \in \mathscr{C}(G,L),
\]
noting this is well-defined as $G$ is compact (hence $\phi$ is bounded).
\end{definition}
This valuation induces the sup norm on $\mathscr{C}(G,L)$, and endows it with the structure of an $L$-Banach space, in the sense of Definition \ref{def:Banach}.

%a  bounded linear functional $\mu : \mathscr{C}(G, L) \to L$ \footnote{Recall that a linear functional $\mu$ is bounded if there exists a constant $C$ such that $v_p(\mu(\phi)) \geq v_{\mathscr{C}}(\phi) + C$ for all $\phi \in \mathscr{C}(G, L)$, and it is equivalent to asking $\mu$ to be continuous. }.
\begin{definition}\label{def:measures}
We define the space $\mathscr{M}(G, L)$ of \emph{$L$-valued measures on $G$} as the continuous linear dual $\mathscr{C}(G,L)^* = \mathrm{Hom}_{\rm cts}(\mathscr{C}(G, L), L)$.  If $\phi \in \mathscr{C}(G, L)$ and $\mu \in \mathscr{M}(G, L)$, the evaluation of $\mu$ at $\phi$ will be denoted by
\[ \int_G \phi(x) \cdot \mu(x), \]
or by $\int_G \phi \cdot \mu$ if the variable of integration is clear from the context.% (in the literature, this is sometimes written alternatively as $\int_G \phi \cdot d\mu$). 

We say that an element $\mu \in \mathscr{M}(G, L)$ is an \emph{$\roi_L$-valued measure}, and write $\mu \in \mathscr{M}(G, \roi_L)$, if $\int_G \phi \cdot \mu \in \roi_L$ for every $\roi_L$-valued function $\phi$. Since measures are continuous (or equivalently, bounded), we have $\mathscr{M}(G, L) = \mathscr{M}(G, \roi_L) \otimes_{\roi_L} L$. We will be mainly concerned with $\roi_L$-valued functions and measures.
\end{definition}

\begin{remark}
All parts of Definitions \ref{def:cts functions} and \ref{def:measures} apply identically if we replace $G$ with any subset $X \subset G$ equipped with the subspace topology (noting that $X$ no longer need be a group).
\end{remark}

The following simple example will be crucial in later chapters.
\begin{example}\label{ex:dirac}
For any $g \in G$, the \emph{Dirac measure} $\delta_g \in \mathscr{M}(G,\roi_L)$ is the linear functional `evaluation at $g$', that is, the measure defined by
\begin{align*} \delta_g : \mathscr{C}(G,\roi_L) &\longrightarrow \roi_L\\
\phi &\longmapsto \phi(g).
\end{align*}
\end{example}

We will give a number of alternative descriptions of $p$-adic measures. Firstly, we make the following simplifications.
\begin{remark}\label{rem:locally constant}
Let $\sC^{\mathrm{lc}}(G,\roi_L)$ denote the space of locally constant functions $G \to \roi_L$; this is a dense subspace of the continuous functions $\sC(G,\roi_L)$. Indeed, any continuous function $\phi \in \mathscr{C}(G, \roi_L)$ can be $p$-adically approximated by its locally constant truncations $\phi_n(x) = \sum_{a \in (\Z / p^n \Z)} \phi(a) \mathbf{1}_{a + p^n \zp}(x)$, where $\mathbf{1}_{a + p^n \zp}(x)$ denotes the characteristic function of $a + p^n \zp$. Let 
\[
    \sM^{\mathrm{lc}}(G,\roi_L) \defeq \sC^{\mathrm{lc}}(G,\roi_L)^*
\]
be the continuous dual, the space of `locally constant measures on $G$'. We claim restriction from $\sC$ to $\sC^{\mathrm{lc}}$ defines a canonical isomorphism
\begin{equation}\label{eq:restrict measures}
    \sM(G,\roi_L) \isorightarrow \sM^{\mathrm{lc}}(G,\roi_L).
\end{equation}
To see this, we write down an inverse. Suppose $\mu^{\mathrm{lc}} \in \sM^{\mathrm{lc}}(G,\roi_L)$, and let $\phi \in \sC(G,\roi_L)$. Using density, choose a sequence $\phi_n \in \sC^{\mathrm{lc}}(G,\roi_L)$ with $\phi_n \to \phi$ and define
\[
    \int_G \phi \cdot \mu \defeq \lim_{n\to \infty} \int_G\phi_n\cdot \mu^{\mathrm{lc}}.
\]
By continuity, the limit is well-defined and independent of the choice of $\phi_n$, and hence we obtain a well defined measure $\mu \in \sM(G,\roi_L)$. This gives a map
\[
    \sM^{\mathrm{lc}}(G,\roi_L) \longrightarrow \sM(G,\roi_L)
\]
visibly inverse to \eqref{eq:restrict measures}.

Henceforth we drop the notation $\mu^{\mathrm{lc}}$ and just write $\mu$.
\end{remark}

\begin{remark}\label{rem:additive functions on compact sets}
We have an identification of $\sM^{\mathrm{lc}}(G,\roi_L)$ with the space of additive functions 
\begin{equation}\label{eq:measures as additive functions}
	 \mu : \{\text{open compact subsets of $G$}\} \longrightarrow \mathscr{O}_L. 
	 \end{equation}
Indeed, if $\mu \in \mathscr{M}^{\mathrm{lc}}(G, \mathscr{O}_L)$ and $U \subset G$ is an open compact set, one defines $\mu(U) \defeq \int_G \mathbf{1}_U(x) \cdot \mu(x)$, where $\mathbf{1}_U(x)$ denotes the characteristic function of $U$.

Conversely, let $\mu$ be such a function and let $\phi \in \mathscr{C}^{\mathrm{lc}}(G, \mathscr{O}_L)$. We will see how to integrate $\phi$ against $\mu$. As $\phi$ is locally constant, there is some open subgroup $H$ of $G$ such that $\phi$ can be viewed as a function on $G/H$. We define the integral of $\phi$ against $\mu$ to be 
\[\int_{G}\phi \cdot \mu \defeq \sum_{[a] \in G/H}\phi(a)\mu(aH).\]
\end{remark}

%We can think of measures as additive functions
%\begin{equation}\label{eq:measures as additive functions}
	 %\mu : \{\text{compact open subsets of $G$}\} \longrightarrow \mathscr{O}_L. 
	% \end{equation}
%Indeed, if $\mu \in \mathscr{M}(G, \mathscr{O}_L)$ and $U \subset G$ is an open compact set, one defines $\mu(U) \defeq \int_G \mathbf{1}_U(x) \cdot \mu(x)$, where $\mathbf{1}_U(x)$ denotes the characteristic function of $U$.

%Conversely, let $\mu$ be such a function and let $\phi \in \mathscr{C}(G, \mathscr{O}_L)$. We will see how to integrate $\phi$ against $\mu$. Assume first $\phi$ is locally constant; then there is some open subgroup $H$ of $G$ such that $\phi$ can be viewed as a function on $G/H$. We define the integral of $\phi$ against $\mu$ to be 
%\[\int_{G}\phi \cdot \mu \defeq \sum_{[a] \in G/H}\phi(a)\mu(aH).\]
%In general, as the locally constant functions are dense in the continuous functions, we can write $\phi = \lim_{n\rightarrow\infty} \phi_n$, where each $\phi_n$ is locally constant. Then we can define
%\[\int_{G}\phi \cdot \mu \defeq \lim_{n\rightarrow\infty}\int_G \phi_n \cdot\mu,\]
%which exists and is independent of the choice of $\phi_n$. This defines an element in $\mathscr{M}(G, \mathscr{O}_L)$. 

Combining Remarks \ref{rem:locally constant} and \ref{rem:additive functions on compact sets}, we have an identification of $\sM(G,\roi_L)$ with the space of additive functions on the open compact subsets of $G$.

\begin{remark}
On $\Zp$, we have a (real-valued) Haar measure defined so that open compact subsets of the form $a+p^n\Zp$ have measure $p^{-n}$. Whilst this is probably the most natural measure one might consider on $\Zp$, observe that this is \emph{not} a $p$-adic measure, as it is not $p$-adically bounded!
\end{remark}

\begin{example}
For $g \in G$, the Dirac measure $\delta_g$ from Example \ref{ex:dirac} corresponds to the function $\widetilde{\delta}_g$ on open compact subsets given by
\[
\widetilde{\delta}_g(X) = \left\{\begin{array}{ll} 1 &\text{if } g \in X\\
0 &\text{if } g \notin X,\end{array}\right.
\]
as can be seen directly from the identification above.
\end{example}

\subsection{The Iwasawa algebra} 
We will now express measures in terms of algebra. As a prototype, we recall a useful fact from representation theory. If $G$ is a finite abelian group, let $\sC(G,\Z)$ be the space of functions $G \to \Z$, and $\mathscr{M}(G,\Z)$ its dual, the space of `continuous measures' on $G$ (when we equip $G$ with the discrete topology). For any $g \in G$ we have the Dirac measure $\delta_g \in \mathscr{M}(G,\Z)$ given by $\delta_g(\phi) \defeq \phi(g)$, as in Example \ref{ex:dirac}.  Then recall the following classical result.

\begin{proposition} \label{PropIwGroupAlgebrafinite}
If $G$ is a finite abelian group, the map $[g] \mapsto \delta_g$ induces an isomorphism between the group algebra $\Z[G]$ and $\mathscr{M}(G,\Z)$. 
\end{proposition}

When $G$ is profinite abelian, we have an analogous $p$-adic result after replacing the group  algebra with its profinite completion, the Iwasawa algebra.

\begin{proposition}\label{prop:iwasawa algebra measures}
We have an natural isomorphism
\[\mathscr{M}(G, \mathscr{O}_L) \cong \varprojlim_H \mathscr{O}_L[G/H],\]
where the limit is over all open subgroups of $G$. 
\end{proposition}

\begin{proof}
%Recall that we have seen after Example \ref{ex:dirac} that measures correspond to additive functions
%\[
%\mu : \{\text{compact open subsets of $G$}\} \longrightarrow \mathscr{O}_L.\]
%We have also observed that these functions are precisely the dual of the space $\varinjlim_{H} \mathscr{C}(G / H, \roi_L)$ of locally constant functions of $G$, i.e. of functions $g : G \to \roi_L$ for which there exist some open compact subgroup $H$ of $G$ such that $f$ factors through $G/H$. So that 
By Remark \ref{rem:locally constant}, we have a canonical isomorphism
\[
    \sM(G,\roi_L) \cong \sM^{\mathrm{lc}}(G,\roi_L) = \mathrm{Hom}_{\mathrm{cts}}(\sC^{\mathrm{lc}}(G,\roi_L),\roi_L). 
\]
As any locally constant function factors through $G/H$ for some open compact subgroup $H \leq G$, we also have a natural isomorphism
\[
    \sC^{\mathrm{lc}}(G,\roi_L) \cong \varinjlim_H \mathscr{C}(G / H, \roi_L).
\]
We thus have
\begin{equation}\label{eq:measures as inverse limit}
\mathscr{M}(G, \roi_L) \cong \mathrm{Hom}_{\rm cts}(\varinjlim_H \mathscr{C}(G / H, \roi_L), \roi_L) \cong \varprojlim_H \mathscr{M}(G/ H, \roi_L),
\end{equation}
the final isomorphism following from compatibility of duals with limits. 
As $G/H$ is a finite group, by Proposition \ref{PropIwGroupAlgebrafinite} we have that $\mathscr{M}(G/ H, \roi_L) \cong \roi_L[G/H]$, giving the result.
\end{proof}

We explicitly describe both maps in this isomorphism.

Let $\mu$ be a measure, considered as an additive function as in \eqref{eq:measures as additive functions}, and let $H$ be an open subgroup of $G$. We define an element $\lambda_H$ of $\roi_L[G/H]$ by setting 
\[
    \lambda_H \defeq \sum_{[a] \in G/H}\mu(aH)[a].
\]
By the additivity property of $\mu$, we see that $(\lambda_H)_H \in \varprojlim \mathscr{O}_L[G/H]$, and this gives the map from measures to the inverse limit.

Conversely, given such an element $\lambda$ of the inverse limit, write $\lambda_H$ for its image in $\roi_L[G/H]$ under the natural projection. Then 
\[\lambda_H = \sum_{[a] \in G/H} c_a[a].\]
We define
\[\mu(aH) = c_a.\]
Since the $\lambda_H$ are compatible under projection maps, this defines an additive function on the open compact subsets of $G$, i.e.\ an element $\mu \in \mathscr{M}(G, \mathscr{O}_L)$.

\begin{definition}
We define the \emph{Iwasawa algebra of $G$}  to be the profinite completion of the group algebra $\roi_L[G]$, i.e.,
\[\Lambda(G) \defeq \lim_{\substack{\longleftarrow \\ H}} \mathscr{O}_L[G/H].\]
(Note that we suppress $L$ from the notation.)
\end{definition}

\begin{remark} \label{RemarkConvolution}
The Iwasawa algebra $\Lambda(\Zp)$ has a natural $\roi_L$-algebra structure, and hence by transport of structure we obtain such a structure on $\mathscr{M}(\Zp,\roi_L).$ As with the classical situation for finite group rings, the algebra structure on the space of measures can be described directly via \emph{convolution of measures}. For a general profinite abelian group $G$, given two measures $\mu,\lambda \in \mathscr{M}(G,\roi_L)$, one defines their convolution $\mu * \lambda$ to be
\[\int_G \phi \cdot(\mu*\lambda) = \int_G\left(\int_G \phi(x + y) \cdot \lambda(y)\right) \cdot \mu(x).\]
One checks that this does give an algebra structure and that the isomorphism above is an isomorphism of $\roi_L$-algebras.
\end{remark}

\begin{example}
Let $a \in \Zp$, and let $\delta_a$ be the Dirac measure on $\Zp$ from Example \ref{ex:dirac}. Recall this corresponds to the function $\widetilde{\delta}_a$ on open compact subsets given by $\widetilde{\delta}_a(a) = 1$ (if $a \in X$) and $\widetilde{\delta}_a(a) = 0$ (if $a \notin X$).  Under the isomorphism of Proposition \ref{prop:iwasawa algebra measures}, $\delta_a$ corresponds to the projective system 
\[
([a + p^n\Zp])_{n \in \N} \in \varprojlim_{n \in \N} \roi_L[\Zp/p^n\Zp].
\]
In the inverse limit this yields an element of the Iwasawa algebra that we denote  $[a]$.
\end{example}

%%======================================================
\subsection{$p$-adic analysis and Mahler transforms}\label{sec:mahler}
So far we have given three equivalent descriptions of $p$-adic measures on a profinite abelian group $G$:
\begin{enumerate}
\item as linear functionals on $\mathscr{C}(G,L)$,
\item as additive functions on open compact subsets of $G$, and
\item as elements of the Iwasawa algebra of $G$.
\end{enumerate}

In this section we specialise to $G = \Zp$, and in this case give yet another equivalent description, via the power series ring $\roi_L\lsem T\rsem$.

\begin{definition}\label{def:binomial polynomials}
For $x \in \zp$, let
\[\binomc{x}{n} \defeq \frac{x(x-1)\cdots(x-n+1)}{n!} \text{ for }n \geq 1, \qquad \text{and } \binomc{x}{0} = 1.\]
\end{definition}

One easily checks that $x \mapsto {x \choose n}$ defines an element in $\mathscr{C}(\zp, \zp)$ of valuation $v_{\mathscr{C}}\big( {x \choose n} \big) = 0$. The following theorem is fundamental in all that follows. It says that the functions ${x \choose n}$ form an orthonormal basis for the $L$-Banach space $\mathscr{C}(\zp, L)$ (in the sense of Definition \ref{def:orthonormal}).

\begin{theorem}[Mahler]\label{thm:Mahler}
Let $\phi : \Zp \rightarrow L$ be a continuous function. There exists a unique expansion
\[ \phi(x) = \sum_{n\geq 0} a_n(\phi) \binomc{x}{n}, \]
where $a_n(\phi) \in L$ and $a_n(\phi) \rightarrow 0$ as $n\rightarrow \infty$. Moreover, $v_{\mathscr{C}}(\phi) = \inf_{n \in \N} v_p(a_n(\phi))$.
\end{theorem}

\begin{proof}
See \cite[Théorème 1.2.3.]{Colm10}.
\end{proof}

\begin{remark}
The coefficients $a_n(\phi)$ are called the \emph{Mahler coefficients} of $\phi$. 
%Moreover, the association $\phi \mapsto (a_n(\phi)))_{n \in \N}$ defines an isometry between $\mathscr{C}(\zp, L)$ and the space $\ell^0_{\infty}(\N, L)$ of infinite sequences in $L$ that tend to 0, i.e.\ 
%\[ v_{\mathscr{C}}(\phi) = \inf_{n \in \N} v_p\big(a_n(\phi)\big). \]
One can write down the Mahler coefficients of $\phi$ very simply: we define the \emph{discrete derivatives} of $\phi$ by 
\[\phi^{[0]} = \phi,\hspace{12pt} \phi^{[k+1]}(x) = \phi^{[k]}(x+1) - \phi^{[k]}(x),\]
and then $a_n(\phi) = \phi^{[n]}(0).$
\end{remark}

 It is natural to study a measure by looking at its values on the elements of the (orthonormal) Mahler basis. We encode these values in the following power series.

\begin{definition}
Let  $\mu \in \mathscr{M}(\Zp, \roi_L)$ be a $p$-adic measure on $\Zp$. Define the \emph{Mahler transform} (or \emph{Amice transform}) of $\mu$ to be
\[\Am_\mu(T) \defeq \int_{\Zp} (1+T)^x \cdot \mu (x) = \sum_{n\geq 0} \left(\int_{\Zp} \binomc{x}{n} \cdot\mu \right) T^n \in \roi_L\lsem T\rsem .\]
\end{definition}

\begin{example}
Let $a \in \Zp$, and recall the Dirac measure $\delta_a$. By definition, its Mahler transform is
\[ 
\Am_{\delta_a}(T) = \sum_{n\geq 0}\binomc{a}{n} T^n = (1+T)^a. 
\]
\end{example}

Before stating the main theorem concerning the Mahler transform, let us consider how it interacts with the isomorphism $\mathscr{M}(\Zp,\roi_L) \isorightarrow \Lambda(\Zp)$ of Proposition \ref{prop:iwasawa algebra measures}. As $1$ is a topological generator of (the additive group) $\Zp$, and likewise $1+T$ is a topological generator of $\roi_L\lsem T\rsem$, one may check that $[1] \mapsto (1 + T)$ induces an isomorphism $\Lambda(\Zp) \isorightarrow \roi_L\lsem T\rsem$, fitting into a commutative diagram
\begin{equation} \label{EqMahler}
\begin{tikzcd}
\mathscr{M}(\zp, \roi_L) \ar[r] \ar[d] & \roi_L\lsem T \rsem \\
\Lambda(\Zp) \ar[ur] &
\end{tikzcd}
\end{equation}
Indeed, by continuity it suffices to check on Dirac measures. As $\delta_a \mapsto (1+T)^a$ under the top arrow, and $\delta_a \mapsto [a] \mapsto (1+T)^a$ under the bottom arrows, we are done.

Given the bottom two arrows in \eqref{EqMahler} are isomorphisms, the following theorem should not be surprising.

\begin{theorem}\label{thm:mahler}
The Mahler transform gives an $\roi_L$-algebra isomorphism
\[ \mathscr{M}(\Zp, \roi_L) \isorightarrow \roi_L\lsem T\rsem . \]
\end{theorem}

\begin{proof}
This is almost a tautology from the definition of orthonormal basis. By continuity and linearity, any measure $\mu \in \sM(\Zp,\roi_L)$ is uniquely determined by the values $\int_{\Zp} {x \choose n}\cdot \mu$. Indeed, let $\phi \in \sC(\Zp,\roi_L)$. By Mahler's theorem, we can write $\phi(x) = \sum_{n \geq 0} a_n(\phi) {x \choose n}$ for some unique $a_n(\phi) \in \roi_L$ such that $a_n(\phi) \to 0$ as $n \to \infty$; and then we have
\[
\int_{\Zp} \phi \cdot \mu = \sum_{n \geq 0} a_n(\phi)\int_{\Zp} {x \choose n} \cdot \mu.
\]

Conversely, given any collection of values $c_n \in \roi_L$, defining an element $g = \sum_{n\geq 0}c_nT^n \in \roi_L\lsem T\rsem$, there is a unique measure $\mu_g$ with $\int_{\Zp} {x \choose n} \cdot \mu_g = c_n$. Concretely, for any $\phi = \sum_{n\geq 0} a_n(\phi){x\choose n}\in \sC(\Zp,\roi_L)$ as above, we define
\[
\int_{\Zp} \phi \cdot \mu_g = \sum_{n \geq 0} a_n(\phi) c_n,
\]
which converges to an element in $\roi_L$. Visibly we have $\sA_{\mu_g} = g$, so this defines an inverse to the Mahler transform. 
\end{proof}

\begin{remark}
Each space in the diagram \eqref{EqMahler} has a description as an inverse limit. For measures, this is \eqref{eq:measures as inverse limit}; for the Iwasawa algebra, this is by definition; and for power series, we have 
\[
    \roi_L\lsem T\rsem \cong \varprojlim_n \roi_L[T]\big/\big((1+T)^{p^n} - 1\big).
\]
The appearance of the expression $(1+T)^{p^n} - 1$ will become clearer in \S\ref{SubSectionphipsi} below (and its importance further recognised in Appendix \ref{sec:mu invariant}).

We invite the reader to spell out maps between the  `level $n$' terms of the inverse limits,  analogous to \eqref{EqMahler} and such that the following diagram commutes:
\[
	\begin{tikzcd}
		\mathscr{M}(\z/p^n\Z, \roi_L) \ar[r] \ar[d] & \roi_L[T]\big/\big((1+T)^{p^n} - 1\big)\\
		\roi_L[\Z/p^n\Z] \ar[ur] &
	\end{tikzcd}
\]

\end{remark}

\begin{definition}
If $g \in \roi_L\lsem T\rsem $, we continue to write $\mu_g \in \mathscr{M}(\Zp, \roi_L)$ for the corresponding ($\roi_L$-valued) measure on $\Zp$ (so that $\Am_{\mu_g} = g)$.
\end{definition}

\begin{remark} \label{integ} %\leavevmode
\begin{enumerate}
    \item Mahler's isomorphism induces an isomorphism 
    \[
    	\mathscr{M}(\Zp, L) \cong \sM(\Zp,\roi_L) \otimes_{\roi_L} L \isorightarrow \roi\lsem T\rsem \otimes_{\roi_L}L \cong  \roi\lsem T\rsem[1/p].
    \]
    \item Let $g \in \roi_L\lsem T\rsem $ with associated measure $\mu_g$. From the definitions, it is easily seen that
\[\int_{\Zp}\mu_g = g(0), \;\;\; \int_{\Zp} x \cdot \mu_g = g'(0), \;\;\; \int_{\Zp} x^2 \cdot \mu_g = g''(0) + g'(0), \]
\[
\int_{\Zp}x^3 \cdot \mu_g = g'''(0) + 3g''(0) + g'(0), \ \ \ \ldots, 
\]
that is, for every $n$, the value $\int_{\Zp} x^n \cdot \mu_g$ can be written as an integer combination of $g^{(r)}(0)$ for $0 \leq r\leq n$. We simplify this in Corollary \ref{cor:eval at x^k} below.

\item Recall (from Definition \ref{def:strong topology}) that there are two natural topologies on $\sM(\Zp,\roi_L)$. One can check that under the isomorphism of Theorem \ref{thm:mahler}, the strong topology corresponds to the $p$-adic topology on $\roi_L\lsem T\rsem$, whilst the weak topology corresponds to the $(p,T)$-adic topology. Analogously to Remark \ref{rem:strong topology}, the $p$-adic topology on $\roi_L\lsem T\rsem$ is that of uniform convergence in the power series coefficients, whilst the $(p,T)$-adic topology is that of pointwise (term-by-term) convergence. For example, consider the sequence $1, T, T^2, T^3, \dots$ in $\roi_L\lsem T\rsem$. This converges to 0 pointwise, and in the $(p,T)$-adic (weak) topology; but it does not converge uniformly, or in the $p$-adic (strong) topology.
\end{enumerate}
\end{remark}

%%======================================================
\subsection{A measure-theoretic toolbox}\label{sec:toolbox}
There are natural operations one might consider on measures, and via the Mahler transform these give rise to operators on power series. The following operations can be considered as a `toolbox' for working with measures and power series; as we shall see in the sequel, the ability to manipulate measures in this way has important consequences. For further details (and more operations), see \cite{Colm10}.

\begin{notation}
	In light of the isomorphism between the space  $\sM(G,\roi_L)$ of measures and the Iwasawa algebra $\Lambda(G)$, we will frequently conflate the two. In particular, it is convenient -- and indeed standard -- to write $\mu \in \Lambda(G)$ for measures on $G$, typically suppressing the coefficient field $L$, which is fixed throughout.
\end{notation}

\subsubsection{Multiplication by $x$}\label{sec:mult by x}
Given a measure $\mu$ on $\Zp$, we naturally wish to compute $\int_{\Zp}x^k \cdot \mu$ for $k$ a positive integer. To allow us to do that, we let $x\mu$ be the measure defined by
\[
\int_{\Zp} f(x) \cdot x\mu = \int_{\Zp} xf(x) \cdot\mu.
\]
We can ask what the operation $\mu \mapsto x\mu$ does on Mahler transforms; we find:

\begin{lemma} \label{LemmaMultiplicationbyx}
We have
\[\Am_{x\mu} = \partial\Am_\mu,\]
where $\partial$ denotes the differential operator $(1+T)\frac{d}{dT}$.
\end{lemma}

\begin{proof}
The result follows directly from computing
\[  x\binomc{x}{n} = (x-n)\binomc{x}{n} + n\binomc{x}{n} = (n+1)\binomc{x}{n+1} + n\binomc{x}{n}. \qedhere\]
\end{proof}

From the above lemma, we immediately obtain the following expressions, completing the examples seen in Remark \ref{integ}.

\begin{corollary}\label{cor:eval at x^k}
For $\mu \in \Lambda(\Zp)$, we have
\[\int_{\Zp}x^k \cdot \mu = \big(\partial^k \Am_{\mu}\big)(0).\]
\end{corollary}

\subsubsection{Multiplication by $z^x$}

Totally analogously to above, if $g \in \sC(\Zp,L)$ and $\mu$ is a measure on $\Zp$, then we can define a measure $g(x)\mu$ by 
\[
\int_{\Zp} f(x) \cdot g(x)\mu \defeq \int_{\Zp} f(x)g(x) \cdot \mu.
\]

Of particular interest is the measure $z^x\mu$, for $z \in \mathscr{O}_L$ such that $|z-1| < 1$. We claim the Mahler transform of $z^x \mu$ is
\[\Am_{z^x \mu}(T) = \Am_{\mu}\big((1+T)z - 1\big).\]
Indeed, from the definition of the Mahler transform, we see that
\[\Am_{\mu}\big((1+T)z - 1\big) = \int_{\Zp}\big((1+T)z\big)^x \cdot\mu,\]
and this is precisely the Mahler transform of $z^x \mu$ (one has to be slightly careful about convergence issues). 
%(More rigorously, one has an equality of this form for a large enough range of $T$ to force equality everywhere by Weierstrass preparation).

\subsubsection{Restriction to open compact subsets}

Consider an open compact subset $X \subset \Zp$, and write $\mathbf{1}_{X}$ for the characteristic function of $X$. The `restriction of $\mu$ to $X$' is the measure $\mathrm{Res}_{X}(\mu)$ on $\Zp$ defined by
\[
\int_{\Zp} f \cdot\mathrm{Res}_{X}(\mu) \defeq \int_{\Zp}f\mathbf{1}_{X} \cdot\mu.
\]
It is also standard (and intuitive) to denote this quantity as
\[
\int_X f\cdot \mu.
\]
We say a measure $\mu$ is \emph{supported on $X$} if $\mu = \mathrm{Res}_X(\mu)$.

\begin{remark}
Note that we may view this restriction as a measure on $X$ as follows. For any continuous function $g : X \to L$, let $\widetilde{g} : \Zp \to L$ denote its extension by 0 outside $X$. The map 
\[
\sC(X,L) \longrightarrow L, \qquad g \longmapsto \int_{\Zp}\widetilde{g} \cdot\mu
\]
defines a measure on $X$. Abusing notation, we also denote this measure by $\mathrm{Res}_X(\mu)$, noting it is intuitively compatible with its cousin defined on $\Zp$. Indeed we will often blur the distinction between considering $\mathrm{Res}_{X}(\mu)$ as a measure on $\Zp$ or on $X$.
\end{remark}

In the case where $X = b +p^n\Zp$, we can write the characteristic function explicitly as
\[\mathbf{1}_{b+p^n\Zp}(x) = \frac{1}{p^n}\sum_{\xi \in \mu_{p^n}}\xi^{x-b},\]
and then using the above, we calculate the Mahler transform of $\mathrm{Res}_{b+p^n\Zp}(\mu)$ to be
\begin{equation} \label{EqRestrictionFormula}
\Am_{\mathrm{Res}_{b + p^n\Zp}(\mu)}(T) = \frac{1}{p^n}\sum_{\xi \in \mu_{p^n}}\xi^{-b}\Am_{\mu}\big((1+T)\xi - 1\big).
\end{equation}

\subsubsection{Restriction to $\Zp^\times$}\label{sec:res to zpe}
From the above applied to $b=0$ and $n=1$, it is immediate that 
\begin{equation}\label{mahler transform restriction}\Am_{\mathrm{Res}_{\Zp^\times}(\mu)}(T) = \Am_\mu(T) - \frac{1}{p}\sum_{\xi \in \mu_{p}}\Am_{\mu}\big((1+T)\xi - 1\big).\end{equation}

In order to calculate a formula for the restriction to an arbitrary open compact subset $X \subseteq \Zp$, we can write $X$ (or its complement, as we did with $\Zp^\times$) as a disjoint union of sets of the form $b + p^n \Zp$ and apply the formulas obtained before.

\subsubsection{The action of $\Zp^\times$, $\varphi$ and $\psi$} \label{SubSectionphipsi}

We introduce an action of $\Zp^\times$ that serves as a precursor to a Galois action later on. Let $a \in \Zp^\times$. We can define a measure $\sigma_a(\mu)$ by
\[\int_{\Zp}f(x)\cdot\sigma_a(\mu) = \int_{\Zp}f(ax)\cdot\mu.\]
This has Mahler transform
\[\Am_{\sigma_a(\mu)} = \Am_\mu\big((1+T)^a - 1\big).\] 
\vspace{0.1pt}

In a similar manner, we can define an operator $\varphi$ acting as `$\sigma_p$' by
\[\int_{\Zp}f(x)\cdot\varphi(\mu) = \int_{\Zp}f(px)\cdot\mu,\]
and this corresponds to 
\begin{equation}\label{eq:varphi power series}
\Am_{\varphi(\mu)} = \varphi(\Am_\mu) \defeq \Am_\mu\big((1+T)^p - 1\big).
\end{equation}
Finally, we also define the analogous operator for $p^{-1}$; we define a measure $\psi(\mu)$ on $\Zp$ by defining
\[\int_{\Zp}f(x)\cdot\psi(\mu) = \int_{p\Zp}f(p^{-1}x)\cdot\mu.\]
Note that $\psi\circ\varphi = \mathrm{id}$, whilst $\varphi\circ\psi(\mu) = \mathrm{Res}_{p\Zp}(\mu)$. Indeed, we have
\[ \int_{\Zp} f(x) \cdot \psi\circ\varphi(\mu) = \int_{\Zp} \mathbf{1}_{p \zp}(x) f(p^{-1}x) \cdot \varphi(\mu) = \int_{\Zp} \mathbf{1}_{p \zp}(px) f(x) \cdot \varphi(\mu) = \int_{\Zp} f(x) \cdot \mu,  \]
\[ \int_{\Zp} f(x) \cdot \varphi\circ\psi(\mu) = \int_{\Zp} f(px) \cdot \psi(\mu) = \int_{p \Zp} f(x) \cdot \mu = \int_{\Zp} f(x) \cdot \mathrm{Res}_{p\Zp}(\mu). \]
In particular, we have
\begin{equation}\label{res to Zp}\mathrm{Res}_{\Zp^\times}(\mu) = (1-\varphi\circ\psi)(\mu).
\end{equation}
The operator $\psi$ also gives an operator on any $F(T) \in \roi_L\lsem T\rsem $ under the Mahler transform, and using the restriction formula above, we see that it is the unique operator satisfying
\begin{equation} \label{Eqphipsi}
\varphi \circ \psi(F)(T) = \frac{1}{p}\sum_{\xi \in \mu_p}F\big((1+T)\xi - 1\big).
\end{equation}
The following result will be useful in Part II.

\begin{corollary} \label{CorollarySupportedZpet}
A measure $\mu \in \Lambda(\Zp)$ is supported on $\Zp^\times$ if and only if $\psi(\Am_\mu) = 0$.
\end{corollary}

\begin{proof}
Let $\mu \in \Lambda(\Zp)$. Then $\mu$ is supported on $\Zp^\times$ if and only if $\mathrm{Res}_{\Zp^\times}(\mu) = \mu$, or equivalently if and only if $\Am_\mu = \Am_\mu - \varphi\circ\psi(\Am_\mu)$, which happens if and only if $\psi(\Am_\mu) = 0$, since the operator $\varphi$ is injective.
\end{proof}

\begin{remark} \label{not subalgebra}
We have an injection $\iota : \Lambda(\Zp^\times) \hookrightarrow \Lambda(\Zp)$ given by
\[\int_{\Zp} \phi \cdot \iota(\mu) = \int_{\Zp^\times} \phi\big|_{\Zp^\times} \cdot \mu,\]
and as $\mathrm{Res}_{\Zp^\times}\circ \iota$ is the identity on $\Lambda(\Zp^\times)$, we can identify $\Lambda(\Zp^\times)$ with its image as a subset of $\Lambda(\Zp)$. By Corollary \ref{CorollarySupportedZpet}, a measure $\mu \in \Lambda(\Zp)$ lies in $\Lambda(\Zp^\times)$ if and only if $\psi(\mu) = 0$. Whilst we identify $\Lambda(\Zp^\times)$ with a subset of $\Lambda(\Zp)$, it is important to remark that it is \emph{not} a subalgebra. Indeed, convolution of measures on a group $G$ uses the group structure of $G$; for $\Zp^\times$ this is multiplicative, and for $\Zp$ this is additive (cf.\ Remark \ref{RemarkConvolution}). If $\lambda$ and $\mu$ are two measures on $\Zp^\times$, writing $\mu *_{\Zp^\times} \lambda$ for the convolution over $\Zp^\times$, we have
\begin{equation}\label{eq:convolution}
\int_{\zpe} f(x) \cdot (\mu*_{\Zp^\times}\lambda) = \int_{\zpe} \bigg( \int_{\zpe} f(x y) \cdot \mu(x) \bigg) \cdot \lambda(y) 
\end{equation}
\end{remark}

\subsection{Pseudo-measures}\label{sec:pseudo-measures}

The Mahler transform gives a correspondence between $p$-adic measures and $p$-adic analytic functions on the open unit ball (explained in Remark \ref{rem:analytic functions} below). The Riemann zeta function, however, is not analytic everywhere, as it has a simple pole at $s=1$. To reflect this, we also need to be able to handle simple poles on the $p$-adic side. We do this via the theory of pseudo-measures.

\begin{definition}
	Let $G$ be a profinite abelian group, and let $Q(G)$ denote the ring of fractions of the Iwasawa algebra $\Lambda(G)$. A \emph{pseudo-measure} on $G$ is an element $\lambda \in Q(G)$ such that 
	\[
 ([g]-[1])\lambda \in \Lambda(G)
 \]
	for all $g \in G$.  (This is using the natural product on the Iwasawa algebra $\Lambda(G)$, which we recall corresponds to convolution \eqref{eq:convolution} of measures).
\end{definition}

 We first explain how to integrate certain functions against pseudo-measures. Let $\chi : G \to \C_p^\times$ be a non-trivial character, and let $\lambda \in Q(G)$ be a pseudo-measure. Then one can define
\begin{equation}\label{eq:integrate pseudo-measure}
\int_G \chi \cdot \lambda \defeq (\chi(g) - 1)^{-1} \int_G \chi \cdot ([g] - [1]) \lambda, 
\end{equation}
where $g \in G$ is any element such that $\chi(g) \neq 1$. This definition does not depend on the choice of $g \in G$. Indeed one has
\[ (\chi(h) - 1) \int_G \chi \cdot ([g] - [1]) \lambda  = \int_G \chi \cdot ([g] - 1) ([h] - 1) \lambda  = (\chi(g) - 1) \int_G \chi \cdot ([h] - 1) \mu, \] 
where in the first and last equalities we used the convolution product and that $G$ is abelian.

 \begin{remark} To motivate this definition, note that if $\chi : G \to \Cp^\times$ is a non-trivial character, then from the definition of convolution product, the function
\begin{align*}
\Lambda(G) &\longrightarrow \Cp,\\
\mu &\longmapsto \int_G \chi \cdot \mu
\end{align*}
is a ring homomorphism. This induces a unique homomorphism $Q(G) \to \Cp$, which -- when applied to a pseudo-measure -- yields the expression \eqref{eq:integrate pseudo-measure} above.
\end{remark}

We will be most interested in pseudo-measures on $G = \Zp^\times$. The following lemma shows that a pseudo-measure $\mu$ on $\Zp^\times$ is uniquely determined by the values $\int_{\Zp^\times}x^k \cdot\mu$ for $k > 0$.

\begin{lemma}\label{lem:zero divisor}
	\begin{itemize}
		\item[(i)]
		Let $\mu \in \Lambda(\Zp^\times)$ be such that
		\[\int_{\Zp^\times} x^k \cdot\mu = 0\]
		for all $k > 0$. Then $\mu = 0$.
		\item[(ii)] Let $\mu \in \Lambda(\Zp^\times)$ be such that
		\[\int_{\Zp^\times} x^k \cdot\mu \neq 0\]
		for all $k > 0$. Then $\mu$ is not a zero divisor in $\Lambda(\Zp^\times).$
		\item[(iii)]Part (i) holds if, more generally, $\mu$ is a pseudo-measure.
	\end{itemize}
\end{lemma}

\begin{proof}
	(i) Note that the vanishing condition forces the Mahler transform $\mathscr{A}_\mu(T) = \sum_{k \geq 0} \left( \int_{\Zp} {x \choose k} \cdot \mu \right) T^k$ of $\mu$ to be constant, since each non-trivial binomial polynomial is a linear combination of strictly positive powers of $x$. As $\mu$ is a measure on $\Zp^\times$, we also have $\psi(\mathscr{A}_\mu)(T) = 0$ by \eqref{res to Zp}. Since $\psi$ is the identity on constants (using, e.g., \eqref{Eqphipsi}), we deduce that $\mathscr{A}_\mu(T) = 0$, so $\mu = 0$.

	(ii) Suppose there exists a measure $\lambda$ such that $\mu *_{\Zp^\times} \lambda = 0$, where the product is the convolution product on $\Zp^\times$ (cf.\ Remark \ref{not subalgebra}). Then
	\[0 = \int_{\Zp^\times}x^k \cdot (\mu *_{\Zp^\times} \lambda) = \int_{\Zp^\times} \bigg( \int_{\Zp^\times} (xy)^k \cdot \mu(x) \bigg) \cdot \lambda(y) = \bigg(\int_{\Zp^\times}x^k \cdot\mu \bigg) \bigg( \int_{\Zp^\times}x^k \cdot\lambda \bigg),\]
	which forces $\lambda = 0$ by part (i). In particular $\mu$ is not a zero divisor.

	(iii) Let $\mu$ be a pseudo-measure satisfying the vanishing condition. Let $a \neq 1$ be an integer prime to $p$; then  $\lambda = ([a] - [1])\mu \in \Lambda(\Zp^\times)$ is a measure by the definition of pseudo-measure, and by \eqref{eq:integrate pseudo-measure} we have
	\[
	\int_{\Zp^\times} x^k \cdot \lambda = (a^k-1) \int_{\Zp^\times} x^k \cdot \mu = 0
	\]
	for all $k > 0$. By part (i), we have $\lambda = 0$. But $[a] - [1]$ satisfies the condition of part (ii), so it is not a zero-divisor, and this forces $\mu = 0$, as required.
\end{proof}

Finally, we give a simpler process for writing down pseudo-measures on $\Zp^\times$. 

\begin{definition} \label{DefAugmentationIdealFiniteLevel}
	The \emph{augmentation ideal} $I((\Zp/p^n)^\times) \subset \roi_L[(\Zp/p^n)^\times]$ is the kernel of the natural `degree' map 
 \[
    \mathscr{O}_L[(\Z/p^n\Z)^\times] \to \mathscr{O}_L, \qquad \sum_a c_a[a] \mapsto \sum_a c_a.
\]
These fit together into a degree map $\Lambda(\Zp^\times) \to \roi_L$; we call its kernel the augmentation ideal $I(\Zp^\times) \subset \Lambda(\Zp^\times)$. One may check directly there is an isomorphism
\[
I(\Zp^\times) \cong \varprojlim I((\Zp/p^n)^\times).
\]
\end{definition}

\begin{lemma}\label{lem:pseudo-measure existence}
	Let $a$ be any topological generator of $\Zp^\times$ (for example, take $a$ to be a primitive root modulo $p$ such that $a^{p - 1} \not\equiv 1 \; (mod \; p^2)$), and $\mu \in \Lambda(\Zp^\times)$ a measure. Then 
	\[
		\mu' \defeq \frac{\mu}{[a]-[1]} \in Q(\Zp^\times)
	\]
	is a pseudo-measure.
\end{lemma}
\begin{proof}
As $p$ is odd, $(\Zp/p^n)^\times$ is cyclic, generated by $\overline{a} \defeq a\newmod{p^n}$, and we have
\[
I((\Zp/p^n)^\times) = ([\overline{a}] - [\overline{1}])\roi_L[(\Zp/p^n)^\times].
\]
In the inverse limit we see that
\[
I(\zp^\times) = ([a]-[1])\Lambda(\Zp^\times).
\]
Thus if $g \in \zp^\times$, we have $[g]-[1] \in I(\Zp^\times)$, and we must have
\[
[g]-[1] = \nu([a]-[1])
\]
for some $\nu \in \Lambda(\zp^\times)$. Then 
\[
([g]-[1])\mu' = \nu([a]-[1])\mu' = \nu \cdot \mu \in \Lambda(\zp^\times),
\]
that is, $\mu'$ is a pseudo-measure.
\end{proof}

Note moreover that \emph{all} pseudo-measures have this shape. Indeed, let $\mu'$ be a pseudo-measure, and $a \in \Zp^\times$ a topological generator; then $\mu = ([a]-[1])\mu'$ is a measure, and $\mu' = \mu/([a]-[1])$ as above.

\subsection{Locally analytic functions and distributions}
\label{sec:locally analytic}

We finally introduce another important space of functions and its dual, namely \emph{locally analytic functions} and \emph{locally analytic distributions}. This subsection may be safely skipped on a first reading, and indeed we make only peripheral use of its content in these notes (in \S\ref{sec:p-adic s=1} and \S\ref{sec:residue}, to study values of the $p$-adic zeta function). The locally analytic theory is nonetheless of fundamental importance in more general settings, so we include a sketch here, and indicate how it dovetails beautifully with the theory of measures we have already studied. All of this -- and other related theories in $p$-adic functional analysis -- are described in detail in \cite{Colm10}.

As motivation, we first note the following.

\begin{remark}\label{rem:analytic functions}
The space $\sM(\Zp,L)$ of measures has an interpretation via rigid analysis. To explain this, consider the $p$-adic open unit ball in $\Cp$, i.e.\ the space
\[
B(0, 1) = \{ z \in \cp \, | \, |z| < 1 \} \subset \Cp.
\]
This is the set of $\Cp$-points of a rigid analytic space in the sense of \cite{BGR}. An $L$-valued function on $B(0,1)$ is \emph{rigid analytic} if it can be written as a power series $\sum_{n \geq 0}a_n T^n \in L\lsem T\rsem$ that is everywhere convergent on $B(0,1)$ (i.e.\ $|a_n|r^n \to 0$ as $n \to \infty$ for all $r < 1$); write $\mathscr{R}^+ \subset L\lsem T\rsem$ for the space of such functions. A rigid analytic function is \emph{bounded} if the $|a_i|$ are bounded. 

Note that  the space of bounded $L$-valued rigid analytic functions is $\roi_L\lsem T\rsem  \otimes_{\roi_L} L$, which, via the Mahler transform (Remark \ref{integ}), is isomorphic to $\sM(\Zp,L)$. Hence $p$-adic measures on $\Zp$ can be viewed as bounded rigid analytic functions on $B(0, 1)$.
\end{remark}

It is natural to ask if the Mahler correspondence, as studied in Theorem \ref{thm:mahler}, can be extended from $\roi_L\lsem T\rsem$ to all of $\mathscr{R}^+$. Such an extension is given by locally analytic distributions, in the sense described in Theorem \ref{thm:mahler la} and equation \eqref{eq:extend mahler} below.

\begin{definition} \label{DefLocAn}
Let $L/\Qp$ be a finite extension, and let $f : \Zp \to L$
be a function. 
\begin{enumerate}
\item For $z \in \Zp$, we say $f$ is \emph{locally analytic at $z$} if $f$ can be described locally around $z$ by a convergent power series. Precisely, this means there exists some integer $n_z \geq 0$ and numbers $\{a_k(z) \in L : k\geq 0\}$, such that 
\[
    \sum_{k\geq 0}a_k(z) \cdot (x-z)^k
\]
converges to $f(x)$ for all $x \in U_z \defeq z + p^{n_z} \zp$.

\item We say $f$ is \emph{locally analytic} if it is locally analytic at all $z \in \Zp$.

\item We write $\sC^{\mathrm{la}}(\Zp,L)$ for the $L$-vector space of all locally analytic functions $\Zp \to L$.
\end{enumerate}
\end{definition}

Recall that $\sC(\Zp,L)$ can be equipped with a valuation that makes it into  an $L$-Banach space, and that the space of measures $\sM(\Zp,L)$ was defined as its continuous dual. Analogously, the space $\sC^{\mathrm{la}}(\Zp,L)$ has a natural topology, described as follows (see \cite[\S I.4]{Colm10} for more details). For each $n \in \N$, we say a function $f : \zp \to L$ is \emph{locally analytic of radius $p^{-n}$} if it is locally analytic and moreover we can take $n_z = n$ for all $z \in \Zp$; in other words, it is analytic (described by a single power series) on each open set of form $z + p^n\Zp$. Denote by $\mathscr{C}^{n-\mathrm{an}}(\zp, L)$ the subspace of such functions. Then one can check that $\mathscr{C}^{n-\mathrm{an}}(\zp, L)$ is an $L$-Banach space with valuation given by 
\[
    v_n(f) \defeq \inf_{z \in (\zp / p^n \zp)} \inf_{k \in \N} \Big(nk + v_p(a_k(z))\Big),
\]
or equivalently norm given by 
\[
    ||f||_n \defeq \sup_{z \in (\zp / p^n \zp)} \sup_{k \in \N} |a_k(z)| p^{-nk}.
\]
Moreover, almost by definition, we have that $\mathscr{C}^{\mathrm{la}}(\zp, L) = \varinjlim_{n \in \N} \mathscr{C}^{n-\mathrm{an}}(\zp, L)$, and so it inherits a natural topology given by the direct limit topology. 

\begin{remark} 
Locally analytic functions are continuous, so $\sC^{\mathrm{la}}(\Zp,L) \subset \sC(\Zp,L)$. Note, however, that the topology we just defined on $\sC^{\mathrm{la}}(\zp,L)$ is \emph{not} the one induced from $\sC(\zp,L)$. Nonetheless the image of this inclusion is dense, as, for example, locally constant functions are locally analytic functions and are dense in $\mathscr{C}(\zp, L)$. 
\end{remark}

Analogously to Definition \ref{def:measures}, we make the following definition.

\begin{definition}
We define the space  $\sD^{\mathrm{la}}(\Zp,L)$ of locally analytic distributions on $\zp$ to be the continuous dual $\mathrm{Hom}_{\mathrm{cts}}(\sC^{\mathrm{la}}(\Zp,L),L)$.
\end{definition}

If $\mu$ is a locally analytic distribution on $\Zp$, and $\phi \in \sC^{\mathrm{la}}(\Zp,L)$, we continue to write 
\[
    \int_{\Zp} \phi(x) \cdot \mu(x) \defeq \mu(\phi).
\]
The binomial polynomials ${x \choose n}$ are visibly locally analytic, so we may also extend the Mahler transform to this generality; namely, if $\mu \in \sD^{\mathrm{la}}(\Zp,L)$, define
\[
\sA_{\mu}(T) \defeq \int_{\zp}(1+T)^x \cdot \mu(x) = \sum_{n \geq 0} \left(\int_{\Zp}{x \choose n}\cdot \mu\right) T^n \in L\lsem T\rsem.
\]

The following crucial result provides the desired extension of the Mahler transform beyond bounded measures/power series.

\begin{theorem}\label{thm:mahler la}
The Mahler transform induces a bijection
\[
    \sD^{\mathrm{la}}(\Zp,L) \longrightarrow \mathscr{R}^+ \subset L\lsem T \rsem.
\]
\end{theorem}
\begin{proof}
This is \cite[Thm.\ II.2.2]{Colm10}.
\end{proof}

As with Theorem \ref{thm:mahler}, the theorem says more: the bijection respects natural topologies on both sides. Here the topologies are as follows:
\begin{itemize}
\item $\sD^{\mathrm{la}}(\Zp,L)$ is the inverse limit of the continuous duals $\sD^{n-\mathrm{an}}(\Zp,L)$; each of these is a Banach space with the natural dual (strong) topology (see Remark \ref{rem:strong topology}). This endows $\sD^{\mathrm{la}}(\zp,L)$ with the corresponding inverse limit topology. 
\item As the open unit disc is the union of the closed discs $ B(0, r) \defeq \{ z \in \Cp : |z| \leq r \}$ of radius $r < 1$, we can write $\mathscr{R}^{+}$ as the inverse limit over $r < 1$ of the Banach spaces $\mathcal{O}(B(0, r))$ of analytic functions on $B(0, r)$. Again we get an inverse limit topology on $\mathscr{R}^+$. 
\end{itemize}
A topology induced from an inverse limit of Banach spaces is called a \emph{Fr\'echet topology}. The Mahler transform is then an isomorphism of Fr\'echet spaces.

\begin{remark}
If $\mu \in \sM(\Zp,L)$ is any measure, then we obtain a locally analytic distribution $\widetilde{\mu} \in \sD^{\mathrm{la}}(\Zp,L)$ by restricting to $\sC^{\mathrm{la}}(\Zp,L) \subset \sC(\Zp,L)$. As locally analytic functions are dense inside continuous functions, the association $\mu \mapsto \widetilde{\mu}$ is injective, and hence this identification allows us to consider $\sM(\Zp,L) \subset \sD^{\mathrm{la}}$ as a subset. The combination of Theorems \ref{thm:mahler} and \ref{thm:mahler la} says that this inclusion is compatible with the natural inclusion of bounded power series inside power series converging in the open unit disc, that is, the following diagram commutes:
\begin{equation}\label{eq:extend mahler}
\xymatrix@C=20mm@R=6mm{
\mathscr{D}^{\mathrm{la}}(\zp,L)\sar{d}{\supset} \ar[r]^{\mu \longmapsto \sA_\mu} & \mathscr{R}^+\sar{d}{\supset}\\
\sM(\Zp,L) \ar[r]^{\mu \longmapsto \sA_\mu} & \roi_L\lsem T\rsem \otimes_{\roi_L} L.
}
\end{equation}
\end{remark}

\begin{remark} \label{rem:distribution toolbox}
Observe that every part of our `measure-theoretic toolbox' from \S\ref{sec:toolbox}, including corresponding operations on Mahler transforms, carries over identically to the setting of locally analytic distributions.
\end{remark}
\begin{remark}
Locally analytic $p$-adic analysis is fundamentally important in the study of $p$-adic $L$-functions (and in many other areas, such as in the study of $p$-adic automorphic forms, in the $p$-adic Langlands correspondence, in $p$-adic Hodge theory, etc.). Indeed, more general $p$-adic $L$-functions -- for example, those attached to elliptic curves and modular forms -- are frequently not measures or pseudo-measures, but rather locally analytic distributions with prescribed growth, in the sense of \cite{AV75}. We discuss this further in Appendix \ref{sec:modular forms}.
\end{remark}

\subsection{Further remarks}

	The following remarks will not be seriously used in the sequel, but are included for completeness, and to illustrate some other ways that the objects studied in this chapter appear in the literature.

\begin{remark}\label{rem:weight space rigid analytic}
We have an analogue of Remark \ref{rem:analytic functions} for measures on $\Zp^\times$, using instead the multiplicative structure. The key object here is the \emph{weight space} 
 \[\mathcal{W}(\Cp) = \mathrm{Hom}_{\mathrm{cts}}(\Zp^\times,\Cp^\times). \]
 Since $\zpe = \mu_{p - 1} \times (1 + p \zp)$, we have 
 \[
 \Homc(\Zp^\times,\Cp^\times) \cong \Homc(\mu_{p-1},\Cp^\times) \times \Homc(1+p\Zp, \Cp^\times).
 \]
Moreover, evaluation at a topological generator of $1+p\Zp$ identifies $\Homc(1+p\zp,\Cp^\times)$ with $B(0,1)$ from above. Hence we may identify $\Homc(\Zp^\times,\Cp^\times)$ with $p-1$ copies of $B(0,1)$. Summarising,
\[
\cW(\Cp)  = \bigsqcup_{\nu \in (\mu_{p-1})^\vee} U_\nu,
\]
where:
\begin{itemize}
    \item $\nu$ ranges over the $p-1$ different characters of $\mu_{p-1}$.
    \item $U_\nu \subset \cW(\Cp)$ is the subset of characters $\chi$ of $\Zp^\times$ with $\chi|_{\mu_{p-1}} = \nu$.
    \item Each $U_\nu$ may be identified with $B(0,1)$.  
\end{itemize}

 This space can also be given more structure; there is a rigid analytic space $\mathcal{W}$ such that the elements of $\mathcal{W}(\Cp)$ are the $\Cp$-points of $\mathcal{W}$. Analogously to above, a theorem of Amice says that giving a measure $\mu$ on $\Zp^\times$ is equivalent to giving a bounded rigid analytic function $F_\mu$ on $\mathcal{W}$. This equivalence is given as follows: if $\mu$ is a measure on $\zpe$ and $\chi : \zpe \to \C_p^\times$ is a character (seen as a point on $\mathcal{W}(\Cp)$), then one defines $F_\mu(\chi) \defeq \int_{\zpe} \chi \cdot \mu$. Observe that the multiplicative convolution product corresponds to pointwise multiplication of rigid functions. 
 
 Finally, if $\lambda \in Q(\zpe)$ is a pseudo-measure, then it is of the form $\frac{\mu}{([a] - [1])}$ for some topological generator $a \in \zpe$ and some measure $\mu \in \Lambda(\zpe)$. Note $\int_{\Zp^\times} \chi \cdot ([a]-[1]) = 0$ if and only if $\chi(a) = 1$, which -- as $a$ is a topological generator of $\Z_p^\times$ -- implies $\chi$ is the trivial character. Hence, as a function on the weight space, $\lambda$ might have a simple pole at the trivial character. So pseudo-measures can be seen as rigid analytic functions on the weight space that possibly have a simple pole at the trivial character.
\end{remark}

\begin{remark}
Power series rings have been generalised to what now are called Fontaine rings. It turns out that Galois representations are connected to certain modules over these rings called $(\varphi, \Gamma)$-modules. The operations described above generalise to fundamental operations on $(\varphi,\Gamma)$-modules, and their interpretation via $p$-adic analysis inspired the proof of the $p$-adic Langlands correspondence for $\mathrm{GL}_2(\qp)$ (see \cite{Colm10b}).
\end{remark}

%%======================================================================
%%======================================================================

\section{The Kubota-Leopoldt $p$-adic $L$-function}\label{kl}

In this section, we prove the following:

\begin{theorem}\label{thm:kubota leopoldt theorem}
	There is a unique pseudo-measure $\zeta_p$ on $\Zp^\times$ such that, for all $k > 0$, we have
	\[\int_{\Zp^\times}x^k \cdot\zeta_p = (1-p^{k-1})\zeta(1-k).\]
\end{theorem}

This pseudo-measure, denoted $\zeta_p^{\mathrm{an}}$ in \S\ref{sec:intro}, is the \emph{Kubota--Leopoldt $p$-adic $L$-function}.

\subsection{The measure $\mu_a$} \label{KTconst}

Recall from Lemma \ref{lem:FormulaZeta} that we can write the Riemann zeta function in the form
\[(s-1)\zeta(s) = L(f, s-1) \defeq \frac{1}{\Gamma(s-1)}\int_0^\infty f(t)t^{s-2}dt,\]
where $f(t) = t/(e^t-1)$, and that $\zeta(-k) = (d^kf/dt^k)(0) = (-1)^k B_{k+1}/(k+1).$ We want to remove the smoothing factor at $s=1$. For this, let $a$ be an integer coprime to $p$ and consider the related function
\[f_a(t) = \frac{1}{e^t-1} - \frac{a}{e^{at}-1}.\]
This is also $\mathscr{C}^\infty$ and rapidly decreasing, so we can apply Theorem \ref{thm:l-function} and consider the function $L(f_a,s)$. The presence of $a$ removes the factor of $s-1$, at the cost of introducing a different smoothing factor.

\begin{lemma}\label{lem:values of zeta}
We have
\[L(f_a,s) =(1-a^{1-s})\zeta(s),\]
which has an analytic continuation to $\C$, and
\[f_a^{(k)}(0) = (-1)^k(1-a^{1+k})\zeta(-k).\]
\end{lemma}

\begin{proof}
This follows from calculations similar to those in the proof of Lemma \ref{lem:FormulaZeta}.
\end{proof}

We now introduce the $p$-adic tools we have developed into the picture. We will start with the function $f_a(t)$, and slowly manipulate it until we construct a (pseudo-)measure with the desired interpolation properties. Note first the following very simple observation.

\begin{lemma}\label{lem:define F_a}
Under the substitution $e^t = T+1$, the derivative $d/dt$ becomes the operator $\partial = (1+T)\frac{d}{dT}$. In particular, if we define
\[ F_a(T) \defeq \frac{1}{T} - \frac{a}{(1+T)^a - 1},\]
we have
\begin{equation}\label{eq:riemann mahler} 
f_a^{(k)}(0) = \big( \partial^k F_a \big)(0).
\end{equation}
\end{lemma}

The left-hand side of \eqref{eq:riemann mahler} computes the $L$-value $\zeta(-k)$ by Lemma \ref{lem:values of zeta}. The right-hand side is similar to Corollary \ref{cor:eval at x^k}, which expressed the integral $\int_{\Zp} x^k \cdot \mu$ in terms of the Mahler transform $\mathscr{A}_\mu$. This motivates us to seek a measure $\mu_a$ with $\mathscr{A}_{\mu_a} = F_a$. This is possible by: 

\begin{proposition} \label{PropFaT}
The function $F_a(T)$ is an element of $\Zp\lsem T\rsem $.
\end{proposition}

\begin{proof}
We can expand 
\[ (1+T)^a - 1 = \sum_{n\geq 1} {a \choose n} T^n = aT\big[1+Tg(T)\big], \]
where $g(T) = \sum_{n\geq 2}\frac{1}{a} {a \choose n} T^{n-2}$ has coefficients in $\zp$ since we have chosen $a$ coprime to $p$. Hence, expanding the geometric series, we find
\[\frac{1}{T} - \frac{a}{(1+T)^a - 1} = \frac{1}{T} \sum_{n \geq 1}(-T)^n g(T)^n,\]
which is visibly an element of $\Zp\lsem T\rsem $.
\end{proof}

\begin{definition} \label{DefinitionMeasuremua}
Let $\mu_a$ be the measure on $\Zp$ whose Mahler transform is $F_a(T)$.
\end{definition}

\begin{proposition}
For $k\geq 0$, we have
\[\int_{\Zp}x^k \cdot\mu_a = (-1)^k(1-a^{k+1}) \zeta(-k).\]
\end{proposition}

\begin{proof}
By Corollary \ref{cor:eval at x^k}, the left-hand side is $(\partial^k\sA_{\mu_a})(0)$. By definition of $\mu_a$ and Lemma \ref{lem:define F_a} this is $(\partial^kF_a)(0) = f_a^{(k)}(0)$. This equals the right-hand side by Lemma \ref{lem:values of zeta}.
\end{proof}

\subsection{Restriction to $\Zp^\times$}

Recall from the introduction that we want the $p$-adic analogue of the Riemann zeta function to be a measure on $\Zp^\times$, not all of $\Zp$. We have already defined a restriction operator in equation (\ref{res to Zp}), which on Mahler transforms acts as $1 - \varphi\circ\psi$.  We begin with a short but important property of the measure $\mu_a$.

\begin{lemma} \label{LemmaPsiInvariant}
We have $\psi(\mu_a) = \mu_a$.
\end{lemma}

\begin{proof}
We show the result by considering the action on power series. We wish to show $\psi(F_a) = F_a$. First note that $F_a(T) = \frac{1}{T} - a \cdot\sigma_a(\frac{1}{T})$, for $\sigma_a$ as in \S\ref{SubSectionphipsi}. As $\psi$ commutes with $\sigma_a$, we have $\psi(F_a) = \psi(\frac{1}{T}) - a\cdot \sigma_a\psi(\frac{1}{T})$, so it suffices to show $\psi(\frac{1}{T}) = \frac{1}{T}$.

By definition (cf.\ equation \eqref{Eqphipsi}) we have 
\begin{align*} (\varphi \circ \psi) \left(\frac{1}{T}\right) &= p^{-1} \sum_{\xi \in \mu_p} \frac{1}{ (1 + T) \xi - 1} = \frac{1}{(1 + T)^p - 1} = \varphi\left(\frac{1}{T}\right), 
\end{align*}
as can be seen by calculating the partial fraction expansion. By injectivity of $\varphi$, we deduce that $\psi(\frac{1}{T}) = \frac{1}{T}$, and conclude.
\end{proof}

\begin{proposition} \label{PropInterpolation1}
We have
\begin{equation}\label{eq:first interpolation}
	\int_{\Zp^\times}x^k \cdot\mu_a = (-1)^k(1-p^k)(1-a^{k+1})\zeta(-k).
\end{equation}
(In other words, restricting to $\Zp^\times$ removes the Euler factor at $p$).
\end{proposition}

\begin{proof}
Since $\mathrm{Res}_{\zpe} = 1 - \varphi \circ \psi$, we deduce that
\[\int_{\zpe} x^k \cdot\mu_a = \int_{\Zp} x^k \cdot (1 - \varphi \circ \psi) \mu_a = \int_{\Zp} x^k \cdot (1 - \varphi)\mu_a = (1 - p^k) \int_{\Zp} x^k \cdot \mu_a,\]
where for the second equality we have used Lemma \ref{LemmaPsiInvariant}. This finishes the proof.
\end{proof}

%%======================================================

\subsection{Rescaling and removing dependence on $a$}\label{sec:dep on a}

Finally we remove the dependence on $a$. Thus far, the presence of $a$ has acted as a `smoothing factor' which removes the pole of the Riemann zeta function; so to remove it, we must be able to handle such poles on the $p$-adic side. We use the notion of pseudo-measures from \S\ref{sec:pseudo-measures}.

\begin{definition}
Let $a$ be an integer that is prime to $p$, and let $\theta_{a}$ denote the element of $\Lambda(\Zp^\times)$ corresponding to $[a] - [1]$. Note that, by definition, we have
\[
\int_{\Zp^\times} x^k \cdot\theta_{a} = a^k - 1
\]
However, in \eqref{eq:first interpolation} it is $a^{k+1} -1$ that appears. To bridge this gap, note that on $\Zp^\times$, we have a well-defined operation `multiplication by $x^{-1}$' given by
\begin{equation}\label{eq:mult by xinverse}
	\int_{\Zp^\times} f(x) \cdot x^{-1}\mu \defeq \int_{\Zp^\times} x^{-1}f(x) \cdot \mu,
\end{equation}
and that 
\[
	\int_{\Zp^\times} x^k \cdot x^{-1} \mu_a = (-1)^k(a^k-1)(1-p^{k-1})\zeta(1-k).
\]
We comment further on this  multiplication by $x^{-1}$ in Remark \ref{rem:renormalise}. 

\begin{definition} \label{DefZetap}
Let $a$ be a topological generator of $\zpe$. The \emph{$p$-adic zeta function} is
\[
	\zeta_p \defeq \frac{x^{-1}\mathrm{Res}_{\Zp^\times}\mu_a}{\theta_a} \in Q(\Zp^\times).
\]
\end{definition}
%Then define $\theta_a \defeq x\widetilde{\theta_a}$, so that
%\[\int_{\Zp^\times} x^k \cdot\theta_{a} = a^{k+1} - 1,\] 
%the term arising in our interpolation formula. Then define
%\[\zeta_p \defeq \frac{\mu_{a}}{\theta_{a}} \in Q(\Zp^\times).\]
\end{definition}

\begin{proposition} \label{PropInterpolation2}
The element $\zeta_p$ is a well-defined pseudo-measure satisfying
\[
\int_{\Zp^\times} x^k \cdot \zeta_p = (1-p^{k-1})\zeta(1-k) \qquad\qquad \text{for all }k > 0.
\]
\end{proposition}

\begin{proof}
We see $\zeta_p$ is a pseudo-measure by Lemma \ref{lem:pseudo-measure existence}. It is independent of the choice of $a$ by Lemma \ref{lem:zero divisor}(iii).

Using Equation \eqref{eq:integrate pseudo-measure} (to integrate the pseudo-measure) and Proposition \ref{PropInterpolation1}, we obtain the interpolation property
 \[
\int_{\Zp^\times} x^k \cdot \zeta_p = (-1)^k(1-p^{k-1})\zeta(1-k).
\]
To get the result, we may remove the $(-1)^{k}$ as $\zeta(1-k) \neq 0$ if and only if $k$ is even.   
\end{proof}

We finally prove Theorem \ref{thm:kubota leopoldt theorem}. Existence of the pseudo-measure is Proposition \ref{PropInterpolation2}. To conclude the proof we need only show uniqueness; but this follows from  Lemma \ref{lem:zero divisor}(iii). \qed

%%======================================================================
%%======================================================================
\section{Interpolation at Dirichlet characters}\label{interpolation}

In our study of the Kubota--Leopoldt $p$-adic $L$-function, the entire construction was essentially built to interpolate special values of the Riemann zeta function, so this property should not have come as a surprise. Now, however, some real magic happens. Since the introduction, we've not mentioned Dirichlet $L$-functions once -- but, miraculously, the Kubota--Leopoldt $p$-adic $L$-function also interpolates Dirichlet $L$-values as well.

\subsection{Characters of $p$-power conductor}

We start studying the interpolation properties when twisting by a Dirichlet character of conductor a power of $p$.

\begin{theorem}\label{thm:tame conductor}
Let $\chi$ be a (primitive) Dirichlet character of conductor $p^n$ for some integer $n \geq 1$ (seen as a locally constant character of $\Zp^\times$, cf.\ \S \ref{sec:dirichlet ideles}). Then, for $k > 0$, we have
\[ \int_{\Zp^\times}\chi(x)x^k \cdot\zeta_p = L(\chi,1-k).\]
\end{theorem}
The rest of this subsection will contain the proof of this result. The proof is somewhat calculation-heavy, but -- given familiarity with the dictionary between measures and power series -- is not conceptually difficult. 

In particular: the Riemann zeta function was the complex Mellin transform of a real analytic function, which -- via Theorem \ref{thm:l-function} -- gave us a formula for its special values. Under the transformation $e^t = T+1$, we obtained a $p$-adic power series; and under the measures--power series correspondence given by the Mahler transform, this gave us a measure on $\Zp$, from which we constructed $\zeta_p$. To obtain interpolation at Dirichlet characters, we pursue this in reverse, as summarised in the following diagram:
\[
\xymatrix{
	(1-a^{1-s})\zeta(s) \ar@{<->}[d]^-{\text{Mellin}} & 	&				(1-\chi(a)a^{1-s})L(\chi,s) & \\
	f_a(t) \ar@{<->}[r]_-{e^t = T+1} & F_a(T)\in \mathcal{O}_L\lsem T\rsem  \ar@{<->}[d]^-{\text{Mahler}} &		f_{a,\chi}(t) \ar@{<->}[r]_-{e^t = T+1}\ar@{<->}[u]^-{\text{Mellin}}& F_{a,\chi}(T) \in \mathcal{O}_L\lsem T\rsem \\
	& \mu_a \in \Lambda(\Zp)\ar[d]\ar@{<-->}[rr]^-{\text{``twist by $\chi$''}} &													&							\mu_{a,\chi} \in \Lambda(\Zp)\ar@{<->}[u]^-{\text{Mahler}}\\
		& \zeta_p	&&
}
\]

Firstly, we introduce a twisting operation on measures. If $\mu$ is a measure on $\Zp$, we define a measure $\mu_{\chi}$ on $\Zp$ by
\begin{equation}\label{eq:twist by chi}
\int_{\Zp}f(x) \cdot\mu_{\chi} = \int_{\Zp}\chi(x)f(x) \cdot\mu. 
\end{equation}
 Observe that, as $\chi$ is supported on $\zpe$, the twisted measure $\mu_{\chi}$ is automatically supported on $\zpe$ as well. In particular, under this we have
\begin{align*}
	\int_{\Zp^\times}\chi(x)x^k \cdot \zeta_p = \int_{\Zp^\times} x^k \cdot (\zeta_p)_{\chi} = \big(\partial^k \sA_{(\zeta_p)_\chi}\big)(0),
\end{align*}
where the last equality follows from Corollary \ref{cor:eval at x^k}. Thus we want to determine the Mahler transform of $\mu_{\chi}$ in terms of $\Am_\mu$, for which we use our measure-theoretic toolkit. This requires a classical definition.

\begin{definition} \label{DefGausssum}
Let $\chi$ be a primitive Dirichlet character of conductor $p^n$, $n \geq 1$. Define the \emph{Gauss sum of $\chi$} as 
\[ G(\chi) \defeq \sum_{c\in(\Z/p^n\Z)^\times} \chi(c) \epsilon_{p^n}^c, \]
where $(\epsilon_{p^n})_{n \in \N}$ denotes a system of primitive $p$-power roots of unity in $\overline{\Q}_p$ such that $\epsilon_{p^{n + 1}}^p = \epsilon_{p^n}$ for all $n \geq 0$ (if we fix an isomorphism $\overline{\Q}_p \cong \C$, then one can take $ \epsilon_{p^n} \defeq e^{2\pi i/p^n}$).
\end{definition}

\begin{remark} \label{RemPropertiesGaussSum}
 We note the following basic properties of Gauss sums (see \cite[\S4.3]{DS05}):
 \begin{itemize}
  \item[(i)] $G(\chi) G(\chi^{-1}) = \chi(-1) p^n.$
  \item[(ii)] $G(\chi) = \chi(a) \sum_{c \in (\Z / p^n \Z)^\times} \chi(c) \epsilon_{p^n}^{ac}$ for any $a \in \zpe$. %Note that, if $a \notin \zpe$, both sides vanish.
 \end{itemize}
\end{remark}

%\begin{remark}
%The Gauss sum of a character is closely related to its Fourier transform. Indeed, for any locally constant function $f : \Qp \to \Cp$ with compact support, one can define its Fourier transform as
%\[ \widehat{f}(x) \defeq p^{-m} \sum_{y \text{ mod } p^m} f(y) e^{-2 \pi i x y}, \]
%where $m$ is a sufficiently large integer and $e^{-2 \pi i x y}$ denotes the $p$-th power root of unity defined by $\Qp \to \Qp / \Zp \cong \Z[1/p] / \Z$ and by our compatible system $(\varepsilon_{p^n})_{n \in \N}$ of primitive roots of unity. For $\chi$ a Dirichlet character of conductor $p^n$, one shows easily that
%\[ \widehat{\chi}(x) = \left\{\begin{array}{cl}
%\frac{1}{G(\chi^{-1})} \chi^{-1}(p^n x) & : n > 0, \\
%\mathbf{1}_{\Zp}(x) - p^{-1} \mathbf{1}_{p^{-1} \Zp}(x) & : n = 0.\end{array}\right.
%\]
%\end{remark}

\begin{lemma}\label{lem:mahler chi}
The Mahler transform of $\mu_{\chi}$ is
\[\Am_{\mu_\chi}(T) = \frac{1}{G(\chi^{-1})}\sum_{c \in (\Z/p^n \Z)^\times} \chi(c)^{-1}\Am_{\mu} \big( (1+T)\epsilon_{p^n}^c - 1 \big).\]
\end{lemma}

\begin{proof}
Since $\chi$ is constant modulo $p^n$, the measure $\mu_{\chi}$ is simply
\[\mu_{\chi} = \sum_{c\in(\Z/p^n \Z)^\times} \chi(c)\mathrm{Res}_{c+p^n\Zp}(\mu).\]
Using this expression and the formula for the Mahler transform of the restriction of a measure to $c + p^n\Zp$ given in \eqref{EqRestrictionFormula}, we find that
\[\Am_{\mu_{\chi}}(T) = \frac{1}{p^n}\sum_{b\in (\Z/p^n \Z)^\times}\chi(b)\sum_{\xi\in\mu_{p^n}}\xi^{-b}\Am_{\mu}\big( (1+T)\xi - 1 \big).\]
Writing $\mu_{p^n} = \{\epsilon_{p^n}^c : c = 0,...,p^n -1\},$ and rearranging the sums, we have
\begin{align*}
\Am_{\mu_{\chi}}(T) &= \frac{1}{p^n}\sum_{c\newmod{p^n}}\sum_{b\in (\Z/p^n \Z)^\times}\chi(b)\epsilon_{p^n}^{-bc}\Am_{\mu}\big( (1+T)\epsilon_{p^n}^c - 1\big) \\
&= \frac{1}{p^n}\sum_{c\in(\Z/p^n \Z)^\times} G(\chi) \chi(-c)^{-1} \Am_{\mu} \big( (1+T)\epsilon_{p^n}^c - 1\big) \\
&= \frac{1}{G(\chi^{-1})} \sum_{c\in(\Z/p^n \Z)^\times} \chi(c)^{-1} \Am_{\mu} \big( (1+T)\epsilon_{p^n}^c - 1\big),
\end{align*}
where the second equality follows from Remark \ref{RemPropertiesGaussSum}(ii) and the last one from Remark \ref{RemPropertiesGaussSum}(i). This finishes the proof.
\end{proof}

We now consider the case where $\mu = \mu_a$ from Definition \ref{DefinitionMeasuremua}, the measure from which we built the Kubota--Leopoldt $p$-adic $L$-function, and which has Mahler transform
\[\Am_{\mu_a}(T) = \frac{1}{T} - \frac{a}{(1+T)^a - 1}.\]
Applying the above transformation, we obtain a measure $\mu_{\chi,a}$ with Mahler transform
\[F_{\chi,a}(T) = \frac{1}{G(\chi^{-1})}\sum_{c \in (\Z/p^n \Z)^\times} \chi(c)^{-1}\left[\frac{1}{(1+T)\epsilon_{p^n}^c - 1} - \frac{a}{(1+T)^a\epsilon_{p^n}^{ac} - 1}\right].\]
Via the standard substitution $e^t = T+1$, this motivates the study of the function 
\[f_{\chi,a}(t) = \frac{1}{G(\chi^{-1})}\sum_{c \in (\Z/p^n \Z)^\times} \chi(c)^{-1}\left[\frac{1}{e^t\epsilon_{p^n}^c - 1} - \frac{a}{e^{at}\epsilon_{p^n}^{ac} - 1}\right],\]
by analogy with the case of the Riemann zeta function.

\begin{lemma}\label{lem:dirichlet integral}
We have
\[L(f_{\chi,a},s) = \chi(-1)\big(1-\chi(a)a^{1-s}\big)L(\chi,s),\]
where $L(f_{\chi,a},s)$ is as defined in Theorem \ref{thm:l-function}. Moreover, for $k \geq 0$, we have
\[
f_{\chi,a}^{(k)}(0) =   \begin{cases}
           -\big(1-\chi(a)a^{k+1}\big)L(\chi,-k) & \quad \text{if } \chi(-1) (-1)^k = -1 \\
    0  & \quad \text{if } \chi(-1) (-1)^k = 1.
  \end{cases}
\]
\end{lemma}

\begin{proof}
We follow a similar strategy to that used for the Riemann zeta function (in Lemma \ref{lem:FormulaZeta}). In particular, expanding as a geometric series, we obtain
\[ \frac{1}{e^t\epsilon_{p^n}^{c} - 1} = \sum_{k\geq 1} e^{-kt}\epsilon_{p^n}^{-kc}.\]
Then we have
\[L(f_{\chi,a},s) = \frac{1}{\Gamma(s)G(\chi^{-1})}\int_0^\infty \sum_{c \in (\Z/p^n \Z)^\times} \chi(c)^{-1}\sum_{k\geq 1} \bigg(e^{-kt}\epsilon_{p^n}^{-kc} - e^{-akt}\epsilon_{p^n}^{-akc}\bigg) t^{s-1}dt.\]
Note that
\[\sum_{c\in(\Z/p^n \Z)^\times}\chi(c)^{-1}\epsilon_{p^n}^{-akc} = \chi(-ak)G(\chi^{-1}),\]
and similarly for the first term, so that the expression collapses to
\[L(f_{\chi,a},s) = \frac{1}{\Gamma(s)}\int_0^\infty \sum_{k\geq 1} \chi(-k) \big(e^{-kt} - \chi(a)e^{-akt}\big)t^{s-1}dt.\]
For Re$(s) \gg 0$, we can rearrange the sum and the integral, and then we can evaluate the $k$-th term of the sum easily to $(1-\chi(a)a^{1-s})k^{-s}$, giving
\[L(f_{\chi,a},s) = \chi(-1)\big(1-\chi(a)a^{1-s}\big)\sum_{k\geq 1}\chi(-k)k^{-s} = \chi(-1)\big(1-\chi(a)a^{1-s}\big) L(\chi, s),\]
showing the equality of $L$-functions. 

 To see the final statement about special values, we use Theorem \ref{thm:l-function}, which immediately says 
 \begin{equation}\label{eq:special value theorem 1}
    f_{\chi,a}^{(k)}(0) = (-1)^{k}\chi(-1)(1-\chi(a)a^{k+1})L(\chi,-k).
\end{equation}
To get the claimed statement, we note that 
\[
    \frac{1}{e^{-t}\epsilon_{p^n}^c - 1} = -1 -\frac{1}{e^t\epsilon_{p^n}^{-c} - 1}, 
\]
and using this twice, we find
\[
    f_{\chi, a}(-t) = -\frac{1}{G(\chi^{-1})}\sum_{c \in (\Z/p^n \Z)^\times} \chi(c)^{-1}\left[\frac{1}{e^t\epsilon_{p^n}^{-c} - 1} - \frac{a}{e^{at}\epsilon_{p^n}^{-ac} - 1}\right].
\]
Changing $c$ for $-c$ yields $f_{\chi, a}(-t) = - \chi(-1) f_{\chi, a}(t)$, whence $(-1)^kf^{(k)}_{\chi,a}(0) = -\chi(-1)f^{(k)}_{\chi,a}(0)$. This implies that $f_{\chi, a}^{(k)}(0) = 0$ unless $\chi(-1) (-1)^k = -1$, concluding the proof.
\end{proof}

\begin{remark}
By  \eqref{eq:special value theorem 1} and the above proof, we recover the well-known fact that $L(\chi,-k) = 0$ if $\chi(-1)(-1)^k = 1$.
\end{remark}

We can now prove Theorem \ref{thm:tame conductor}.

\begin{proof}\emph{(Theorem \ref{thm:tame conductor}).}
Since $\chi$ is 0 on $p\Zp$, we have
\[ \int_{\Zp^\times}\chi(x)x^k \cdot\mu_a = \int_{\Zp}\chi(x)x^k \cdot\mu_a = \int_{\Zp}x^k \cdot\mu_{\chi,a}, \]
where $\mu_{\chi,a}$ is the twist of $\mu_a$ by $\chi$. We know this integral to be
\[\big( \partial^k F_{\chi,a} \big)(0) = f_{\chi,a}^{(k)}(0),\]
under the standard transform $e^t = T+1$. Hence, by Lemma \ref{lem:dirichlet integral}, we find
\[\int_{\Zp^\times}\chi(x)x^k \cdot\mu_a = -(1-\chi(a)a^{k+1})L(\chi,-k),\]
so that 
\[\int_{\Zp^\times}\chi(x)x^k \cdot x^{-1}\mu_a = -(1-\chi(a)a^{k})L(\chi,1-k).\]
By definition, we have
\[\int_{\Zp^\times}\chi(x)x^k \cdot\theta_a = -(1-\chi(a)a^{k}),\]
and hence we find
\[\int_{\Zp^\times}\chi(x)x^k \cdot\zeta_p = L(\chi,1-k),\]
as was to be proved.
\end{proof}

%%======================================================
\subsection{Non-trivial tame conductors}\label{sec:non-trivial tame conductor}

We can go even further. The theorem above deals with the case of `tame conductor 1', in that we have constructed a $p$-adic measure that interpolates all of the $L$-values $L(\chi,1-k)$ for $k > 0$ and cond$(\chi) = p^n$ with $n \geq 0$ (where trivial conductor corresponds to the Riemann zeta function). More generally, we have the following result.

\begin{theorem} \label{thm:nontame}
Let $D > 1$ be any integer coprime to $p$, and let $\eta$ denote a (primitive) Dirichlet character of conductor $D$. There exists a unique measure $\zeta_\eta \in \Lambda(\Zp^\times)$ such that, for all primitive Dirichlet characters $\chi$ with conductor $p^n$, $n \geq 0$, and for all $k > 0$, we have
\[\int_{\Zp^\times} \chi(x) x^k \cdot\zeta_{\eta} = \big(1 - \chi \eta(p) p^{k-1}\big) L(\chi \eta,1-k). \]
\end{theorem}

\begin{remarks} \leavevmode
\begin{enumerate}
\item In this case, we obtain a genuine measure rather than a pseudo-measure. As $L$-functions of non-trivial Dirichlet characters are everywhere holomorphic, there is no need for the smoothing factor involving $a$.
\item Implicit in this theorem is the fact that the relevant Iwasawa algebra is defined over a (fixed) finite extension $L/\Qp$ containing the values of $\eta$.
\end{enumerate}
\end{remarks}

\begin{proof}
Since many of the ideas involved in proving the above theorem are present in the case of trivial tame conductor, the proof of Theorem \ref{thm:nontame} is a good exercise.  As such, we give only the main ideas involved in the proof. 

For $\chi$ of $p$-power conductor, note that the calculation relating $L(f_{\chi,a},s)$ to $L(\chi,s)$ above was entirely classical, in the sense that $p$ did not appear anywhere. We can thus perform a similar calculation for general conductors. As there is no need for the smoothing factor $a$, we consider the function
\begin{equation*}
f_\eta(t) = \frac{-1}{G(\eta^{-1})}\sum_{c \in (\Z/D\Z)^\times} \frac{\eta(c)^{-1}}{e^t\epsilon_{D}^c - 1}.
\end{equation*}
(This scaling by $-1$ also appears in the trivial tame conductor situation, but it is incorporated into $\theta_a$).  In the above definition, the Gauss sum $G(\eta^{-1})$ of a Dirichlet character $\eta$ of conductor $D$ is defined as in Definition \ref{DefGausssum}, replacing the power of $p$ by $D$. Define $F_\eta(T)$ by substituting $T+1$ for $e^t$, i.e.
\begin{equation} \label{EquationFeta}
    F_{\eta}(T) \defeq \frac{-1}{G(\eta^{-1})}\sum_{c \in (\Z/D\Z)^\times} \frac{\eta(c)^{-1}}{(1 + T)\epsilon_{D}^c - 1}.
\end{equation}
Expanding the geometric series, we find
\[F_{\eta}(T) = \frac{-1}{G(\eta^{-1})}\sum_{c\in(\Z/D\Z)^\times} \eta(c)^{-1}\sum_{k\geq 0}\frac{\epsilon_D^{kc}}{(\epsilon_D^c - 1)^{k+1}}T^k.\]
This is an element of $\roi_L\lsem T\rsem $ for some sufficiently large finite extension $L$ of $\Qp$, since the Gauss sum is a $p$-adic unit (indeed, we have $G(\eta)G(\eta^{-1}) = \eta(-1)D$ and $D$ is coprime to $p$) and $\epsilon_D^c -1 \in \roi_L^\times$ (since it has norm dividing $D$). There is therefore a measure $\mu_{\eta} \in \Lambda(\Zp)$, the Iwasawa algebra over $\roi_L$, corresponding to $F_{\eta}$ under the Mahler transform.

\begin{lemma}\label{lem:non-trivial}
We have $L(f_\eta,s) = -\eta(-1)L(\eta,s).$ Hence for $k \geq 0$ we have
\[\int_{\Zp}x^k \cdot\mu_\eta = L(\eta,-k).\]
\end{lemma}
\begin{proof}
The first statement is proved in a similar manner to Lemma \ref{lem:dirichlet integral}. The second is proved by equating $\partial$ with $d/dt$ and using the general theory described in Theorem \ref{thm:l-function}.
\end{proof}
%The measure we desire will be the restriction of $\mu_\eta$ to $\Zp^\times$. 

The following is the analogue of Lemma \ref{LemmaPsiInvariant}.

\begin{lemma} \label{LemmaPsiTwisted}
We have $\psi(F_\eta) = \eta(p)F_\eta$.
\end{lemma}

\begin{proof}
We show first that 
\begin{equation}\label{trace invariant}\frac{1}{p}\sum_{\xi \in \mu_p}\frac{1}{(1+T)\xi\epsilon_D^c -1} = \frac{1}{(1+T)^p \epsilon_D^{pc} - 1}.\end{equation}
Expanding each summand as a geometric series, the left hand side is 
\[\frac{-1}{p}\sum_{\xi\in\mu_p}\sum_{n\geq 0}(1+T)^n\epsilon_D^{nc}\xi^n = -\sum_{n\geq 0}(1+T)^{pn}\epsilon_D^{pcn},\]
and summing the geometric series now gives the right hand side of (\ref{trace invariant}). It follows that
\begin{align*}(\varphi\circ\psi)(F_\eta) &= \frac{-1}{pG(\eta)^{-1}}\sum_{\xi\in\mu_p}\sum_{c\in(\Z/D\Z)^\times}\frac{\eta(c)^{-1}}{(1+T)\xi\epsilon_D^c - 1}\\
&= \frac{-1}{G(\eta^{-1})}\sum_{c\in(\Z/D\Z)^\times}\frac{\eta(c)^{-1}}{(1+T)^p\epsilon_D^{pc} - 1}\\
&= \eta(p)\varphi(F_\eta).
\end{align*}
The result now follows from the injectivity of $\varphi$.
\end{proof}

We can now show the interpolation property at powers of $x$.

\begin{lemma} \label{LemmaInterpolation3}
We have
\[\int_{\Zp^\times}x^k \cdot\mu_{\eta} = \big(1-\eta(p)p^k\big)L(\eta,-k).\]
\end{lemma}

\begin{proof}
By Lemma \ref{LemmaPsiTwisted} we have
\begin{align*}\mathrm{Res}_{\Zp^\times}(\mu_\eta) &= (1-\varphi\circ\psi)(\mu_\eta) = \mu_\eta - \eta(p)\varphi(\mu_\eta),
\end{align*}
and
\[\int_{\Zp}x^k\cdot\varphi(\mu_\eta) = p^k\int_{\Zp}x^k \cdot\mu_\eta.\]
The result now follows from Lemma \ref{lem:non-trivial}.
\end{proof}

Now let $\chi$ be a Dirichlet character of conductor $p^n$ for some $n\geq 0$, and let $\theta \defeq \chi\eta$ the product (a Dirichlet character of conductor $Dp^n$). For such $\theta = \chi \eta$, we define
\begin{equation} \mu_\theta \defeq (\mu_\eta)_\chi. \end{equation}
Using Lemma \ref{lem:mahler chi}, we find easily that:

\begin{lemma} \label{LemmaFtheta}
The Mahler transform of $\mu_\theta$ is
\[F_\theta(T) \defeq \Am_{\mu_\theta}(T) = \frac{-1}{G(\theta^{-1})} \sum_{c\in(\Z/Dp^n\Z)^\times}\frac{\theta(c)^{-1}}{(1+T)\epsilon_{Dp^n}^c - 1}.\]
\end{lemma}

Via a calculation essentially identical to the cases already seen, we can prove
\begin{align*}\int_{\Zp}\chi(x)x^k \cdot\mu_{\eta} &= \int_{\Zp}x^k\cdot\mu_{\theta} = L(\theta,-k),
\end{align*}
and that
\begin{equation}\label{eq:restriction mu theta}
\mathrm{Res}_{\Zp^\times}(\mu_\theta) = \big(1-\theta(p)\varphi\big) \mu_\theta.
\end{equation}
(Here, note that if $\chi$ is non-trivial, then $\mu_\theta$ is already supported on $\Zp^\times$; but this is consistent, as $\theta(p) = 0$ in this case).
Combining the above we find
\[\int_{\Zp^\times}\chi(x)x^k\cdot\mu_{\eta} = \big(1-\theta(p)p^k\big)L(\theta,-k).\]

Finally, to complete the proof of Theorem \ref{thm:nontame} and to ensure compatibility with the construction of $\zeta_p$, we introduce a shift by 1. The following is directly analogous to the construction of $\zeta_p$; note again that $\zeta_\eta$ is truly a measure, not a pseudo-measure.
\begin{definition}
	Define $\zeta_\eta \defeq x^{-1} \mathrm{Res}_{\Zp^\times}(\mu_\eta).$
\end{definition}

We see that 
\[
	\int_{\Zp^\times}\chi(x)x^k\cdot\zeta_{\eta} = \big(1-\theta(p)p^{k-1}\big)L(\theta,1-k).
\]
which completes the proof of Theorem \ref{thm:nontame}.
\end{proof}

%%======================================================================
%%======================================================================

\subsection{Analytic functions on $\Zp$ via the Mellin transform}\label{mellin}
The reader should hopefully now be convinced that measures are a natural language with which to discuss $p$-adic $L$-functions. In this subsection, we use this (more powerful) language to answer the question we originally posed in the introduction: namely, we define analytic functions on $\Zp$ interpolating the  values $\zeta(1-k)$ for $k > 0$. In passing from measures to analytic functions on $\Zp$, we lose the clean interpolation statements. In particular, there is no \emph{single} analytic function on $\Zp$ interpolating the values $\zeta(1-k)$ for all $k>0$, but rather $p-1$ different `branches' of the Kubota--Leopoldt $p$-adic $L$-function on $\Zp$, each interpolating a different range.

The reason we cannot define a single $p$-adic $L$-function on $\Zp$ is down to the following technicality. We'd \emph{like} to be able to define ``$\zeta_p(s) = \int_{\Zp^\times}x^{-s}\cdot\zeta_p$'' for $s\in\Zp$. The natural way to define the exponential $x \mapsto x^s$ is as
\[x^s = \mathrm{exp}(s\cdot\mathrm{log}(x)),\]
but unfortunately in the $p$-adic world the exponential map does not converge on all of $\Zp$, so this is not well-defined for general $x\in\Zp^\times$. Instead:

\begin{lemma} \label{LemmapadicLogarithm}
The $p$-adic exponential map converges on $p\Zp$. Hence, for any $s \in \Zp$, the function $1+p\Zp \rightarrow \Zp$ given by $x \mapsto x^s \defeq \mathrm{\exp}(s\cdot\mathrm{log}(x))$ is well-defined.
\end{lemma}
\begin{proof}
This is a standard result in the theory of local fields; see e.g.\ \cite[\S12]{cassels}.
\end{proof}

\begin{definition}
Recall that we assume $p$ to be odd and that we have a decomposition $\Zp^\times \cong \mu_{p  - 1} \times (1+p\Zp)$. Let
\begin{align*}\omega : \Zp^\times &\longrightarrow \mu_{p - 1},\\
\langle\cdot\rangle : \Zp^\times &\longrightarrow 1+p\Zp,
\end{align*}
where $\omega(x) \defeq $ Teichm\"uller lift of the reduction modulo $p$ of $x$ and $\langle x \rangle \defeq \omega^{-1}(x) x$ denote the projections to the first and second factors respectively. If $x\in\Zp^\times$, then we can write $x = \omega(x)\langle x\rangle$.
\end{definition}

By Lemma \ref{LemmapadicLogarithm}, the function $x \mapsto \langle x\rangle^s$ is well-defined for any $s \in \zp$. For each $i = 1,..,p-1$ we can define an injection
\begin{align*}\Zp &\longhookrightarrow \mathrm{Hom}_{\mathrm{cts}}(\Zp^\times,\Cp^\times)\\
s &\longmapsto \big[x \mapsto \omega(x)^i \langle x\rangle^s\big],
\end{align*}
and hence we can define  a meromorphic function as follows.
\begin{definition} \label{DefpadicZeta}
 We define the $i$-th branch of the $p$-adic zeta function as
\begin{align*}
\zeta_{p,i} : \Zp &\longrightarrow \Cp\\
s &\longmapsto \int_{\Zp^\times}\omega(x)^{i}\langle x\rangle^{1-s} \cdot\zeta_p.
\end{align*}
\end{definition}
This function \emph{does not} interpolate as wide a range of values as the measure $\zeta_p$, since the character $x^k$ can be written in the form $\omega(x)^i\langle x\rangle^k$ if and only if $k \equiv i \newmod{p-1}$, and in this case $x^k$ is the value of $\omega(x)^{i}\langle x \rangle^{1-s}$ at the value $s = 1-k$. Then we have the following result.

\begin{theorem} \label{thm:kubota leopoldt analytic}
For all $k\geq 1$ with $k \equiv i \newmod{p-1}$, we have
\[\zeta_{p,i}(1-k) = (1-p^{k-1})\zeta(1-k).\]
\end{theorem}
%Similarly, when $p=2$, we have functions $\zeta_{2,i}(s)$ for $i = 0,1$, and a similar interpolation property holds for $k \equiv i \newmod{2}$.\\

Note that the above theorem implies that $\zeta_{p, i}$ is identically zero whenever $i$ is odd (as by Corollary \ref{cor:rational} the value $\zeta(1-k)$ is zero for every odd positive integer $k \geq 1$). More generally, we can twist by Dirichlet characters as we have done before.

\begin{definition} \label{DefinitionLp}
Let $\theta = \chi\eta$ be a Dirichlet character, where $\eta$ has conductor $D$ prime to $p$ and $\chi$ has conductor $p^n$ for $n\geq 0$. Define
\[L_p(\theta,s) \defeq \int_{\Zp^\times}\chi(x)\langle x\rangle^{1-s} \cdot \zeta_\eta, \hspace{12pt} s\in\Zp.\]
\end{definition}

\begin{remark}
\item An equivalent definition is  
\begin{equation}\label{eq:alternative}
	L_p(\theta,s) = \int_{\Zp^\times} \chi\omega^{-1}(x)\langle x\rangle^{-s}\cdot \mu_\eta = \int_{\Zp^\times} \chi\omega^{s-1}(x)x^{-s}\cdot \mu_\eta.
\end{equation}
Note here we use the measure $\mu_\eta$, rather than the (shifted-by-1) analogue $\zeta_\eta$. 
In \cite{Washington2}, the analytic functions $L_p(\theta,s)$ are constructed directly without using measures, and the more direct approach differs from the one obtained using our measure-theoretic approach by precisely this factor of $\omega$. This twist by 1 also appears naturally when we study the Iwasawa Main Conjecture.
\end{remark}

\begin{theorem} \label{TheoremLeopoldtAnalyticTwist}
 For all $k\geq 1$, we have
\[L_p(\theta,1-k) = \big(1 - \theta \omega^{-k}(p)p^{k-1}\big) L(\theta \omega^{-k},1-k).\]
\end{theorem}
\begin{proof}
We use the description of \eqref{eq:alternative}. From the definitions, we have \[
    \chi\omega^{-1}(x)\langle x\rangle^{k-1} = \chi\omega^{-k}(x)\cdot\omega^{k-1}(x)\langle x\rangle^{k-1} = \chi\omega^{-k}(x)x^{k-1},
\]
so that
\begin{align*}
\int_{\Zp^\times} \chi(x) \langle x \rangle^{k-1} \cdot \mu_\eta &= \int_{\Zp^\times} \chi \omega^{-k}(x) x^{k-1} \cdot\mu_\eta\\
&= \big( 1 - \theta \omega^{-k}(p) p^{k-1}\big) L(\theta \omega^{-k}, 1-k),
\end{align*}
as required. The last equality here is Lemma \ref{LemmaInterpolation3}.
\end{proof}

\begin{remark} Directly from the definitions, we have $\zeta_{p,i}(s) = L_p(\omega^{i},s)$. Hence for arbitrary $k > 0$, Theorem \ref{TheoremLeopoldtAnalyticTwist} gives
\[\zeta_{p,i}(1-k) = (1-\omega^{i-k}(p)p^{k-1})L(\omega^{i-k},1-k).\]
Of course, $\omega^{i-k}$ is just the trivial character when $i \equiv k \newmod{p-1}$, so we recover Theorem \ref{thm:kubota leopoldt analytic} from Theorem \ref{TheoremLeopoldtAnalyticTwist}.
\end{remark}

In general, for any measure $\mu$ on $\Zp^\times$ one can define
\[\mathrm{Mel}_{\mu,i}(s) = \int_{\Zp^\times} \omega(x)^i \langle x\rangle^s \cdot\mu,\]
the \emph{Mellin transform} of $\mu$ at $i$. We have then $\zeta_{p,i}(s) = \mathrm{Mel}_{\zeta_p,i}(1-s)$. This transform gives a way to pass from $p$-adic measures on $\Zp$ to analytic functions on $\Zp$.

\begin{remark}
	The results of this section are simply a more concrete version of Remark
\ref{rem:weight space rigid analytic}. There we described how a measure (resp.\ pseudo-measure) $\mu$ on $\Zp^\times$ gives rise to a rigid analytic (resp. rigid meromorphic) function $F_\mu$ on weight space $\cW(\Cp)$. The function $\zeta_{p,i}$ above corresponds to the restriction of $F_\mu$ to the open ball $U_{\omega^i} \subset \cW(\Zp)$ (again explaining why we need $p-1$ such functions, corresponding to the $p-1$ disjoint open balls).
 \end{remark}

%%======================================================
\section{The values at $s=1$}\label{sec:s=1}

In the following we give an example of further remarkable links between the classical and $p$-adic zeta functions. Let $\theta$ be a non-trivial Dirichlet character, which as usual we write in the form $\chi\eta$, where $\chi$ has conductor $p^n$ and $\eta$ has conductor $D$ prime to $p$. By Theorem \ref{thm:nontame}, for any integer $k > 0$, we have
\[ \int_{\Zp^\times}\chi(x)x^k \cdot \zeta_{\eta} = L(\theta,1-k).\]
 It's natural to ask what happens outside the range of interpolation $k >0$. In particular, what happens when we take $k = 0$? Since this is \emph{outside} the range of interpolation, this value may \emph{a priori} have nothing to do with classical $L$-values. Indeed, the classical value $L(\theta,1)$ is transcendental\footnote{This follows from Baker's theorem and Theorem \ref{s=1 theorem}, part (i).}, so one cannot see it as a $p$-adic number in a natural way. However, even though we cannot directly equate the two values, it turns out that there is a formula for the $p$-adic $L$-function at $s=1$ which is strikingly similar to its classical analogue.

\begin{theorem}\label{s=1 theorem} Let $\theta$ be a non-trivial Dirichlet character of conductor $N$, and let $\varepsilon_N$ denote a primitive $N$th root of unity. Then:
\begin{itemize}
\item[(i)] (Classical value at $s=1$). We have
\[L(\theta,1) = -\frac{1}{G(\theta^{-1})} \sum_{c \in (\Z/N\Z)^\times} \theta^{-1}(c) \log\big( 1-\varepsilon_N^c \big).\]
\item[(ii)] ($p$-adic value at $s=1$). We have
\[ L_p(\theta,1) = -\big( 1 - \theta(p) p^{-1} \big) \frac{1}{G(\theta^{-1})} 		\sum_{c \in (\Z/N\Z)^\times} \theta^{-1}(c) \log_p\big(1-\varepsilon_N^c).\]
\end{itemize}
\end{theorem}

\begin{remark}\label{rem:beilinson 1}
Part $(ii)$ of Theorem \ref{s=1 theorem} is due to Leopoldt. Values in the range of interpolation, where the link to classical $L$-values is explicit, are often called \emph{critical} values. Values outside this range, such as those studied in Theorem \ref{s=1 theorem}, are called \emph{non-critical} values. The above result is an instance of the $p$-adic Beilinson or Perrin-Riou conjectures, which give arithmetic descriptions of non-critical special values of $p$-adic $L$-functions. We refer the interested reader to \cite{PR} (or Remark \ref{rem:gross etc} below) for more details on this. 
\end{remark}

If $\theta$ is an odd character, both sides of the $p$-adic formula vanish. In any case, the formulae are identical up to replacing $\log$ with its $p$-adic avatar and, as usual, deleting the Euler factor at $p$. This result can be used to prove a $p$-adic analogue of the class number formula. 

\subsection{The complex value at $s = 1$}

For completeness, we prove the complex case of Theorem \ref{s=1 theorem}, following \cite[Thm.\ 4.9]{Washington2}.

\begin{proof}\emph{(Theorem \ref{s=1 theorem}(i)).}
 Write
 \[ L(\theta,s) = \sum_{a \in (\Z / N \Z)^\times} \theta(a) \sum_{n \equiv a \newmod{D}} n^{-s}. \]
 Using the fact that
 \[ \frac{1}{N} \sum_{c \in (\Z / N \Z)} \epsilon_N^{(a - n)c} = \left \{
   \begin{array}{c l}
      0 &: \text{ if $n \not \equiv a$ mod $N$} \\
      1 &: \text{ if $n \equiv a$ mod $N$,}
   \end{array}
   \right. \]
   we show that the above formula equals
   \begin{align}
   \sum_{a \in (\Z / N \Z)^\times} \theta(a) \frac{1}{N} \sum_{n \geq 1} \sum_{c \in (\Z / N \Z)} \varepsilon_N^{(a - n) c} n^{-s} &= \frac{1}{N} \sum_{c \in (\Z / N \Z)} \left( \sum_{a \in (\Z / N \Z)^\times} \theta(a) \varepsilon_N^{ ac} \right) \sum_{n \geq 1} \frac{\epsilon_N^{- n c}}{n^{s}} \notag\\
   &= \frac{G(\theta)}{N} \sum_{c \in (\Z / N \Z)} \theta^{-1}(c) \sum_{n \geq 1} \frac{\varepsilon_N^{- n c}}{n^{s}} \notag\\
   &= \frac{\theta(-1)}{G(\theta^{-1})} \sum_{c \in (\Z / N \Z)^\times} \theta^{-1}(c) \sum_{n \geq 1} \frac{\varepsilon_N^{- n c}}{n^{s}} \notag\\
   &= \frac{1}{G(\theta^{-1})} \sum_{c \in (\Z / N \Z)^\times} \theta^{-1}(c) \sum_{n \geq 1} \frac{\varepsilon_N^{ n c}}{n^{s}}. \label{eq:classical 6.1}
   \end{align}
Here the penultimate equality uses the standard identity $G(\theta)G(\theta^{-1}) = \theta(-1) \mathrm{cond}(\theta)$ of Gauss sums (cf.\ Remark \ref{RemPropertiesGaussSum}(i)) and that $\theta^{-1}(c) = 0$ if $(c, N) \neq 1$, and the last equality follows from the change of variables $c \mapsto -c$. 

Finally we evaluate this expression at $s = 1$. As $\theta$ is not trivial we have $N>1$, so $\varepsilon_N^c \neq 1$ for any $c \in (\Z/N\Z)^\times$. Thus we may consider the Taylor series expansion 
\[	
	-\log(1-\varepsilon_N^{c}) = \sum_{n\geq 1}\varepsilon_N^{nc}n^{-1}.
\] Substituting this into \eqref{eq:classical 6.1}, we see the series converges at $s=1$ to the required result.
\end{proof}

\begin{rem} \label{s=1}
 We can further refine this expression depending on the parity of $\theta$. If $\theta$ is even then $\theta^{-1}(c)\log(1-\varepsilon_N^{c}) + \theta^{-1}(-c)\log(1-\varepsilon_N^{-c}) = 2\theta^{-1}(c)\log|1-\varepsilon_N^c|$, so, rearranging
 \[ L(\theta,1) = - \frac{1}{G(\theta^{-1})} \sum_{c \in (\Z / N \Z)^\times} \theta^{-1}(c) \log |1 - \varepsilon^{c} |.  \]
If $\theta$ is odd, we can use the functional equation to obtain
 \[ L(\theta,1) = -i \pi \frac{1}{G(\theta^{-1})} B_{1, \theta^{-1}}, \]
 where $B_{k,\theta^{-1}}$ denotes the $k$-th twisted Bernoulli number (see \cite[Chapter 4]{Washington2}).
\end{rem}

\subsection{The $p$-adic value at $s = 1$}\label{sec:p-adic s=1}

We now compute
\begin{align} \label{EqLp1} L_p(\theta, 1) &\defeq \int_{\zpe} \chi(x) x^{-1} \cdot \mu_\eta = \int_{\zpe} x^{-1} \cdot \mu_\theta.
\end{align} 

A tempting argument to study this goes as follows. Suppose $k$ is divisible by $p-1$. We know from Equation \eqref{eq:alternative} that 
\[
L_p(\theta,1-k) = \int_{\zpe}\chi(x)x^{k-1}\cdot \mu_\eta,
\]
noting that $\omega^{(1-k)-1} = \omega^{-k} = 1$ as $\omega$ has order $p-1$. Now if $k>0$, we showed above that
\[
\int_{\zpe}\chi(x)x^{k-1} \cdot \mu_\eta = \sA_{\mathrm{Res}_{\zpe}(x^{k-1}\mu_\theta)}(0) = (1-\theta(p)  p^{k-1})\big(\partial^{k-1}F_\theta\big)(0).
\]
We want to compute this for $k=0$. Identically we could try to argue that
\begin{equation}\label{eq:false argument}
L_p(\theta, 1) = \sA_{\mathrm{Res}_{\zpe}(x^{-1}\mu_\theta)}(0)= (1-\theta(p)p^{-1})\big(\partial^{-1} F_\theta\big)(0).
\end{equation}
Then, recalling that by Lemma \ref{LemmaFtheta} we have
\[ F_\theta(T) = -\frac{1}{G(\theta^{-1})} \sum_{c\in(\Z/N\Z)^\times}\frac{\theta(c)^{-1}}{(1+T)\epsilon_{N}^c - 1}, \]
we observe that, if we define  
\[
\widetilde{F}_\theta(T) = - \frac{1}{G(\theta^{-1})} \sum_{c \in (\Z / N \Z)^\times} \theta^{-1}(c) \log \big( (1 + T) \epsilon_N^c - 1 \big),
\]
then formally $\partial \widetilde{F}_\theta = F_\theta$; we will show this in the proof of Lemma \ref{lem:mu theta'} below. In particular, $\widetilde{F}_\theta$ is a good candidate for $\partial^{-1}F_\theta$. Plugging in $T=0$ and combining with \eqref{eq:false argument} would give the claimed value of $L_p(\theta,1)$. 

In order to make this reasoning rigorous, one needs to deal with the fact that $x^{-1}$ is not a well-defined operation on measures on $\Zp$, rendering $x^{-1}\mu_\theta$ ill-defined. On power series, this is captured by the indeterminacy in defining $\partial^{-1}$. In particular, $\widetilde{F}_\theta(T)$ is \emph{not} necessarily a bounded power series, so under the Mahler correspondence, does not correspond to a $p$-adic measure. However we do have the following. Recall that $\mathscr{R}^+$ (from Remark \ref{rem:analytic functions}) denotes the space of power series $\sum a_nT^n$ such that $|a_n|r^n \to 0$ for any $0\leq r <1$.

\begin{lemma}\label{lem:bounded power series}
The power series $\widetilde{F}_\theta(T)$ is an element of $\mathscr{R}^+$.
\end{lemma}

\begin{proof}
We can write
\begin{align*} \log((1 + T) \epsilon_N^c - 1) &= \log_p(\epsilon_N^c - 1) + \log\left(1 + \frac{\epsilon_N^c T}{\epsilon_N^c - 1}\right)\\
=& \log_p(\epsilon_N^c - 1) +  \sum_{n = 1}^{\infty} \frac{(-1)^{n-1}}{n} \cdot\frac{\epsilon_N^{cn}}{(\epsilon_N^c - 1)^n} T^n. 
\end{align*}
We now consider two cases. If $(N, p) = 1$, we know that $(\epsilon_N^c - 1)$ is a $p$-adic unit; then the coefficient of $T^n$ has $p$-adic valuation bounded below by $-v_p(n)$. This means the coefficients in $\widetilde{F}_\theta(T)$ have logarithmic growth, and in particular $\widetilde{F}_\theta(T) \in \mathscr{R}^+$.  More generally, suppose $N = Dp^n$ with $(D, p) = 1$. We write $\theta = \eta\chi$, with $\eta$ and $\chi$ characters of conductors $D$ and $p^n$ respectively. Then, as in Lemma \ref{lem:mahler chi} or Lemma \ref{LemmaFtheta}, we have
\[ \widetilde{F}_\theta(T) = (\widetilde{F}_\eta)_\chi(T) = - \frac{1}{G(\chi^{-1})} \sum_{c \in (\Z / p^n \Z)^\times} \chi^{-1}(c) \widetilde{F}_\eta((1 + T) \epsilon_{p^n}^c - 1). \]
As $\widetilde{F}_\eta(T) \in \mathscr{R}^+$, the same holds for $\widetilde{F}_\theta(T)$.
\end{proof}

By Theorem \ref{thm:mahler la}, $\widetilde{F}_\theta(T)$ is the Mahler transform of a locally analytic distribution $\widetilde{\mu}_\theta$ on $\Zp$. We now relate this distribution to $x^{-1}\mathrm{Res}_{\zpe}(\mu_\theta)$, as appeared in \eqref{EqLp1}.

\begin{lemma}\label{lem:mu theta'}
We have 
\[x \widetilde{\mu}_\theta = \mu_\theta.\]
In particular,
\[
    \mathrm{Res}_{\Zp^\times}(\widetilde{\mu}_\theta) = x^{-1}\mathrm{Res}_{\zpe}(\mu_\theta).
\]
\end{lemma}

\begin{proof}
The first equality can be checked on Mahler transforms. By Lemma \ref{LemmaMultiplicationbyx}, this means showing 
\begin{equation}\label{eq:mahler primitive}
\partial \widetilde{F}_\theta(T) = (1+T)\frac{d}{dT}\widetilde{F}_\theta(T) = \sA_{\mu_\theta}(T).
\end{equation}
By Lemma \ref{LemmaFtheta}, we have
\[\mathscr{A}_{\mu_\theta}(T) = F_\theta(T) = \frac{-1}{G(\theta^{-1})} \sum_{c\in(\Z/N\Z)^\times}\frac{\theta(c)^{-1}}{(1+T)\epsilon_{N}^c - 1}. \]
Then \eqref{eq:mahler primitive} follows immediately from the formula
\[ \partial \log \big( (1 + T) \epsilon_D^c - 1 \big) = \frac{(1 + T) \epsilon_D^c}{(1 + T) \epsilon_D^c - 1} = 1 + \frac{1}{(1 + T) \epsilon_D^c - 1} \] 
and the fact that $\sum_{c \in (\Z / D \Z)^\times} \theta^{-1}(c) = 0$.

To see the second equality, note that as measures on $\Zp^\times$, we visibly have $x\mathrm{Res}_{\Zp^\times}(\widetilde{\mu}_\theta) = \mathrm{Res}_{\Zp^\times}(\mu_\theta)$. It follows that $\mathrm{Res}_{\Zp^\times}(\widetilde{\mu}_\theta)  =x^{-1}\mathrm{Res}_{\Zp^\times}(\mu_\theta)$ as `multiplication by $x$' is an invertible operator on measures/distributions on $\zpe$.
\end{proof}

From the above, we now know that
\begin{equation}\label{eq:Lptheta 1,2}
L_p(\theta,1) = \sA_{x^{-1}\mathrm{Res}_{\zpe}(\mu_\theta)}(0) = \sA_{\mathrm{Res}_{\zpe}(\widetilde{\mu}_\theta)}(0) = \Big((1-\varphi\circ\psi)\widetilde{F}_\theta\Big)(0),
\end{equation}
where the last equality follows from the formula for the restriction of a distribution to $\zpe$ (cf.\ Equation \eqref{res to Zp}, and Remark \ref{rem:distribution toolbox}). We are now ready to prove Theorem \ref{s=1 theorem}(ii).

\begin{proof}\emph{(Theorem \ref{s=1 theorem}(ii))}.
We compute the right-hand side of \eqref{eq:Lptheta 1,2}. Recall from \S\ref{SubSectionphipsi} that $\varphi\circ\psi(\widetilde{F}_\theta)$ is the Mahler transform of $\mathrm{Res}_{p\Zp}(\widetilde{\mu}_\theta)$. Recall that $N = Dp^n$, and $\theta = \eta\chi$, where $\eta$ has prime-to-$p$ conductor $D$, and $\chi$ has conductor $p^n$. To compute this we break into two cases.

\begin{enumerate}
\item First assume that $n > 1$, so that $\chi \neq 1$; then, as $\chi|_{p\Zp} = 0$, we see $\widetilde{\mu}_\theta = (\widetilde{\mu}_\eta)_\chi$ is automatically supported on $\zpe$ by \eqref{eq:twist by chi}. In particular, $\mathrm{Res}_{p\Zp}(\widetilde{\mu}_\theta) = 0$, and thus $\varphi\circ\psi(\widetilde{F}_\theta) = 0$. In particular, in this case
\[
L_p(\theta,1) = \widetilde{F}_\theta(0).
\]
It is convenient to write this in the form $L_p(\theta,1) = (1-\theta(p)p^{-1})\widetilde{F}_\theta(0)$ using that $\theta(p) = 0$.

\item Now assume $n = 0$, so $N = D$ is coprime to $p$ and hence $\theta = \eta$.  By \eqref{Eqphipsi}, we know
\begin{align*}
\varphi\circ\psi(\widetilde{F}_\theta) &= \frac{1}{p}\sum_{\xi \in \mu_p} \widetilde{F}_\theta((1+T)\xi - 1)\\
&= \frac{-1}{G(\theta^{-1})} \cdot \frac{1}{p}\sum_{c\in (\Z/N\Z)^\times}\theta^{-1}(c)\sum_{\xi \in \mu_p} \log_p ((1+T)\xi\varepsilon_N^c - 1).
\end{align*}
Evaluating at $T = 0$, we get
\begin{align*}
  \varphi\circ\psi(\widetilde{F}_\theta)(0) &=   \frac{-1}{G(\theta^{-1})} \cdot \frac{1}{p}\sum_{c\in (\Z/N\Z)^\times}\theta^{-1}(c)\sum_{\xi \in \mu_p} \log_p(\xi\varepsilon_N^c - 1)\\
  &= \frac{-1}{G(\theta^{-1})} \cdot \frac{1}{p}\sum_{c\in (\Z/N\Z)^\times}\theta^{-1}(c)\log_p(\varepsilon_N^{pc} - 1) \\
  &=\frac{-1}{G(\theta^{-1})} \cdot \frac{\theta(p)}{p}\sum_{c'\in (\Z/N\Z)^\times}\theta^{-1}(c')\log_p(\varepsilon_N^{c'} - 1)\\
 &= \frac{\theta(p)}{p} \widetilde{F}_\theta(0).
\end{align*}
Here, we used that, since $p\nmid N$, the assignment $c \mapsto c' = pc$ defines an automorphism of $(\Z / N\Z)^\times$. Hence in this case we also find that 
\begin{align*}
    L_p(\theta,1) &= \Big((1-\varphi\circ\psi)\widetilde{F}_\theta\Big)(0) = \Big(1 - \theta(p)p^{-1}\Big) \widetilde{F}_\theta(0).
\end{align*}
\end{enumerate}

To complete the proof, we simply evaluate the expression $\widetilde{F}_\theta(0)$ (and use that $\log_p(x) = \log_p(-x)$ for $x \in \Cp^\times$) to find, for all $N$, that
\begin{align*}
L_p(\theta,1)    &= -\Big(1 - \theta(p)p^{-1}\Big)\frac{1}{G(\theta^{-1})} 		\sum_{c \in (\Z/N\Z)^\times} \theta^{-1}(c) \log_p\big(1- \epsilon_N^c),
\end{align*}
as claimed.
\end{proof}

\begin{remark}
Theorem \ref{s=1 theorem} has been generalised by Coleman in \cite{Colemandilog} for every positive integer value $s = k\geq 1$. More precisely, for $s, z \in \C$, let $\mathrm{Li}_s(z) \defeq \sum_{n \geq 1} \frac{z^n}{n^s}$ be the polylogarithm fuction; recall that it admits a unique analytic continuation to $\C \backslash\{ z \in \R : z \geq 1\}$. In particular, one sees that $\mathrm{Li}_s(1) = \zeta(s)$, and $\mathrm{Li}_1(z) = - \log(1 - z)$. Coleman constructed $p$-adic analogues $\mathrm{Li}_{k, p}(z)$, which are locally analytic functions on $\C_p \backslash \{1\}$, and showed:

\begin{theorem} [\cite{Colemandilog}] \label{s=k theorem} Let $\theta$ be a non-trivial Dirichlet character of conductor $N$, let $k \geq 1$ be an integer, and let $\varepsilon_N$ denote a primitive $N$th root of unity. Then:
\begin{itemize}
\item[(i)] (Classical value at $s=k$). We have
\[L(\theta,k) = \frac{1}{G(\theta^{-1})} \sum_{c \in (\Z/N\Z)^\times} \theta^{-1}(c) \mathrm{Li}_{k}(\varepsilon_N^c).\]
\item[(ii)] ($p$-adic value at $s=k$). We have
\[ L_p(\theta,k) = \big( 1 - \theta(p) p^{-k} \big) \frac{1}{G(\theta^{-1})} 		\sum_{c \in (\Z/N\Z)^\times} \theta^{-1}(c) \mathrm{Li}_{k, p}(\varepsilon_N^c).\]
\end{itemize}
\end{theorem}
\end{remark}

\begin{remark}\label{rem:gross etc}
We highlighted in Remark \ref{rem:beilinson 1} that Theorem \ref{s=1 theorem}(ii) is an instance of Perrin-Riou's $p$-adic Beilinson conjectures. More precisely, her conjectures describe special, non-critical values of $p$-adic $L$-functions of motives in terms of arithmetic data. Specialised to the case of the Kubota--Leopoldt $p$-adic $L$-function (see \cite[\S 4.3.3]{PerrinRiou}), this gives formulas for the values of $L_p(\theta, k)$ in terms of $p$-adic regulators of the cyclotomic units, special elements that we will see later; the right-hand sides of Theorems \ref{s=1 theorem}(ii) and \ref{s=k theorem}(ii) can be reinterpreted in such terms, justifying Remark \ref{rem:beilinson 1}. We refer to \cite{PR} for more details, and \cite{ColmezFonctL} for an excellent survey on all of this. 

We also mention the pioneering work of Gross \cite{GrossyntomicI}, based on Coleman's work, as well as \cite{MTsyntomic}, for another approach expressing values of Dirichlet $p$-adic $L$-functions at odd positive integers in terms of syntomic regulators and $K$-theory. 
\end{remark}

%%=========================================================================

\section{The residue of $\zeta_p$ at $s = 1$}\label{sec:residue}

In the previous section, we described the value of the Dirichlet $p$-adic $L$-functions $L_p(\theta,s)$ at $s=1$ for a non-trivial Dirichlet character $\theta$.  We now turn to the value at $s=1$ of the \emph{untwisted} $p$-adic zeta function, that is, the analogue when $\theta$ is trivial. Recall the Riemann zeta function has a simple pole at $s = 1$ with residue $1$ and that, in the $p$-adic world, we have defined the $p$-adic zeta function as a pseudo-measure (rather than a measure) which implies, as explained in Remark \ref{rem:weight space rigid analytic}, that there might be a potential pole at the trivial character. We now show that there is indeed a simple pole here, and calculate its residue.

As with $L_p(\theta,s)$, it is convenient to use the language of analytic functions $\zeta_{p,i}$ from Definition \ref{DefpadicZeta}. The behaviour of $\zeta_p$ at the trivial character is captured by the behaviour of $\zeta_{p,p-1}(s)$ at $s=1$. The main result of this section is the following.

\begin{theorem}\label{thm:residue}
Let $i \in \{1, 2, \ldots, p-1 \}$. The following assertions hold:
\begin{itemize}
\item[(i)] If $i \neq p-1$, then $\zeta_{p,i}$ is analytic at $s=1$.
\item[(ii)] The function $\zeta_{p,p-1}$ has a simple pole at $s=1$ with residue $(1 - p^{-1})$. 
\end{itemize}
\end{theorem}

The proof will occupy the rest of this section.

For any $i \in \{1,2, \ldots, p-1\}$, by Definition \ref{DefpadicZeta} we have
\begin{align} 
\zeta_{p, i}(s) = \int_{Z_p^\times} \omega(x)^i \langle x \rangle^{1 - s} \cdot \zeta_p. \notag
\end{align}
Now, from Definition \ref{DefZetap}, we have 
\[
    \zeta_p = \frac{x^{-1} \mathrm{Res}_{\Z_p^\times}(\mu_a)}{\theta_a} = \frac{x^{-1} \mathrm{Res}_{\Z_p^\times}(\mu_a)}{[a]-[1]},
\]
where $a$ is any topological generator of $\Z_p^\times$. Thus by the expression \eqref{eq:integrate pseudo-measure} for evaluating pseudo-measures, we find
\begin{align}
\zeta_{p,i}(s)   &= \frac{\int_{\Z_p^\times} \omega(x)^{i} \langle x\rangle ^{1-s} x^{-1} \cdot \mu_a}{\omega(a)^{i}\langle a\rangle^{1-s} - 1}\notag\\
&=\frac{\int_{\Z_p^\times} \omega(x)^{s + i - 1} x^{-s} \cdot \mu_a}{\omega(a)^{i}\langle a\rangle^{1-s} - 1},\label{Eqtmp2}
%{\theta_a(\omega(x)^i \langle x \rangle^{1-s})},\label{Eqtmp2}
\end{align}
where we have used $\langle x\rangle  = \omega^{-1}(x)x$ in the second equality. Let \[
    g_{a,i}(s) \defeq \omega(a)^i\langle a \rangle^{1-s} - 1
\]
be the denominator of \eqref{Eqtmp2}.

\begin{lemma}\label{lem:g p-1} The following assertions hold.
\begin{itemize}
\item[(i)] If $i \neq p-1$, then $g_{a,i}(1) \neq 0$. In particular, Theorem \ref{thm:residue}(i) holds.
\item[(ii)] We have $g_{a,p-1}(1) = 0$, and 
\[
\lim_{s\to1} (s-1)^{-1}g_{a,p-1}(s) = -\log_p(a).
\]
\end{itemize}
\end{lemma}

\begin{proof}
Since $a$ is a topological generator of $\Z_p^\times$, we see $\omega(a)$ is a primitive $(p-1)$-th root of unity. Hence the denominator $\omega(a)^{i}\langle a\rangle^{1-s} - 1$ of equation \eqref{Eqtmp2} vanishes at $s = 1$ if and only $i = p-1$. This already implies Theorem \ref{thm:residue}(i), as the expression \eqref{Eqtmp2} does not have a pole at $s=1$.

If $i = p-1$, we know $\omega(a)^{i} = 1$, so $g_{a,p-1}(1) = 0$. Moreover
\begin{align}
g_{a,p-1}(s) = \omega(a)^{p-1} \langle a \rangle^{1 - s} - 1 \notag &= \langle a \rangle^{1 - s} - 1 \notag\\
&= \sum_{n \geq 1} {{1 - s} \choose n} (\langle a \rangle - 1)^n \notag\\
&= (1 - s) \sum_{n \geq 1} \frac{1}{n} {{-s} \choose {n-1}} (\langle a \rangle - 1)^n, \label{eq:g p-1}
\end{align}
where in the last equality we have used the identity ${{1 - s} \choose n} = \frac{1-s}{n} {{-s} \choose {n-1}}$ for $n \geq 1$. The sum in \eqref{eq:g p-1} evaluates at $s = 1$ to
\begin{align*}
\sum_{n \geq 1} \frac{1}{n} {{-1} \choose {n-1}} (\langle a \rangle - 1)^{n} = \sum_{n \geq 1} \frac{(-1)^{n-1}}{n} (\langle a \rangle - 1)^n
&= \log_p(\langle a \rangle) = \log_p(a),
\end{align*}
where we have used ${ {-1} \choose {n - 1}} = (-1)^{n-1}$ (direct from Definition \ref{def:binomial polynomials}). 
We deduce that
\[ \lim_{s\to 1}[(s-1)^{-1}g_{a,s-1}(s)] = -\log_p(a), \] as claimed.
\end{proof}

Combining Lemma \ref{lem:g p-1}(ii) with \eqref{Eqtmp2}, we deduce
\begin{equation}\label{eq:zeta p residue}
\lim_{s\to1} (s-1)\zeta_{p,p-1}(s)  = -\frac{\int_{\Zp^\times} x^{-1} \cdot \mu_a}{\log_p(a)}.
\end{equation}

We calculate the numerator in this expression via similar methods to those used in \S\ref{sec:p-adic s=1}. Recall that $F_a(T) = \frac{1}{T} - \frac{a}{(1 + T)^a - 1}$ is the Mahler transform of $\mu_a$; we find a power series $\widetilde{F}_a(T)$ such that $\partial \widetilde{F}_a(T) = F_a(T)$, where $\partial = (1 + T) \frac{d}{dT}$. Then, via Lemma \ref{lem:mu theta'} and directly analogously to \eqref{eq:Lptheta 1,2}, we find
\begin{equation}\label{eq:zeta p residue 2}
\int_{\Zp^\times}x^{-1}\mu_a = \Big((1 - \varphi\circ\psi)\widetilde{F}_a\Big)(0).
\end{equation}

To this end, let
\[
\widetilde{F}_a(T) \defeq \log\left(\frac{T}{1+T} \cdot \frac{(1+T)^a}{(1+T)^a-1}\right).% = \log\left(\frac{T}{\sigma_a(T)}\cdot\frac{\sigma_a(1+T)}{1+T}\right).
\]

\begin{lemma}
Formally, we have
\[
\partial \widetilde{F}_a(T) = F_a(T).
\]
\end{lemma}

\begin{proof}
We let the reader check that
\[ \partial \log\left(\frac{(1+T)^a-1}{(1+T)^a}\right) = \frac{a}{(1+T)^a-1}. \]
In particular, taking $a=1$, we also get 
\[
\partial\log\left(\frac{T}{1+T}\right) = \frac{1}{T}.
\]
We conclude as
\begin{align*}
\partial \widetilde{F}_a(T) &= \partial\log\left(\frac{T}{1+T}\right) - \partial\log\left(\frac{(1+T)^a-1}{(1+T)^a}\right)\\
&=\frac{1}{T} - \frac{a}{(1+T)^a-1} = F_a(T). \qedhere
\end{align*}
\end{proof}

As in Lemma \ref{lem:bounded power series}, we must also check $\widetilde{F}_a(T) \in \mathscr{R}^+$ to use the Mahler correspondence.

\begin{lemma}
We have $\widetilde{F}_a(T) \in \mathscr{R}^+$.
\end{lemma}

\begin{proof}
It is convenient to note
\begin{equation}\label{eq:tilde F_a 2}
\widetilde{F}_a(T) = \log\left(\frac{T}{(1+T)^a-1} \cdot (1+T)^{a-1}\right).
\end{equation}
As in Proposition \ref{PropFaT}, write
\[ (1+T)^a-1 = aT (1 + T g(T)),  \qquad g(T) = \sum_{n \geq 2} a^{-1} {a \choose n} T^{n - 2}.\]
Recall from the proof of that proposition that we have
\begin{equation}\label{eq:poly expansion}
\frac{1}{(1+T)^a - 1} = \frac{1}{aT}(1 + T h(T)), 
\end{equation}
with 
\[
h(T) = \sum_{n \geq 1} (-1)^n T^{n-1} g(T)^n \in \Zp\lsem T\rsem.
\]
We thus see that
\[
    \frac{T}{(1+T)^a-1} = a^{-1} (1 + T h(T)),
\]
whose logarithm is given by
\begin{align}
\log\left(\frac{T}{(1+T)^a-1}\right) &= -\log_p(a) + \log(1+Th(T))\notag\\ 
&= -\log_p(a) + \sum_{n\geq 1} \frac{(-1)^{n+1}}{n} T^n h(T)^n. \label{eq:F_a tilde}
\end{align}
As in Lemma \ref{lem:bounded power series}, the coefficients here have logarithmic growth in $n$, so this lies in $\mathscr{R}^+$. Identically, $(1 + T)^{a-1} = 1 + T \sum_{n \geq 1} {{a - 1} \choose n} T^{n-1}$ also has well defined logarithm in $\mathscr{R}^+$. Adding these two elements of $\mathscr{R}^+$ yields $\widetilde{F}_a(T)$ and completes the proof.
\end{proof}

\begin{lemma}\label{lem:numerator}
We have $\big((1-\varphi\circ\psi)\widetilde{F}_{a}\big)(0) = -(1-p^{-1})\log_p(a)$.
\end{lemma}
\begin{proof}
First, by \eqref{eq:tilde F_a 2} we know that
\begin{equation}\label{eq:F_a(0)}
    \widetilde{F}_a(0) = \log\left(\frac{T}{(1+T)^a-1}\right)\Big|_{T=0}  + \log\left( (1+T)^{a-1}\right)|_{T=0} = -\log_p(a) + 0,
\end{equation}
where we use \eqref{eq:F_a tilde} to evaluate the first summand. Secondly,  we have
\begin{align*}
\varphi\circ\psi(\widetilde{F}_a)(T) &= \frac{1}{p}\sum_{\xi \in \mu_p}\widetilde{F}_a((1+T)\xi -1)\\
&= \frac{1}{p}\sum_{\xi \in \mu_p}\log\left(\frac{(1+T)\xi -1}{(1+T)\xi} \cdot \frac{(1+T)^a\xi^a}{(1+T)^a\xi^a-1}\right)\\
&= \frac{1}{p}\sum_{\xi \in \mu_p}\log\left(\frac{(1+T)\xi -1}{(1+T)^a\xi^a-1} \cdot (1+T)^{a-1}\xi^{a-1}\right).
\end{align*}
This rearranges to
\begin{align*}
\frac{1}{p}\log\left(\prod_{\xi \in \mu_p}\frac{(1+T)\xi -1}{(1+T)^a\xi^a-1} \cdot (1+T)^{a-1}\xi^{a-1}\right) &= \frac{1}{p}\log\left(\frac{(1+T)^p -1}{(1+T)^{ap}-1} \cdot (1+T)^{(a-1)p}\right).
\end{align*}
Here we simplify both terms of the fraction using $\prod_{\xi\in\mu_p}(X\xi - 1) = X^p - 1$, in the denominator noting that as $a$ is a topological generator of $\Z_p^\times$, we have $\{\xi^a : \xi \in \mu_p\} = \mu_p$. In the final term we use that $\prod_{\xi\in\mu_p}\xi^{a-1} = (\prod_{\xi \in \mu_p}\xi)^{a-1} = 1$.

Writing 
\[
    (1+T)^p-1 = pT(1+Tj(T)), \qquad \frac{1}{(1+T)^{ap}-1} = \frac{1}{apT}(1+Tk(T))
\]
analogously to \eqref{eq:poly expansion}, we find ultimately that
\[
\varphi\circ\psi(\widetilde{F}_a)(T) = \frac{1}{p}\log\left(\frac{1}{a} (1+Tj(T))(1+Tk(T)) \cdot (1+T)^{(a-1)p}\right),
\]
and the right-hand side at $T=0$ collapses to $-p^{-1}\log_p(a)$. Combining with \eqref{eq:F_a(0)},
\[
    \Big((1-\varphi\circ\psi)\widetilde{F}_a\Big)(0) = \widetilde{F}_a(0) - (\varphi\circ\psi)\widetilde{F}_a(0) = -\log_p(a) -p^{-1}\log_p(a) = -(1-p^{-1})\log_p(a),
\]
as required.
\end{proof}

Combining Equations \eqref{eq:zeta p residue} and \eqref{eq:zeta p residue 2} with Lemma \ref{lem:numerator}, we deduce that
\[ \lim_{s \to 1} (s - 1) \zeta_{p, p-1}(s) = 1 - p^{-1}, \]
completing the proof of Theorem \ref{thm:residue}(ii). \qed

%%======================================================
%%======================================================
\section{The $p$-adic family of Eisenstein series}\label{sec:eisenstein}

We finally take a brief detour to illustrate another example of $p$-adic variation in number theory, namely the $p$-adic variation of modular forms. In constructing the Kubota--Leopoldt $p$-adic $L$-function, we have seen many of the key ideas that go into the simplest example of this, namely the $p$-adic family of Eisenstein series, which we will illustrate below.

Let $k\geq 4$ be an even integer. The \emph{Eisenstein series of level $k$}, defined as
\[G_k(z) \defeq \sum_{\substack{c,d \in \Z\\(c,d) \neq (0,0)}} \frac{1}{(cz+d)^k},\hspace{12pt}z \in \uhp \defeq \{z\in\C: \mathrm{Im}(z) >0\}\]
can be viewed as a two-dimensional analogue of the zeta value $\zeta(k)$. It is an example of a \emph{modular form of weight $k$}. In the classical theory of modular forms, one computes the normalised Fourier expansion of such an object to be

\begin{align*}E_k(z) &\defeq \frac{G_k(z) (k-1)!}{2\cdot (2\pi i)^k} = \frac{\zeta(1-k)}{2} + \sum_{n\geq 1}\sigma_{k-1}(n)q^n,
\end{align*}
where $\sigma_{k-1}(n) = \sum_{0<d|n}d^{k-1}$ and $q = e^{2 i \pi z}$. In particular, it is a power series with rational coefficients. (This is a standard exercise; see \cite[Chapter 1.1]{DS05} for details).\\

From the definition, we see the Kubota--Leopoldt $p$-adic $L$-function as a pseudo-measure that, when evaluated at $x^k$ with $k \geq 4$ even, gives (up to an Euler factor) the constant coefficient of the Eisenstein series of weight $k$. The idea now is to find measures giving similar interpolations of the other coefficients. 
Fortunately, these are much easier to deal with: we only need interpolations of the functions $k \mapsto d^{k}$, where $k$ is varying $p$-adically. When $d$ is coprime to $p$, we do this by viewing $d$ as an element of $\zpe$ and considering the Dirac measure $\delta_d$ at $d$ (that is, evaluation at $d$). Indeed, $\int_{\zpe} x^k \cdot \delta_d = d^{k}$ for any $k \in \Z$.

When $d$ is divisible by $p$, however, we run into an immutable obstacle. There is no Dirac measure on $\Zp^\times$ corresponding to evaluation at $p$, since $p \notin \Zp^\times$. Moreover, the function $k \mapsto p^k$ can \emph{never} be interpolated continuously $p$-adically; it simply behaves too badly for this to be possible. Suppose there was indeed a measure $\theta_p$ with 
\[\int_{\Zp^\times} x^k \cdot \theta_p = p^k,\]
and then suppose $k_n$ is a strictly increasing sequence of integers $p$-adically tending to $k$. Then 
\[p^{k_n} = \int_{\Zp^\times}x^{k_n} \cdot \theta_p \longrightarrow \int_{\Zp^\times}x^{k} \cdot \theta_p = p^k,\]
which is clearly impossible since $p^{k_n}$ tends to 0.

We get around this issue by taking $p$-stabilisations to kill the coefficients at $p$.

\begin{definition}
 We define the \emph{$p$-stabilisation of $E_k$} to be 
\[E_k^{(p)}(z) \defeq E_k(z) - p^{k-1}E_k(pz).\]
An easy check shows that
\[E_{k}^{(p)} = \frac{(1-p^{k-1})\zeta(1-k)}{2} + \sum_{n\geq 1}\sigma_{k-1}^p(n)q^n,\]
where
\[\sigma_{k-1}^p(n) = \sum_{\substack{0< d|n\\ p\nmid d}} d^{k-1}.\]
Note $E_{k}^{(p)}$ is a modular form of weight $k$ and level $\Gamma_0(p) = \{\smallmatrd{a}{b}{c}{d} \in \mathrm{SL}_2(\Z) : p|c\}$.
\end{definition}

We've done all the work in proving the following result.

\begin{theorem} There exists a power series
\[\ \mathbf{E}(z) = \sum_{n\geq 0} A_n q^n \in Q(\Zp^\times)\lsem q\rsem  \]
such that:
\begin{itemize}
\item[(a)] $A_0$ is a pseudo-measure, and $A_n \in \Lambda(\Zp^\times)$ for all $n\geq 1$;
\item[(b)] For all even $k \geq 4$, we have
\[ \int_{\zpe} x^{k-1} \cdot \mathbf{E}(z) \defeq \sum_{n\geq 0} \bigg( \int_{\zpe} x^{k-1} \cdot A_n \bigg) q^n  = E_k^{(p)}(z). \]
\end{itemize}
\end{theorem}

\begin{proof}
The pseudo-measure $A_0$ is simply $x\zeta_p/2$ (shifting by 1 again, but in the opposite direction to before). We then define 
\[A_n = \sum_{\substack{0< d|n\\ p\nmid d}} \delta_d \in \Lambda(\Zp^\times).\]
By the interpolation property of the Kubota--Leopoldt $p$-adic $L$-function, $A_0$ interpolates the constant term of the Eisenstein series. We also have
\begin{align*}\int_{\Zp^\times}x^{k-1}\cdot A_n &= \sum_{\substack{0< d|n\\ p\nmid d}} \int_{\Zp^\times} x^{k-1} \cdot \delta_d = \sum_{\substack{0< d|n\\ p\nmid d}}  d^{k-1} = \sigma_{k-1}^{p}(n),
\end{align*}
so we get the required interpolation property.
\end{proof}

%\begin{exercise-credit} We can naturally view $\Lambda(\Zp^\times)$ (and its fraction ring $Q(\Zp^\times)$) as a ring of continuous\footnote{In fact, \emph{rigid analytic} functions. Rigid analysis is a theory of $p$-adic geometry.} 
%functions on $\mathcal{W}$. Show that there exists a power series
%\[\mathbf{E}(X,z) = \sum_{n\geq 0} A_n(X)q^n \in Q(\Zp^\times)\lsem q\rsem ,\]
%where $X$ is a parameter on $\mathcal{W}$, such that:
%\begin{itemize}
%\item[(a)] $A_0$ is a pseudo-measure, and $A_n \in \Lambda(\Zp^\times)$ for all $n\geq 1$, and, for all even $k \geq 4$, we have
%\[\mathbf{E}(\kappa_k,z) q^n = E_k^{(p)}(z),\]
%where $\kappa_k$ is the character $x \mapsto x^k$.
%\end{itemize}
%\end{exercise-credit}

\begin{remarks}\leavevmode
\begin{enumerate}

\item The power series $\mathbf{E}(z)$ is an example of a \emph{$\Lambda$-adic modular form}. In particular, it can be colloquially described as the statement:
\begin{quote}
``Eisenstein series vary $p$-adically continuously as you change the weight; if $k$ and $k'$ are close $p$-adically, then the Fourier expansions of $E_k$ and $E_{k'}$ are close $p$-adically.''
\end{quote}
The theory of $p$-adic modular forms, and in particular the construction and study of $p$-adic families of Eisenstein series, was introduced by Serre \cite{SerreFormesModulairesetZeta} to give a new construction of the $p$-adic zeta function of a totally real number field.  Indeed, the main idea of Serre's paper (cf.\ \cite[Corollaire 2]{SerreFormesModulairesetZeta}) was to show that if one can interpolate all of the non-constant coefficients -- which, as we saw above, is quite simple -- then this automatically gives an interpolation of the constant term, namely the $p$-adic zeta function, which is much more difficult to interpolate directly.

	\item These results are often presented instead using the weight space $\cW$ from Remark \ref{rem:weight space rigid analytic}. The integers are naturally a subset of $\mathcal{W}(\Cp)$ via the maps $\kappa_k : x \mapsto x^k$, and two integers $k, k'$ lie in the same unit ball if and only if $k \equiv k' \newmod{p-1}$. Let $\roi^+(\cW)$ be the space of bounded rigid analytic functions on $\cW$ (corresponding to measures on $\Zp^\times$), and $\mathcal{Q}(\cW)$ the space of rigid meromorphic functions on $\cW$ with a possible (simple) pole only at the trivial character (corresponding to pseudo-measures on $\Zp^\times$). Then we can view $\mathbf{E}$ as a power series $\mathbf{E}(q) = \sum_{n \geq 0} B_n q^n \in \mathcal{Q}(\cW)\lsem q\rsem$, with $B_n \in \roi^+(\cW)$ for all $n>0$, such that for all $k\geq 4$, we have $\mathbf{E}(q)(\kappa_k)\defeq \sum_{n\geq 0}B_n(\kappa_k) q^n = E_k^{(p)}(z)$. Hence we see $\mathbf{E}$ as a $p$-adic interpolation of the Eisenstein series over the weight space.

\item These two remarks go hand in hand. Indeed, pioneering work of Hida went much further on the study of $p$-adic weight families of modular forms, showing that similar families (known as \emph{Hida families}) exist for far more general modular forms. His work has been vastly generalised to the theory of Coleman families and eigenvarieties, parametrising the $p$-adic variation of modular/automorphic forms over appropriate weight spaces. Such families have important applications to the construction and study of $p$-adic $L$-functions; notable constructions in this direction are given in \cite{PS12, Bel12}. For a flavour of the theory of Hida and Coleman families of modular forms, see the books \cite{HidE} or \cite{Eigenbook}. 
\end{enumerate}
\end{remarks}

\vspace{40pt}

%%======================================================
%%======================================================
%%======================================================
%%======================================================
%%======================================================
%%======================================================
 
\begin{center}\scshape{{\LARGE Part II: Iwasawa's Main Conjecture}}\\[20pt]
\end{center}
\addcontentsline{toc}{part}{Part II: Iwasawa's Main Conjecture}

The second part of this work is devoted to the motivation, formulation and study of Iwasawa's Main Conjecture. We will start by studying the Coleman map, a map between towers of local units and $p$-adic measures. This gives a connection between the tower of cyclotomic units -- historically important for their connection to class numbers -- and the Kubota--Leopoldt $p$-adic $L$-function $\zeta_p$ from Part I, and hence a new arithmetic construction of $\zeta_p$ (Theorem \ref{thm:coleman to kl}). This construction can be seen as an arithmetic manifestation of the Euler product expression of the zeta function, and this point of view has led to beautiful generalisations now known as the theory of Euler Systems. We then prove a theorem of Iwasawa (Theorem \ref{thm:iwasawa 2}), relating the zeros of the $p$-adic $L$-function to arithmetic information in terms of units. Using these two results and class field theory, we will naturally arrive at the formulation and proof of (a special case of) the Main Conjecture (Theorem \ref{IMC}).

\section{Notation}\label{sec:coleman map notation}
%\begin{notation}\label{coleman map notation}
Our study of the Iwasawa Main Conjecture requires a certain amount of notation, which we introduce straight away for convenience. The following should be used as an index of the key notation, and the reader is urged to consult the definition of new objects as they appear in the text. 

Let $p$ be an odd prime. Throughout this section, we work with coefficient field $L = \Qp$. For $n \in \N$, write
\[ F_n \defeq \Q(\mu_{p^n}), \;\;\; F_n^+ \defeq \Q(\mu_{p^n})^+; \]
\[ \mathscr{V}_n \defeq \mathscr{O}_{F_n}^\times, \;\;\; \mathscr{V}_n^+ \defeq \mathscr{O}_{F_n^+}^\times; \]
\[ K_n \defeq \qp(\mu_{p^n}), \;\;\; K_n^+ \defeq \qp(\mu_{p^n})^+; \]
\[ \mathscr{U}_n \defeq \mathscr{O}_{K_n}^\times, \;\;\; \mathscr{U}_n^+ \defeq \mathscr{O}_{K_n^+}^\times, \]
where $(-)^+$ denotes the maximal totally real subfield (i.e.\ the fixed points under complex conjugation). The extensions $F_n / \Q$, $K_n / \qp$, $F^+_n / \Q$ and $K_n^+ / \qp$ are Galois and totally ramified at $p$ (the first two of degree $(p - 1)p^{n - 1}$ and the last two of degree $\frac{p - 1}{2} p^{n - 1}$) and we denote by $\mathfrak{p}_n$ (resp.\  $\mathfrak{p}_n^+$) the unique prime ideal of $F_n$ (resp.\ $F_n^+$) above the rational prime $p$. We let
\[ F_\infty = \Q(\mu_{p^\infty}) = \bigcup_{n\geq 1}F_n, \;\;\; F_\infty^+ \defeq (F_\infty)^+ = \bigcup_{n\geq 1}F_n^+, \]
\[ K_\infty = \Qp(\mu_{p^\infty}) = \bigcup_{n\geq 1}K_n, \;\;\; K_\infty^+ \defeq (K_\infty)^+ = \bigcup_{n\geq 1}K_n^+, \]
and denote  by $\pri$ (resp.\ $\pri^+$) the unique prime of $F_\infty$ (resp.\ $F^+_\infty$) above $p$. 

Write $\GG \defeq \Gal(F_\infty/\Q)$, $\GG^+ \defeq \Gal(F_\infty^+ / \Q) = \GG / \langle c \rangle$, where $c$ denotes the complex conjugation.  Since $\Gal(F_n/\Q)$ sends a primitive $p^n$th root of unity to a primitive $p^n$th root of unity, one deduces an isomorphism
\[ \chi_n : \Gal(F_n/\Q) \isorightarrow (\Z/p^n\Z)^\times\]
determined by the identity
\[\sigma(\xi) = \xi^{\chi_n(\sigma)},\]
for $\sigma \in \Gal(F_n/\Q)$ and $\xi \in \mu_{p^n}$ any primitive $p^n$th root of unity. By infinite Galois theory,
\begin{equation}\label{eq:cyclo char}
	 \GG = \Gal(F_\infty/\Q) \defeq \varprojlim_n \Gal(F_n/\Q) \isorightarrow \varprojlim_n (\Z/p^n\Z)^\times \cong \Zp^\times,
	 \end{equation}
% \begin{align*}\GG = \Gal(F_\infty/\Q) &\defeq \varprojlim_n \Gal(F_n/\Q)\\
% &= \varprojlim_n (\Z/p^n\Z)^\times \cong \Zp^\times,
% \end{align*}
via the \emph{cyclotomic character} $\chi \defeq \varprojlim \chi_n$. Note $\chi$ induces an isomorphism $\GG^+ \cong \zpe / \{ \pm 1 \}$. \\

We also define
\begin{equation}\label{eq:U1} \mathscr{U}_{n, 1} \defeq \{ u \in \mathscr{U}_n \; : \; u \equiv 1 \text { (mod } \mathfrak{p}_n) \}, \;\;\; \mathscr{U}_{n, 1}^+ \defeq \mathscr{U}_{n, 1} \cap \mathscr{U}_n^+. 
\end{equation}
The subsets $\mathscr{U}_{n,1}$ and $\mathscr{U}_{n,1}^+$ are important as they have the structure of $\zp$-modules (indeed, if $u \in \mathscr{U}_{n, 1}$ or $\mathscr{U}_{n, 1}^+$ and $a \in \zp$, then $u^a = \sum_{k \geq 0} {a \choose k} (u - 1)^k$ converges). By contrast, the full local units $\mathscr{U}_n$ and $\mathscr{U}_n^+$ are only $\Z$-modules. 

In general, our notation satisfies the following logic: if $X_n$ is any subgroup of $\mathscr{U}_n$, then we let $X_n^+ = X_n \cap \mathscr{U}_n^+$, $X_{n, 1} = X_n \cap \mathscr{U}_{n, 1}$ and $X_{n, 1}^+ = X_n^+ \cap \mathscr{U}_{n, 1}^+$. Observe that, since $\mathscr{V}_n \subseteq \mathscr{U}_n$, the same applies for any subgroup $X_n$ of $\mathscr{V}_n$.

It will be essential for our constructions and methods to consider these modules at all levels simultaneously. We define
%\[ \mathscr{E}_{\infty} \defeq \varprojlim_n \mathscr{E}_n, \;\;\; \mathscr{E}_{\infty, 1} \defeq \varprojlim_n \mathscr{E}_{n, 1}; \]
%\[ \mathscr{E}_{\infty}^+ \defeq \varprojlim_n \mathscr{E}_n^+, \;\;\; \mathscr{E}_{\infty, 1}^+ \defeq \varprojlim_n \mathscr{E}_{n, 1}^+; \]
\begin{equation}\label{eq:Uinfty 1} \mathscr{U}_{\infty} \defeq \varprojlim_n \mathscr{U}_n, \;\;\; \mathscr{U}_{\infty, 1} \defeq \varprojlim_n \mathscr{U}_{n, 1}; \end{equation}
\[ \mathscr{U}_{\infty}^+ \defeq \varprojlim_n \mathscr{U}_n^+, \;\;\; \mathscr{U}_{\infty, 1}^+ \defeq \varprojlim_n \mathscr{U}_{n, 1}^+, \]
where all limits are taken with respect to the norm maps. All of these infinite level modules are compact $\zp$-modules (since they are inverse limits of compact $\zp$-modules) and moreover they are all endowed with natural continuous actions of $\GG = \mathrm{Gal}(F_\infty / \Q)$ or $\GG^+ = \mathrm{Gal}(F_\infty^+ / \Q)$. Accordingly, they are endowed with continuous actions of the Iwasawa algebras $\Lambda(\GG)$ or $\Lambda(\GG^+)$ (which is the primary reason for passing to infinite level objects).

We fix once and for all a compatible system of roots of unity $(\xi_{p^n})_{n \in \N}$, that is, a sequence where $\xi_{p^n}$ is a primitive $p^n$th root of unity such that $\xi_{p^{n+1}}^p = \xi_{p^n}$ for all $n \in \N$. We let $\pi_n = \xi_{p^n} - 1$, which is a uniformiser of $K_n$. 

%\end{notation}

\section{The Coleman map}\label{sec:coleman map}

%In this section, we further explore the properties of $\zeta_p$. In particular, we start to bring the arithmetic of cyclotomic fields into the picture, via the theory of \emph{local units}, that is, units in cyclotomic extensions of $\Qp$. We prove a theorem of Coleman that relates local units to power series over $\Zp$, and hence -- under the Mahler transform -- to $p$-adic measures. We can use this theory to reconstruct $\zeta_p$ using \emph{cyclotomic units}.\\

In this section we prove a theorem of Coleman relating local units to power series over $\Zp$. Using this result, we construct in \S\ref{sec:coleman map} the \emph{Coleman map}, a device for constructing $p$-adic $L$-functions from the data of a compatible system of units. We will explain how the Kubota--Leopoldt $p$-adic $L$-function can be constructed from towers of cyclotomic units using the Coleman map. This map thus provides an important bridge between analytic objects ($p$-adic $L$-functions) and arithmetic structures (cyclotomic units), and will serve as the key step in our formulation of the Main Conjecture.

In \S\ref{sec:pr log}, we discuss a program started by Perrin-Riou to generalise Coleman's work. Given a $p$-adic Galois representation, Perrin-Riou's big logarithm maps construct a $p$-adic $L$-function from certain compatible systems of cohomology classes. Specialising to the representation $\Qp(1)$, her map recovers the Coleman map. The results of this section are thus a prototype for studying $p$-adic $L$-functions in a larger, more conceptual framework.

\subsection{Notation and Coleman's theorem}\label{sec:notation coleman}
Recall that $K_n = \Qp(\mu_{p^n})$ and $K_\infty = \Qp(\mu_{p^\infty})$ are the local versions of $F_n = \Q(\mu_{p^n})$ and $F_\infty = \Q(\mu_{p^\infty})$. We also defined
\[\mathscr{U}_n = \roi_{K_n}^\times\]
to be the module of local units at level $n$, took a compatible system ($\xi_{p^n}$) of primitive $p^n$th roots of unity, and defined $\pi_n \defeq \xi_{p^n}-1$, a uniformiser of $K_n$. Recall that we defined
\[\UU_\infty \defeq \varprojlim_n \UU_n,\]
where the projective limit is taken with respect to the norm maps $N_{n,n-1} : K_n \rightarrow K_{n-1}$.

Let us first motivate Coleman's theorem. The elements $\pi_n$ give a sequence of elements in the open unit ball 
\[
B(0,1) = \{z \in \C_p : |z|<1\},
\]
that approaches the boundary $\{|z| = 1\}$ as $n\to \infty$. Now, recall from Remark \ref{rem:analytic functions} that an element $f \in \Zp\lsem T\rsem$ can be viewed as a bounded rigid analytic function on $B(0,1)$; so any such $f$ gives a sequence $f(\pi_n)$ of elements of $\roi_{K_n}$. As we shall see below, the Weierstrass preparation theorem implies that $f$ is uniquely determined by this sequence. 

In the spirit of earlier chapters, where we sought analytic objects interpolating collections of specific values, it is natural to ask which sequences arise in this way. Precisely, given a sequence $\{u_n \in \roi_{K_n}\}$, can it be interpolated by a power series $f$, in the sense that $f(\pi_n) = u_n$ for all $n$? Coleman's theorem gives a positive answer to this question for norm-compatible sequences of units $(u_n)_{n \in \N} \in \UU_\infty$.

We begin with a simple observation, for a single fixed $n$.

\begin{lemma}
Let $u \in \UU_n$ be a local unit at level $n$. There exists a power series $f \in \Zp\lsem T\rsem^\times$ such that $f(\pi_n) = u.$
\end{lemma}
\begin{proof}
This is essentially immediate from the fact that $\pi_n$ is a uniformiser. Indeed, $K_n$ is totally ramified, so one can choose some $a_0 \in \Zp$ such that
\[a_0 \equiv u \newmod{\pi_n},\]
and then $a_1 \in \Zp$ such that
\[a_1 \equiv \frac{u - a_0}{\pi_n} \newmod{\pi_n},\]
and so on, defining $f(T) = \sum_n a_n T^n \in \Zp\lsem T\rsem.$ By construction, $f(\pi_n) = u$.  As $u$ is a unit, we have $a_0 \in \Zp^\times$. It's then an exercise to see $f \in \Zp\lsem T\rsem^\times$ is invertible.
\end{proof}

The problem with this proposition is that such a power series $f$ is far from being unique, since we had an abundance of choices for each coefficient. In the usual spirit of Iwasawa theory, Coleman realised that it was possible to solve this problem by passing to the infinite tower $K_\infty$, considering all $n$ simultaneously. %Coleman's theorem says that for each $u \in \UU_\infty$, there is a unique power series $f_u$ satisfying the condition of the above proposition for all $n$.

\begin{theorem} [Coleman] \label{thm:coleman power series}
There exists a unique injective homomorphism
\begin{align*}
\mathscr{U}_\infty &\longrightarrow \zp\lsem T\rsem^\times  \\
u &\longmapsto f_u
\end{align*}
of multiplicative groups such that $f_u(\pi_n) = u_n$ for all $u \in \mathscr{U}_\infty$ and $n\geq 1$.
\end{theorem}

Coleman actually proved something stronger. He described a precise subspace of $\Zp\lsem T\rsem$ in which the associated interpolating power series $f_u$ lives; it is invariant under a certain norm operator $\cN$ on $\Zp\lsem T\rsem$. In particular, norm-compatibility on the right-hand side translates into norm-invariance on the left-hand side. We will prove all of this in \S\ref{sec:proof of Coleman} below.

First, though, we study an important application, explaining how this theorem is related to the Kubota--Leopoldt $p$-adic $L$-function.

\subsection{Example: cyclotomic units} \label{sec:cyclo units coleman}

Let $a \in \Z$ prime to $p$, and define
\[c_n(a) \defeq \frac{\xi_{p^n}^a -1}{\xi_{p^n} - 1} \in \UU_n.\]
 This is indeed a unit, as both $\xi_{p^n}^a -1$ and $\xi_{p^n} - 1$ are uniformisers in $K_n$.

\begin{lemma}
We have $c(a) \defeq (c_n(a))_n \in \UU_\infty.$
\end{lemma}

\begin{proof}
	This is equivalent to proving that $N_{n,n-1}(c_n(a)) = c_{n-1}(a).$ Since the minimal polynomial of $\xi_{p^n}$ over $K_{n-1}$ is $X^p - \xi_{p^{n-1}}$, for any $b$ prime to $p$ we see that
	\begin{equation*}N_{n,n-1}(\xi_{p^n}^b - 1) = \prod_{\eta \in \mu_p}(\xi_{p^n}^b \eta - 1)
		= \xi_{p^n}^{bp} - 1 = \xi_{p^{n-1}}^b -1,
	\end{equation*}
	where in the penultimate equality we have used the identity $X^p - 1 = \prod_{\eta \in \mu_p} (X \eta - 1)$. Applying this with $b=a$ shows the numerator of $c(a)$ is norm-compatible, and with $b=1$ the denominator. We conclude as norm is multiplicative.
\end{proof}

It is possible to write down $f_{c(a)} \in \Zp\lsem T\rsem^\times$ directly by inspection. Indeed, we see that
\[
f_{c(a)}(T) = \frac{(1+T)^a - 1}{T}
\]
satisfies the required property (and $f_{c(a)}$ is even a polynomial). We now connect this to the construction studied in \S\ref{kl}. Recall the operator $\partial = (1+T)\tfrac{d}{dT}$ from Lemma \ref{LemmaMultiplicationbyx}.

\begin{proposition}\label{prop:coleman zetap}
We have
\[\partial \log f_{c(a)}(T) = a - 1 - F_a(T),\]
where $F_a(T)$ is the power series defined in Lemma \ref{lem:define F_a}.
\end{proposition}

\begin{proof}
We compute directly that
\begin{align*}\partial \log f_{c(a)} &=   \partial \log \big( (1+T)^a - 1\big) - \partial\log(T)\\
&=  \frac{a(1+T)^a}{(1+T)^a - 1} - \frac{T+1}{T}\\
&= a - 1 - \frac{1}{T} + \frac{a}{(1+T)^a - 1} \\
&= a - 1 - F_a(T).\qedhere
\end{align*}
\end{proof}

\begin{lemma}\label{lem:relate cyclo to mua}
We have
\[
	\mathrm{Res}_{\Zp^\times}(\mu_{\partial\log f_{c(a)}}) = -\mathrm{Res}_{\Zp^\times}(\mu_a),
	\]
where $\mu_a$ is the measure of Definition \ref{DefinitionMeasuremua}.
\end{lemma}

\begin{proof}
    In terms of power series, the restriction to $\Zp^\times$ corresponds to applying the operator $(1-\varphi\circ\psi)$. As $1-\varphi\circ\psi$ kills the term $a-1$, we find that, as required,
	\[
	(1-\varphi\circ\psi)\partial\log f_{c(a)} = - (1-\varphi\circ\psi)F_a. \qedhere
	\]
\end{proof}

\begin{remark} \label{rem:restriction Coleman}
	The measure $\mu_a$ was used in the construction of $\zeta_p$. Later in \S\ref{sec:coleman map} we will use Theorem \ref{thm:coleman power series} to give a new construction of $\zeta_p$ via the cyclotomic units. We will see more about the units $c_n(a)$, and in particular the module they generate in $\UU_n$, in \S\ref{sec:iwasawa zeros}.
\end{remark}

\subsection{Proof of Coleman's theorem}
\label{sec:proof of Coleman}
 First we see that there is at most one power series $f_u$ attached to a system of units $u$.

\begin{lemma}\label{lem:unique coleman}
Suppose $u = (u_n) \in \UU_\infty$ and $f, g \in \Zp\lsem T\rsem^\times$ both satisfy 
\[f(\pi_n) = g(\pi_n) = u_n\]
for all $n\geq 1$. Then $f = g$.
\end{lemma}
\begin{proof}
The Weierstrass preparation theorem says that we can write any non-zero $h(T) \in \Zp\lsem T\rsem $ in the form $p^m u(T)r(T),$ where $u(T)$ is a unit and $r(T)$ is a polynomial. Any such $h(T)$ converges to a function on the maximal ideal in the ring of integers of $\overline{\Q}_p$, and since $u(T)$ cannot have zeros, we deduce that $h(T)$ has a finite number of zeros in this maximal ideal. Now $(\pi_n)_{n\geq 1}$ is an infinite sequence of elements in this maximal ideal, so the fact that $(f-g)(\pi_n) = 0$ for all $n\geq 1$ implies that $f = g$, as required.
\end{proof}

We now move to showing the existence of such a series $f_u$. The key idea in the proof is to identify the subspace of $f \in \Zp\lsem T\rsem^\times$ such that $(f(\pi_n))_n \in \UU_\infty$; that is, identify the \emph{image} in Theorem \ref{thm:coleman power series}. For this, we want norm-compatibility of $f(\pi_n)$. Lemma \ref{lem:norm and trace} and Proposition \ref{prop:R} below will show the existence of a norm operator on power series, and then translate the norm compatibility condition of units into norm invariance of power series; Lemma \ref{lem:norm continuity} will show certain continuity properties of this norm operator, which will allow us to prove Coleman's theorem by a standard diagonal argument. \\

Recall that the action of $\varphi$ on $f(T) \in \Zp\lsem T\rsem $ is defined by $\varphi(f)(T) = f((1+T)^p - 1)$  (see \eqref{eq:varphi power series}), and that this action is injective. Importantly,  we also have
\begin{equation}\label{eq:varphi pin}
\varphi(f)(\pi_{n+1}) = f((\pi_{n+1} + 1)^p - 1) = f(\xi_{p^{n+1}}^p - 1)  = f(\xi_{p^n} - 1) = f(\pi_n).
\end{equation}
From our work with measures (cf.\ \S\ref{SubSectionphipsi}), we have also seen the existence of an additive operator $\psi$ with the property that
\[(\varphi\circ\psi)(f)(T) = \frac{1}{p}\sum_{\eta \in \mu_p} f(\eta (1+T)-1). \] 
We henceforth call $\psi$ the \emph{trace} operator (this terminology will become clear after Lemma \ref{lem:norm and trace}). We now define a multiplicative version of this operator.

\begin{lemma}\label{lem:norm and trace}
There exists a unique multiplicative operator $\cN$ on $\zp\lsem T\rsem$, the \emph{norm operator}, such that
\[
(\varphi\circ\cN)(f)(T) = \prod_{\eta \in \mu_p} f(\eta(1+T)-1).
\]
\end{lemma}

\begin{proof}
The ring $B = \zp\lsem T\rsem $ is an extension of $A = \zp\lsem \varphi(T)\rsem  = \varphi(\Zp\lsem T\rsem )$ of degree $p$, the former being obtained by adjoining a $p$th root of $(1 + T)^p$ to the latter. Each automorphism of $B$ over $A$ is given by $T \mapsto (1 + T) \eta - 1$ for some $\eta \in \mu_p$. There is a norm map 
\begin{align*}
N_{B/A}: \Zp\lsem T\rsem  &\longrightarrow \varphi(\Zp\lsem T\rsem )\\
f(T) &\longmapsto \prod_{\eta \in \mu_p} f((1 + T) \eta - 1).
\end{align*}
The norm operator $\cN$ is then defined to be $\varphi^{-1} \circ N_{B/A}$, recalling that $\varphi$ is injective.
\end{proof}

We similarly have $\psi = p^{-1}\varphi^{-1}\circ \mathrm{Tr}_{B/A}$, where $\mathrm{Tr}_{B/A}$ is the trace operator for the extension $B/A$ in the proof of Lemma \ref{lem:norm and trace}. Note that as $\mathcal{N}$ is multiplicative, it preserves $\Zp\lsem T\rsem^\times$. Moreover, it is closely related to the norm operator $N_{n+1,n} : \UU_{n+1} \to \UU_n$ used to defined $\UU_\infty$, via the following lemma.

\begin{lemma}\label{lem:norm power series vs units}
The following diagram commutes:
\begin{equation} \label{EqCommutDiagNorm}
\begin{tikzcd}
\zp\lsem T \rsem^\times \ar[rr, "{\ \ f \mapsto f(\pi_{n+1}) \ \ \ }"] \ar[d, "\mathcal{N}"'] & & \mathcal{U}_{n+1} \ar[d, "{N_{n+1, n}}"] \\
\zp\lsem T \rsem^\times \ar[rr, "{\ \ f \mapsto f(\pi_n)\ \ \ }"'] & & \mathcal{U}_n.
\end{tikzcd}
\end{equation}
\end{lemma}
\begin{proof}
If $f \in \Zp\lsem T \rsem^\times$, then $f(\pi_n) \in \UU_n$ for all $n$, as $f(\pi_n)^{-1} = f^{-1}(\pi_n)$ is also integral. In particular, the horizontal maps are well-defined. 

Observe now that, as the minimal polynomial of $\xi_{p^{n+1}}$ over $K_n$ is $X^p - \xi_{p^n} = 0$, we can write the right-hand norm as
\begin{align*}
N_{n+1,n}\big(f(\pi_{n+1})\big) &= \prod_{\eta \in \mu_p}f(\eta \xi_{p^{n+1}} - 1) \\
&= (\varphi \circ \mathcal{N})(f)(\pi_{n+1}) \\
&= (\mathcal{N}f) (\pi_{n}),
\end{align*}
giving exactly the claimed commutativity. In the final step we have used \eqref{eq:varphi pin}.
\end{proof}
In particular, we get the following.

% \begin{proof}(Sketch). One can check (see exercises) that 
% \[\mathrm{Im}(\varphi) = \big\{h \in \Zp\lsem T\rsem : h(\zeta(1+T) -1) = h(T) \text{ for all }\zeta\in\mu_p\big\}.\]
% With this in hand, consider the function 
% \[g(T) = \prod_{\zeta \in \mu_p} f(\zeta(1+T)-1).\]
% An easy check shows that $g(\zeta(1+T)-1) = g(T)$ for all $p$th roots of unity $\zeta$, so that $g = \varphi(h)$ for some $h \in \Zp\lsem T\rsem $. Define $\mathcal{N}(f) = h$.
% \end{proof}

\begin{proposition} \label{prop:R}
	There is an injective map 
	\begin{align*}
		R : (\Zp\lsem T\rsem^\times)^{\cN=\mathrm{id}} &\longhookrightarrow \UU_\infty,\\
		f &\longmapsto (f(\pi_n))_n.
	\end{align*}
%Let $f \in \Zp\lsem T\rsem ^\times$. Then $f(\pi_n) \in \UU_n$ for all $n$. Moreover, if $\cN(f) = f$, that is, $f$ is invariant under the norm map, then we have
%\begin{equation}\label{eq:cyclo norm}
%	N_{n+1,n}\big(f(\pi_{n+1})\big) = f(\pi_n).
%\end{equation}
%In particular, there is an injective map
\end{proposition}

\begin{proof}
Suppose $\mathcal{N}(f) = f$. By Lemma \ref{lem:norm power series vs units}, we deduce that 
\begin{equation}\label{eq:cyclo norm}
N_{n+1,n}\big(f(\pi_{n+1})\big) = f(\pi_n),
\end{equation}
so $(f(\pi_n))_n \in \UU_\infty$.
\end{proof}

To prove Theorem \ref{thm:coleman power series} it suffices to prove that the map $R$ is surjective. We need the following lemma on the behaviour of $\cN$ modulo powers of $p$.

\begin{lemma} \label{lem:norm continuity} Let $f(T) \in \zp\lsem T\rsem $. Then:
\begin{itemize}
%\item[(i)] If $\varphi(f)(T)\equiv 0 \newmod{p^k}$ for some $k \geq 0$, then $f(T) \equiv 0 \newmod{p^k}.$
\item[(i)] If $\varphi(f)(T)\equiv 1 \newmod{p^k}$ for some $k \geq 0$, then $f(T) \equiv 1 \newmod{p^k}.$
\item[(ii)] We have
\[\cN(f) \equiv f \newmod{p}.\]
\end{itemize}
Now suppose $f \in \Zp\lsem T\rsem^\times$. Then:
\begin{itemize}
\item[(iii)] If $f \equiv 1 \newmod{p^k}$ with $k\geq 1$, then
\[\cN(f) \equiv 1 \newmod{p^{k+1}}.\]
\item[(iv)] If $k_2 \geq k_1 \geq 0$, then 
\[\cN^{k_2}(f) \equiv \cN^{k_1}(f) \newmod{p^{k_1+1}}.\]
\end{itemize}
\end{lemma} 

\begin{proof}
We leave parts (i) and (ii) as an exercise (cf.\ \cite[Lem.\ 2.3.1]{CS06}). To see part (iii), suppose that $f \equiv 1 \newmod{p^k}$ with $k\geq 1$, and recall that $\pri_1$ is the maximal ideal of the ring of integers of $K_1 = \Qp(\mu_p)$. For each $\eta \in \mu_p$, as $(\eta - 1)(1+T) \in \pri_1\Zp\lsem T\rsem $, we have
\[\eta(1+T)- 1 \equiv T\newmod{\pri_1 \Zp\lsem T\rsem },\]
so that
\[f\big(\eta(1+T) - 1\big) \equiv f(T) \newmod{\pri_1 p^k \Zp\lsem T\rsem }\]
by considering each term separately. It follows that
\begin{align*}\varphi\circ \cN(f)(T) &= \prod_{\eta \in \mu_p} f\big(\eta(1+T) - 1\big)\\
&\equiv f(T)^p \newmod{\pri_1 p^k\Zp\lsem T\rsem }.
\end{align*}
Since both $\varphi\circ\cN(f)$ and $f(T)^p$ are elements of $\Zp\lsem T\rsem $, this is actually an equivalence modulo $\pri_1 p^k \cap \Zp = p^{k+1}.$ If $f(T) \equiv 1 \newmod{p^k}$, then $f(T)^p \equiv 1 \newmod{p^{k+1}}$, and then the proof follows from part (i).

To see part (iv), from part (ii) we see that
\[\frac{\cN^{k_2-k_1}f }{f} \equiv 1 \newmod{p}.\]
Then iterating $\cN$ and using part (iii) $k_1$ times, we obtain the result.
\end{proof}

\begin{proposition}\label{prop:R surjective}
	The map $R : (\Zp\lsem T\rsem^\times)^{\cN=\mathrm{id}} \hookrightarrow \UU_\infty$ is surjective.
\end{proposition}

\begin{proof}
Let $u  = (u_n)_{n \geq 1} \in \UU_\infty$. For each $n$, choose $f_n \in \Zp\lsem T\rsem^\times$ such that
\[
f_n(\pi_n) = u_n.
\]
We claim that $\mathcal{N} f_{n+1}(\pi_n) = u_n$. Indeed, using Lemma \ref{lem:norm power series vs units} we have
\[ \mathcal{N} f_{n+1}(\pi_n) = N_{n+1, n}(f_{n+1}(\pi_{n+1})) = N_{n+1, n}(u_{n+1}) = u_n.  \]
Iterating, for any $k \geq 0$ we have
\begin{equation} \label{eqtmp1}
(\mathcal{N}^k f_{n+k})(\pi_n) = u_n.
\end{equation}

In view of Lemma \ref{lem:norm continuity}(iv), we define
\[ 
g_n \defeq \mathcal{N}^{n} f_{2n} \in \Zp\lsem T\rsem^\times. 
\] 
Then, for any $m \geq n$, we have
\begin{align*}
u_n &= \mathcal{N}^{2m - n} f_{2m}(\pi_n) \\
&\equiv \mathcal{N}^{m} f_{2m}(\pi_n) = g_m(\pi_n) \newmod{p^{m+1}},
\end{align*}
where the first equality is \eqref{eqtmp1} and the congruence is Lemma \ref{lem:norm continuity}(iv), taking $k_2 = 2m - n$ and $k_1 = m$. Hence as $m \to \infty$, we have $g_m(\pi_n) \to u_n$ for all $n$. It thus suffices to find a convergent subsequence of $(g_m)$; but such a subsequence exists, as $\Zp\lsem T\rsem^\times$ is compact. Letting $f_u \in \zp\lsem T\rsem^\times$ denote the limit of this subsequence, we have $f_u(\pi_n) = u_n$ for all $n$.

It remains to show that $\mathcal{N}(f_u) = f_u$. Indeed, as the sequence $(u_n)$ is norm compatible, we have, using Lemma \ref{lem:norm power series vs units}, 
\[
    \mathcal{N}(f_u)(\pi_{n}) = N_{n+1, n}f_u(\pi_{n+1}) = N_{n+1, n}(u_{n+1}) = u_n = f_u(\pi_n).
\]
As $\mathcal{N}(f_u)$ and $f_u$ are both Coleman power series for $u$, by Lemma \ref{lem:unique coleman} they are equal.
\end{proof}

With this in hand, we have proved the following  more precise version of Theorem \ref{thm:coleman power series}.

\begin{theorem} \label{thm:coleman map 2}
	There exists a unique isomorphism of groups
	\begin{align*}
		\mathscr{U}_\infty &\to \big(\zp\lsem T\rsem^\times\big)^{\cN=\mathrm{id}}  \\
		u &\mapsto f_u
	\end{align*}
	such that $f_u(\pi_n) = u_n$ for all $u \in \mathscr{U}_\infty$ and $n\geq 1$.
\end{theorem}

\begin{proof}
	By Propositions \ref{prop:R} and \ref{prop:R surjective}, we have a bijection $R: (\Zp\lsem T\rsem^\times)^{\cN=\mathrm{id}} \isorightarrow \UU_\infty$. This is an isomorphism, and  $R^{-1}$ gives the required map. We have $f_u(\pi_n) = u_n$ by construction of $R$ and uniqueness follows from Lemma \ref{lem:unique coleman}. 
\end{proof}

\subsection{Definition of the Coleman map} \label{sec:coleman map definition}

The Coleman map is motivated by the example of \S\ref{sec:cyclo units coleman}, where we saw that a distinguished family of local units -- the cyclotomic units -- are strongly linked to the Kubota--Leopoldt $p$-adic $L$-function. In particular, given the construction of $\zeta_p$ in \S\ref{kl} and Lemma \ref{lem:relate cyclo to mua}, $\zeta_p$ can be defined by the following procedure:
\begin{enumerate}
	\item Consider the tower $c(a)$ of cyclotomic units.
	\item Take its Coleman power series $f_{c(a)}$. 
	\item Apply $\partial\log$.
	\item Apply $(1-\varphi\circ\psi)$.
	\item Apply $\partial^{-1}$.
	\item Pass to the corresponding measure on $\zp^\times$ by inverting the Mahler transform.
	\item Finally, divide by $\theta_a$.
\end{enumerate} 
Recall that in terms of measures, step (4) corresponds to restriction to $\Zp^\times$, and (5) to multiplication by $x^{-1}$. We are therefore led to consider the following construction.
\begin{definition} \label{def:coleman map}
	Let
\[
\mathrm{Col} :	\sU_\infty \xrightarrow{u \mapsto f_u(T)} (\Zp\lsem T\rsem ^\times)^{\cN=\mathrm{id}} \xrightarrow {\partial\log} \Zp\lsem T\rsem  \xrightarrow{1-\varphi\circ\psi} \Zp\lsem T\rsem ^{\psi = 0}\xrightarrow{\partial^{-1}}\Zp\lsem T\rsem ^{\psi = 0} \xrightarrow{\sA^{-1}} \Lambda(\Zp^\times),
\]
where the first map is Coleman's isomorphism, the second is the logarithmic derivative appearing in \S\ref{sec:cyclo units coleman}, the third is the measure-theoretic restriction from $\Zp$ to $\Zp^\times$, the fourth is multiplication by $x^{-1}$, and the last is the Mahler correspondence (\S\ref{sec:mahler} and Corollary \ref{CorollarySupportedZpet}).
\end{definition}

Via \S\ref{sec:cyclo units coleman}, we have the following description of the Kubota--Leopoldt $p$-adic $L$-function.

\begin{theorem} \label{thm:coleman to kl} 
For any topological generator $a$ of $\Zp^\times$, we have an equality of pseudo-measures 
\[
	\zeta_p = \frac{\mathrm{Col}(c(a))}{\theta_a} \in Q(\Zp^\times).
\]
\end{theorem}
%In the next section, we put this this connection between cyclotomic units and the Kubota--Leopoldt $p$-adic $L$-function into a Galois-theoretic framework.

\subsection{Generalisations: The Kummer sequence, Euler systems and $p$-adic $L$-functions}\label{sec:pr log}

We conclude this section with a digression on the generalisation of  the Coleman map that leads to a conjectural construction, under the assumption of the existence of certain global cohomological elements, of $p$-adic $L$-functions of more general motives. This section is included as additional context and may be skipped on a first reading.  Throughout, if $F$ is a number field, we let $\mathscr{G}_F$ denote its absolute Galois group.

Consider, for $m \geq 1$, the Kummer exact sequence
\begin{equation} \label{eq:kummer}
0 \to \mu_{p^m} \to \mathbf{G}_{\rm m} \xrightarrow{x \mapsto x^{p^m}} \mathbf{G}_{\rm m} \to 0.
\end{equation}
Evaluating at $\overline{\Q}$ and taking fixed points by $\mathscr{G}_F$, this short exact sequence induces, for any number field $F$, a long exact sequence on cohomology
\begin{equation}\label{eq:kummer LES}
0 \to \mu_{p^m}(F) \to F^\times \xrightarrow{x \mapsto x^{p^m}} F^\times \longrightarrow H^1(F, \mu_{p^m}) \to H^1(F, \overline{\Q}^\times). 
\end{equation}
Here, for any topological $\mathscr{G}_F$-module $A$, we write $H^1(F, A) \defeq H^1(\mathscr{G}_F, A)$ for the Galois cohomology, i.e.\ the continuous group cohomology of $\mathscr{G}_F$. By Hilbert 90, we have $H^1(F, \overline{\Q}^\times) = 0$. Taking inverse limits over $m \geq 1$, which is exact, we obtain an isomorphism called the Kummer map
\begin{equation} \label{EqKummerMap}
\delta : F^\times \otimes \zp \isorightarrow H^1(F, \zp(1)).
\end{equation}
 Explicitly, at each finite level, the isomorphism 
 \[F^\times \otimes \Z / p^n \Z = F^\times / (F^\times)^{p^n} \isorightarrow H^1(\mathscr{G}_F, \mu_{p^n})\]
is given as follows. Take $a \in F^\times$ and take any $b \in \overline{\Q}^\times$ such that $b^{p^n} = a$. Then $c_a : \sigma \mapsto \frac{\sigma(b)}{b}$ defines a $1$-coycle on $\mathscr{G}_F$ and it is a coboundary if and only if $a$ is a $p^n$th power in $F^\times$, which shows that the map sending the class of $a$ to the class of $c_a$ is well defined. 

Let $m = D p^n$, $n \geq 1$, and define
\[ \mathbf{c}_m \defeq \frac{\xi_m^{-1} - 1}{ \xi_m - 1} \in \mathscr{O}_{\Q(\mu_m)}^\times, \]
a generalisation of the cyclotomic units $c_n(-1)$ (where $D=1$) from Example \ref{sec:cyclo units coleman}, where $(\xi_m)_m$ denote a compatible system of $m$th roots of unity. One can show that these elements satisfy the following relations with respect to the norm maps:
\[ \mathrm{N}_{\Q(\mu_{m \ell}) / \Q(\mu_m)}(\mathbf{c}_{m \ell}) =
  \begin{cases}
    \mathbf{c}_{m}       & \quad \text{if } \ell \mid m \\
    (1 - \ell^{-1}) \mathbf{c}_m  & \quad \text{if } \ell \nmid m.
  \end{cases}
\]
Using the Kummer map from Equation \eqref{EqKummerMap}, we get elements $\mathbf{z}_m \defeq \delta(\mathbf{c}_m) \in H^1(\Q(\mu_m), \zp(1))$ satisfying
\[ \mathrm{cores}_{\Q(\mu_{m \ell}) / \Q(\mu_m)}(\mathbf{z}_{m \ell}) =
  \begin{cases}
    \mathbf{z}_{m}       & \quad \text{if } \ell \mid m \\
    (1 - \mathrm{Frob}_\ell^{-1}) \mathbf{z}_m  & \quad \text{if } \ell \nmid m,
  \end{cases}
\]
where we have used that $\mathrm{Frob}_\ell$ acts on $\zp(1)$ simply by multiplication by $\ell$ on $\zp(1)$. Observe also that $(1 - \ell^{-1})$ is the Euler factor at $\ell$ of the Riemann zeta function (evaluated at $s = 1$). This admits the following huge generalisation, as described comprehensively in \cite{Rub00}.

\begin{definition}\label{def:euler systems}
Let $\Sigma$ be a finite set of primes containing $p$, let $V \in \mathrm{Rep}_L \mathscr{G}_\Q$ be a global $p$-adic Galois representation which is unramified outside $\Sigma$, and let $T \subseteq V$ be an $\mathscr{O}_L$-lattice stable under $\mathscr{G}_\Q$. An \emph{Euler system} for $(V, T, \Sigma)$ is a collection of classes
\[ \mathbf{z}_m \in H^1(\Q(\mu_m), T), \]%\;\;\; (m, \Sigma) = \{ p \} \] 
where $m$ is of the form $m = p^n m'$ with $n \geq 0$, and where $m'$ is a square-free product of prime noumbers not belonging to $\Sigma$,  satisfying
\[ \mathrm{cores}_{\Q(\mu_{m \ell}) / \Q(\mu_m)}(\mathbf{z}_{m \ell}) =
  \begin{cases}
    \mathbf{z}_{m}       & \quad \text{if } \ell = p \\
    P_\ell(V^*(1), \sigma_\ell^{-1}) \mathbf{z}_m  & \quad \text{if } \ell \neq p,
  \end{cases}
\]
where $P_\ell(V^*(1), X) = \det(1 - \mathrm{Frob}_\ell^{-1} X | {V^*(1)}^{I_\ell})$ is the Euler factor at $\ell$ of the $L$-function associated to $V^*(1)$ and $\sigma_\ell$ denotes the image of $\mathrm{Frob}_\ell$ in $\mathrm{Gal}(\Q(\mu_m) / \Q)$.
\end{definition}

The cyclotomic units form an Euler system for the representation $\zp(1)$, and lie  at the heart of Rubin's proof of the Main Conjecture. In general, constructing Euler systems for a Galois representation is a very difficult task, and few examples exist at the moment.

 We now describe a rephrasing of Coleman's map that more easily generalises. Above, we showed that evaluating Kummer's exact sequence \eqref{eq:kummer} at $\overline{\Q}$, taking the long exact sequence in Galois cohomology for $F$, and taking an inverse limit, we get an isomorphism $F^\times \otimes \Zp \cong H^1(F,\zp(1))$. In exactly the same way, replacing $\overline{\Q}$ by $\overline{\Q}_p$, and $F$ by the finite extension $K_n$ of $\qp$ for $n \geq 1$, we obtain an isomorphism
\begin{equation*}
K_n^\times \otimes \zp \cong H^1(K_n, \zp(1)).
\end{equation*}
These isomorphisms intertwine the norm maps on the left hand side with the corestriction maps in cohomology on the right hand side, and hence, considering the inverse limit over all $n$, we see that there is an isomorphism
\begin{equation} \label{eq:kummer 2} \varprojlim_{n \geq 1} K_n^\times \otimes \zp \cong \varprojlim_{n \geq 1} H^1(K_n, \zp(1)).
\end{equation}
We define the \emph{Iwasawa cohomology} to be
\[ H^1_{\rm Iw}(\qp, \qp(1)) \defeq \varprojlim_{n \geq 1} H^1(K_n, \zp(1)) \otimes_{\zp} \qp. \]
Such groups can be attached to general Galois representations (see below), and they are a natural generalisation of the local units. 

%To make this precise, note $K_n^\times = \mathscr{O}_{K_n}^\times \times (\pi_n)^\Z$. As $\mathscr{O}_{K_n}^\times$ only has prime-to-$p$ torsion, we deduce that $K_n^\times \otimes \zp \cong \mathscr{O}_{K_n}^\times / (\mathscr{O}_{K_n}^\times)_{\rm tors} \times \zp$. We obtain a natural map $\UU_n = \mathscr{O}_{K_n}^\times \to K_n^\times \otimes \Zp$, yielding a map $\UU_\infty \to \varprojlim K_n^\times \otimes \Zp$. Composing this with \eqref{eq:kummer 2}, we obtain a map
To make this precise, note that the inclusion $\UU_n = \roi_{K_n}^\times \subset K_n^\times$ induces a natural map $\UU_n \to K_n^\times \otimes \Zp$, yielding a map $\UU_\infty \to \varprojlim K_n^\times \otimes \Zp$. Composing this with \eqref{eq:kummer 2}, we obtain a map

\[ 
\kappa: \mathscr{U}_\infty \longrightarrow \varprojlim_{n \geq 1} H^1(K_n, \zp(1)). 
\]
One can then show that there exists a map
\[ \mathrm{Col'} : H^1_{\rm Iw}(\qp, \qp(1)) \to \mathscr{M}(\Zp^\times, \qp), \] 
where we recall that $\mathscr{M}(\Zp^\times, \qp) = \Lambda(\Zp^\times) \otimes_{\zp} \qp$ is the space of $\qp$-valued measures on $\GG$, making the diagram
\[
\begin{tikzcd}
\mathscr{U}_\infty \ar[rr, "\kappa"] \ar[rd, "\mathrm{Col}"'] &  & H^1_{\rm Iw}(\qp, \qp(1)) \ar[dl, "\mathrm{Col}'"] \\
& \mathcal{M}(\zpe, \qp) &
\end{tikzcd}
\]
commute. 

By localising, the Euler system of cyclotomic units give rise to an element of the Iwasawa cohomology. By combining the above with Proposition \ref{prop:coleman zetap}, we see that the $p$-adic zeta function can be obtained by evaluating $\mathrm{Col}'$ at this Iwasawa cohomology class (and dividing through by the measure $\theta_a$ to make it independent of $a$, which introduces a pole). 

The advantage of this reformulation is that Iwasawa cohomology generalises well, as we now explain. Let $V \in \mathrm{Rep}_L \mathscr{G}_{\qp}$ be any $p$-adic representation of $\mathscr{G}_{\qp}$, i.e\ a finite dimensional $L$-vector space $V$ equipped with a continuous linear action of $\mathscr{G}_{\qp}$. As before, we define its Iwasawa cohomology groups as
\[ 
H^1_{\rm Iw}(\qp, V) \defeq \varprojlim_{n \geq 1} H^1(K_n, T) \otimes_{\mathscr{O}_L} L, 
\] 
where $T \subseteq V$ denotes any $\mathscr{O}_L$-lattice of $V$ stable under the action of the Galois group $\mathscr{G}_{\qp}$, and where as before the inverse limit is taken with respect to the corestriction maps in cohomology. Morally, Iwasawa cohomology groups are the groups where the local part at $p$ of an Euler system of a global $p$-adic representation lives. Assuming that the representation is crystalline\footnote{Loosely, a $p$-adic representation of $\mathscr{G}_{\Qp}$ being \emph{crystalline} is a condition from $p$-adic Hodge theory that is the $p$-adic equivalent to an $\ell$-adic representation of $\mathscr{G}_{\Qp}$ (with $\ell \neq p$) being unramified. For the Galois representation attached to an elliptic curve $E$ defined over $\Q$, this amounts to asking that $E$ has good reduction at $p$. An extension of these results in the case of bad reduction can be found in \cite{RodriguesPhiGamma}.}, the Coleman map has been generalised by Perrin-Riou \cite{PR}. Under some choices, she constructed \emph{big logarithm maps}
\[ \mathrm{Log}_V : H^1_{\rm Iw}(\qp, V) \to \mathscr{D}^{\mathrm{la}}(\Zp^\times, L), \] 
where $\mathscr{D}^{\mathrm{la}}(\Zp^\times, L)$ denotes the space of $L$-valued  locally analytic distributions on $\Zp^\times$ (in the sense of \S\ref{sec:locally analytic}). The map $\mathrm{Log}_V$ satisfies certain interpolation properties expressed in terms of Bloch-Kato's exponential and dual exponential maps and, for $V = \qp(1)$, we recover $\mathrm{Col}'$.

The general idea is that, given an Euler system for a global $p$-adic Galois representation, localising it at the place $p$ and applying Perrin-Riou's map, one can construct a $p$-adic $L$-function for $V$. In a diagram:
\[ \big\{ \text{Euler systems} \big\} \xrightarrow{\mathrm{loc}_p} H^1_{\rm Iw}(\qp, V) \xrightarrow{\mathrm{Log}_V} \big\{ p\text{-adic $L$-functions} \big\}. \] This splits the problem of constructing $p$-adic $L$-functions for motives into a global problem (finding an Euler system) and a purely local problem (constructing the big logarithm maps). See \cite{ColmezFonctL} for further references on this subject.

%}

%%%%%%%%%%%%%%%%%%%%%=============================================================

\section{Iwasawa's theorem on the zeros of the $p$-adic zeta function}\label{sec:iwasawa zeros}

In the previous section, the Coleman map allowed us to give a construction of the Kubota--Leopoldt $p$-adic $L$-function $\zeta_p$ using a specific tower of cyclotomic units. We now describe a theorem of Iwasawa  (Theorem \ref{thm:iwasawa}) that puts this on a deeper footing. This theorem describes the zeros of $\zeta_p$ -- captured by a canonically attached ideal in the Iwasawa algebra -- in terms of arithmetic data, via the \emph{module} of cyclotomic units inside the local units. The Coleman map from \S\ref{sec:coleman map} will be the key step for connecting both worlds.

 With the aim of moving all the analytic information to the Galois side, we will start by reformulating the definition of the $p$-adic zeta function as a pseudo-measure on the Galois group $\GG = \Gal(F_\infty/\Q) \cong \Zp^\times$. We then introduce the global and local modules of cyclotomic units (which will be systematically studied later), stating the connection to class numbers, and state Iwasawa's theorem.

%This section is devoted to the study of a theorem of Iwasawa (Theorem \ref{thm:iwasawa}) relating the zeros of the $p$-adic $L$-function in terms of local arithmetic data. We have already seen that the $p$-adic zeta function is closely related to local units, which is the first key step for connecting both worlds. With the aim to moving all the analytic information to the Galois side, we will start by reformulating the definition of the $p$-adic zeta function as a pseudo-measure on a Galois group.

%In Part I, we constructed the Kubota--Leopoldt $p$-adic $L$-function $\zeta_p$ as a pseudo-meaure on $\Zp^\times$. From an arithmetic perspective, it is natural to reinterpet $\zeta_p$ via measures on Galois groups, and we make this change of perspective here.

%To anchor the next few sections, we also introduce the cyclotomic units, and recall their close connection to an arithmetic staple: class groups of number fields. We state a theorem of Iwasawa, which concretely connects cyclotomic units to the Kubota--Leopoldt $p$-adic $L$-function and ultimately motivated his formulation of the Main Conjecture. The next few sections will be dedicated to carefully motivating and proving this theorem. 

\subsection{Measures on Galois groups} \label{sec:measures on galois groups}

%\CWnote{Move: In the previous section, the fact that we can use the cyclotomic character to see the $p$-adic zeta function in terms of measures on $\GG$ was heavily trailed, and we even took the Coleman map to have values in $\Lambda(\GG)$. In the process, we gave a conceptual reason for the twist by 1 introduced in the definition of $\zeta_p$, which ensured that the Coleman map was $\GG$-equivariant. We now elaborate on this and further conceptually explain why this twist is introduced in the context of the Main Conjecture. }

%In the process, we pin down the normalisations around $\zeta_p$ that we will be using for the remainder of these notes.\\

Recall that $F_\infty = \cup_{n \geq 1} \Q(\mu_{p^n})$, that $\GG = \mathrm{Gal}(F_\infty / \Q)$, and that the cyclotomic character gives an isomorphism $\chi : \GG \isorightarrow \zpe$. This isomorphism induces an identification between measures on $\Zp^\times$ and measures on the Galois group $\GG$. From now on, we will let $\Lambda(\GG)$ be the space of measures on $\GG$, which we identify with $\Lambda(\Zp^\times)$ via the cyclotomic character. We may thus naturally consider $\zeta_p$ as a pseudo-measure on $\GG$.

Similarly, the Galois group $\GG^+ = \mathrm{Gal}(F_\infty^+ / \Q) = \GG / \langle c \rangle$ is identified through the cyclotomic character with $\zpe / \{ \pm 1 \}$. Observe that $\zeta_p$, which ostensibly is an element of $Q(\GG)$, vanishes at the characters $\chi^k$, for any odd integer $k > 1$. We will use this fact to show that $\zeta_p$ actually descends to a pseudo-measure on $\GG^+$.

\begin{lemma} \label{lem:decompose plus minus}
	Let $c \in \GG$ denote complex conjugation. Let $R$ be a ring in which $2$ is invertible and $M$ an $R$-module with a continuous action of $\GG$. Then $M$ decomposes as
	\[M \cong M^+ \oplus M^-,\]
	where $c$ acts as $+1$ on $M^+$ and as $-1$ on $M^-$.
\end{lemma}

\begin{proof}
	This follows directly by using the idempotents $\frac{1+c}{2}$ and $\frac{1-c}{2}$, which act as projectors to the corresponding $M^+$ and $M^-$.
\end{proof}

We are assuming that $p$ is odd, so $\Lambda(\GG) \cong \Lambda(\GG)^+ \oplus \Lambda(\GG)^-$. In fact, the module $\Lambda(\GG)^+$ admits a description solely in terms of the quotient $\GG^+$.

\begin{lemma}
	There is a natural isomorphism
	\[\Lambda(\GG)^+ \cong \Lambda(\GG^+).\]
\end{lemma}

\begin{proof}
	We work at finite level. Let $\GG_n \defeq \Gal(F_n/\Q)$, and $\GG_n^+ \defeq \Gal(F_n^+/\Q)$. Then there is a natural surjection
	\[\Zp[\GG_n] \to \Zp[\GG_n^+] \]
	induced by the natural quotient map on Galois groups. Since this must necessarily map $\Zp[\GG_n]^-$ to 0, this induces a map $\Zp[\GG_n]^+ \rightarrow \Zp[\GG_n^+]$. The result now follows at finite level by a dimension count (as both are free $\Zp$-modules of rank $(p-1)p^{n-1}/2$,  and one sees easily that the map sends a basis of the first module to a basis of the second). We obtain the required result by passing to the inverse limit.
\end{proof}

We henceforth freely identify $\Lambda(\GG^+)$ with the submodule $\Lambda(\GG)^+$ of $\Lambda(\GG)$.

\begin{lemma}
	Let $\mu \in \Lambda(\GG)$. Then $\mu \in \Lambda(\GG^+)$ if and only if
	\[\int_{\GG}\chi(x)^k \cdot\mu = 0\]
	for all odd $k \geq 1$.
\end{lemma}

\begin{proof}
By Lemma \ref{lem:decompose plus minus}, we can write $\mu = \mu^+ + \mu^-$, where $\mu^\pm = \frac{1 \pm c}{2} \mu$. We want to show that $\mu^- = 0$ if and only if $ \int_{\GG}\chi(x)^k \cdot\mu = 0$ for all odd $k \geq 1$. Since $\chi(c) = -1$, we have
\[ \int_{\GG}\chi(x)^k \cdot \mu^+ = \frac{1}{2} \bigg( \int_\GG \chi^k \cdot \mu +  (-1)^k \int_\GG \chi^k \cdot \mu \bigg). \]
If $k$ is odd, the above expression vanishes, showing that $\int_\Gamma \chi(x)^k \cdot \mu = \int_\Gamma \chi(k) \cdot \mu^+$ for all odd $k$. On the other hand, the same argument shows that $\int_\Gamma \chi(x)^k \cdot \mu^-$ vanishes for all $k$ even. The result follows then by Lemma \ref{lem:zero divisor}.
\end{proof}

\begin{corollary}
	The $p$-adic zeta function is a pseudo-measure on $\GG^+$.
\end{corollary}

\begin{proof}
	This follows from the interpolation property, as $\zeta(1-k) = 0$ for odd $k\geq 1$.
\end{proof}

\subsection{The ideal generated by the $p$-adic zeta function}

It is natural to ask about the zeros of the $p$-adic zeta function. Since the zeros are not modified if we multiply by a unit, studying the zeros of a measure on $\GG$ is equivalent to studying the ideal in $\Lambda(\GG)$ generated by the measure. 

Even though Kubota--Leopoldt is only a pseudo-measure -- hence not an element of $\Lambda(\GG)$ -- we now see that it still `generates' a natural ideal in $\Lambda(\GG)$. By definition of pseudo-measures, the elements $([g] - [1]) \zeta_p$ belong to the Iwasawa algebra $\Lambda(\GG)$ for any  $g\in \GG$. Recall from Definition \ref{DefAugmentationIdealFiniteLevel} that $I(\GG)$ denotes the \emph{augmentation ideal} of $\Lambda(\GG)$, that is, the ideal
\[ I(\GG) = \ker(\Lambda(\GG) \to \Zp), \]
where $\Lambda(\GG) \twoheadrightarrow \Zp$ is the map induced by $[g] \mapsto 1$ for any $\sigma \in \GG$. We define $I(\GG^+)$ similarly.

\begin{proposition}
	The module $I(\GG)\zeta_p$ is an ideal in $\Lambda(\GG)$. Similarly, the module $ I(\GG^+)\zeta_p$ is an ideal in $\Lambda(\GG^+)$.
\end{proposition}
\begin{proof}
	Since $\zeta_p$ is a pseudo-measure, we know $([g]-[1])\zeta_p \in \Lambda(\GG)$ for all $g \in \GG$. Hence the result follows as $I(\GG)$ is the topological ideal generated by the elements $[g] - [1]$ for $g \in \GG$. The same argument holds for $I(\GG^+)\zeta_p$.
\end{proof}

\subsection{Cyclotomic units and Iwasawa's theorem}

Iwasawa's theorem describes the ideal $I(\GG)\zeta_p$ in terms of the module of cyclotomic units. We now recall this module, and its classical connection to class numbers, and then state Iwasawa's theorem.  %and in particular study the subgroup generated by them in the (local and global) unit groups. In the global case, this subgroup has finite index in the whole unit group $\mathscr{V}_n$. Since the determination of the units of a number field is in general a difficult problem, and cyclotomic units provide a partial answer in the case of cyclotomic fields, they are objects of classical interest and have been extensively studied.

\begin{definition}
	For $n \geq 1$, we define the group $\mathscr{D}_n$ of cyclotomic units of $F_n$ to be the intersection of $\mathscr{O}_{F_n}^\times$ and the multiplicative subgroup of $F_n^\times$ generated by $\{ \pm \xi_{p^n}, \xi_{p^n}^a - 1 \, : \, 1 \leq a \leq p^n - 1 \}$. We set $\mathscr{D}_n^+ = \mathscr{D}_n \cap F_n^+$.
\end{definition}

We will study the structure of cyclotomic units more in detail in subsequent sections. The following result shows their connection to class numbers.

\begin{theorem} \label{thm:cyclo units class number}
	Let $n \geq 1$. The group $\mathscr{D}_n$ (resp. $\mathscr{D}_n^+$) is of finite index in the group of units $\mathscr{V}_n$ (resp.\ $\mathscr{V}_n^+$) in $F_n$ (resp. $F_n^+$), and we have
	\[ 
	h_n^+ = [ \mathscr{V}_n : \mathscr{D}_n ] = [ \mathscr{V}_n^+ : \mathscr{D}_n^+ ],
	 \]
  where $h_n^+ \defeq \# \mathrm{Cl}(F_n^+)$ is the class number of $F_n^+$.
\end{theorem}

\begin{proof}
	We will not prove this here; see \cite[Theorem 8.2]{Washington2}. The proof goes by showing that the regulator of cyclotomic units is given in terms of special $L$-values at $s = 1$ of Dirichlet $L$-functions, and then using the class number formula. 
\end{proof}

As we explained in \S\ref{sec:coleman map definition}, the construction of the $p$-adic zeta function via the Coleman map goes as follows. The cyclotomic units $c_n(a)$, introduced in \S\ref{sec:cyclo units coleman}, are naturally elements of $\DD_n$, hence global. One then considers their image inside the space of local units, and then applies the Coleman map (Definition \ref{def:coleman map}), which is a purely local procedure. In this spirit it is natural to switch here from studying the global modules $\DD_n$ and $\DD_n^+$ to their closures in the space of local units. Recall $\UU_{\infty,1}^+$ from the notational introduction to Part II; it is the group of norm-compatible local units congruent to $1 \newmod{p}$.

%The cyclotomic units $c_n(a)$, introduced in \S\ref{sec:cyclo units coleman}, are actually elements of $\DD_n$, hence global. To use them in Theorem \ref{thm:coleman to kl}, it was necessary to study their image in the local units, and in norm-compatible towers. In this spirit it is natural to switch here from studying the global modules $\DD_n$ and $\DD_n^+$ to their closures in the space of local units. Recall $\UU_{\infty,1}^+$ from the notational introduction to Part II; it is the group of norm-compatible local units congruent to $1 \newmod{p}$. Also:

%In the same spirit as \S\ref{sec:pr log}, where the construction of the $p$-adic $L$-function takes as input the image of global objects (Euler systems) in a local space (Iwasawa cohomology), Iwasawa's theorem is a local avatar of a global phenomenon. To state it, we must switch from studying the global modules $\DD_n$ and $\DD_n^+$ to their closures in the space of local units. Recall $\UU_{\infty,1}^+$ from the notational introduction to Part II; it is the group of norm-compatible local units congruent to $1 \newmod{p}$. Also:

\begin{definition} For any $n \geq 1$, define $\CC_n$ as the $p$-adic closure of $\DD_n$ inside the local units $\UU_n$\footnote{ We will describe this closure more explicitly in Lemma \ref{lem:closure} below.}, let $\CC^+_n \defeq \CC_n \cap \UU_n^+$, and let
	\[ \CC_{n, 1} \defeq \CC_n \cap \UU_{n, 1}, \;\;\; \CC^+_{n, 1} \defeq \CC^+_n \cap \UU_{n, 1}; \]
	\[ \CC_{\infty, 1} \defeq \varprojlim_{n \geq 1} \CC_{n, 1}, \;\;\; \CC^+_{\infty, 1} \defeq \varprojlim_{n \geq 1} \CC^+_{n, 1}. \]

\end{definition}

We will see that $\UU_{\infty,1}^+$, and its quotient $\UU_{\infty,1}^+/\CC_{\infty,1}^+$, naturally have $\Lambda(\GG^+)$-module structures. Moreover, Iwasawa explicitly related this quotient to the $p$-adic zeta function. The following theorem, which we prove in \S\ref{sec:proof Iwasawa}, says that the cyclotomic units capture the zeros of $\zeta_p$ and ultimately motivated Iwasawa to formulate his Main Conjecture.

\begin{theorem} \label{thm:iwasawa}
The Coleman map induces an isomorphism of $\Lambda(\GG^+)$-modules
		\[ 
		\mathscr{U}^+_{\infty, 1} / \mathscr{C}^+_{\infty, 1} \xrightarrow{\sim} \Lambda(\GG^+) / I(\GG^+) \zeta_p. 
		\]
\end{theorem}

The quotient $\mathscr{U}_{\infty,1}^+/\CC_{\infty,1}^+$ is a local analogue, at infinite level, of the cyclotomic units inside the global units, whose indices compute class numbers in the cylotomic tower (Theorem \ref{thm:cyclo units class number}).  They will in turn be related (cf.\ Corollary \ref{cor:CFTunits2}) to the Galois modules appearing in the formulation of the Iwasawa Main Conjecture. This theorem is hence the first step into proving a remarkable and deep connection between class groups and the $p$-adic zeta function, which will be the main purpose of the Iwasawa Main Conjecture.

%%%%%%%%%%%%%%%%%%%%%=============================================================

\section{Proof of Iwasawa's theorem}\label{sec:proof Iwasawa}

In this section, we prove Theorem \ref{thm:iwasawa}. First, we equip the local units with an action of $\Lambda(\GG)$, and prove that the Coleman map is equivariant with respect to this action. Then, in Theorem \ref{thm:fund exact seq}, we compute the kernel and cokernel of the Coleman map. Finally we describe generators of the modules of cyclotomic units, and compute their image under the Coleman map. We combine all of this to prove Theorem \ref{thm:iwasawa}.

\subsection{Equivariance properties of the Coleman map}\label{sec:equivariance}

Theorem \ref{thm:iwasawa} is a statement about $\Lambda(\GG^+)$-modules. Here it is important that we work over the full Iwasawa algebra; the structure theorem for modules over $\Lambda(\GG)$ and $\Lambda(\GG^+)$ -- stated in Theorem \ref{thm:structure theorem} below -- is crucial in studying the Iwasawa Main Conjecture. It is desirable, then, to equip $\sU_\infty$ with a $\Lambda(\GG)$-module structure. As $\Lambda(\GG)$ is the completed group ring of $\GG$ over $\Zp$, this amounts to equipping it with compatible actions of $\Zp$ and $\GG$. For the latter, we use the natural Galois action on the local units. For the former, however, we are stuck: whilst there is a natural action of $\Z$ on $\sU_\infty$ by $u \mapsto u^a$ for an integer $a$, this does not extend to an action of $\Zp$.

\subsubsection{The action of $\Zp$}

To fix the absence of a $\Zp$-action on local units, we recall the definition of the subgroup $\sU_{\infty,1} \subset \sU_\infty$ introduced in \eqref{eq:Uinfty 1}. In particular, we showed there that the action of $\Z$ \emph{does} extend to $\Zp$ on $\sU_{\infty,1}$. %The next series of lemmas shows that $\mathrm{Col}'$ descends to a map on $\sU_{\infty,1}$ with no loss of information.
For convenience, we recall (cf.\ Definition \ref{def:coleman map}) that the Coleman map was defined as the following composition
\[
\mathrm{Col} :	\sU_\infty \xrightarrow{u \mapsto f_u(T)} (\Zp\lsem T\rsem ^\times)^{\cN=\mathrm{id}} \xrightarrow {\partial\log} \Zp\lsem T\rsem  \xrightarrow{1-\varphi\circ\psi} \Zp\lsem T\rsem ^{\psi = 0}\xrightarrow{\partial^{-1}}\Zp\lsem T\rsem ^{\psi = 0} \xrightarrow{\sA^{-1}} \Lambda(\Zp^\times).
\]

\begin{proposition}
	The map $\mathrm{Col}$ restricts to a $\Zp$-equivariant map
	\[
	\mathrm{Col} : \sU_{\infty,1} \longrightarrow \Lambda(\Zp^\times).
	\]
\end{proposition}
\begin{proof}
	It suffices to check $\Zp$-equivariance for each map in the composition in Definition \ref{def:coleman map}. 	The action of $a \in \Zp$ on $u \in \sU_{\infty,1}$ is given by $u \mapsto u^a \defeq \sum_{k \geq 0} {a \choose k} (u-1)^k$. Write $f_u = \sum_{k \geq 1}a_k(u) T^k$. We first claim that 
	\begin{equation}\label{eq:1 mod p}
	a_0(u) \equiv 1 \newmod{p}.
	\end{equation}
	Indeed, by definition $f_u(\pi_n) = u_n \equiv 1 \newmod{\pri_n}$ for each $n$, and as $\pi_n$ is a uniformiser for $K_n$, we see 
	\[
	f_u(\pi_n) = a_0(u) + \sum_{k\geq 1}a_k(u) \pi_n^k \ \ \ \in a_0(u) + \pri_n,
	\]
	from which we see that $a_0(u) \equiv 1 \newmod{\pri_n}$. But $a_0(u)$ lies in $\Zp$, giving \eqref{eq:1 mod p}.
	
	Thus $f_u(T) - 1 \in (p,T)$. As $\Zp\lsem T\rsem $ is complete in the $(p,T)$-adic topology,
	\[
	f_u(T)^a = \sum_{j \geq 0}\binomc{a}{j} (f_u(T)-1)^j
	\]
	converges to a power series $f_u^a(T) \in \Zp\lsem T\rsem $. Since by construction $f_u(\pi_n)^a = u_n^a$, by Lemma \ref{lem:unique coleman} we have, 
	\[
	f_{u}^a = f_{u^a} \in (\Zp\lsem T\rsem ^\times)^{\cN=\mathrm{id}}.
	\]
	As a result, we have equipped the image of $\sU_{\infty,1}$ inside $\Zp\lsem T\rsem$ under the map $u \mapsto f_u$ with a $\Zp$-action such that the restriction of the Coleman isomorphism is $\Zp$-equivariant. We compute that $\partial\log(f_u^a) = a\partial\log(f_u)$, so $\partial\log$ is equivariant for the natural $\Zp$-action on $\Zp\lsem T\rsem $. Finally the maps $(1- \varphi\circ\psi)$, $\partial^{-1}$ and $\sA^{-1}$ are $\Zp$-equivariant by definition.
\end{proof}

The next two lemmas show that we have not lost any information by restricting.
\begin{lemma}
	We have $\sU_\infty = \mu_{p-1} \times \sU_{\infty,1}.$
\end{lemma}
\begin{proof}
	We start at finite level $n$. As $p$ is totally ramified in $K_n$ for all $n$, there is a unique prime $\pri_n$ of $K_n$ above $p$, and reduction modulo $\pri_n$ gives a short exact sequence
	\[
	1 \to \sU_{n,1} \to \sU_{n} \to \mu_{p-1} \to 1,
	\]
	which is split, so $\sU_{n} = \mu_{p-1} \times \sU_{n,1}$. The result follows in the inverse limit.
\end{proof}

\begin{lemma}
	The subgroup $\mu_{p-1}$ of $\sU_{\infty}$ is killed by $\mathrm{Col}$. In particular, no information is lost when restricting to $\sU_{\infty,1}$.
\end{lemma}
\begin{proof}
	Note $\mu_{p-1} \subset \Zp^\times$. The first map $u \mapsto f_u$ is an isomorphism that sends $v = (v)_{n \in \N} \in \mu_{p-1} \subset \UU_\infty$ to the constant power series $f_v(T) = v$. But constant power series are killed by the second map $\Zp\lsem T\rsem  \xrightarrow{\partial\log}\Zp\lsem T\rsem$, which involves differentiation. Thus $\mu_{p-1}$ is mapped to zero under the composition, and hence under $\mathrm{Col}$.
\end{proof}
\begin{remark}\label{rem:ker Delta}
	The kernel of $\partial \log$ is comprised of constant power series. Moreover, if $f \in \Zp\lsem T\rsem $ is constant and invariant under $\cN$, then this forces $f^p = f$. Thus the kernel of the composition of the first two maps is exactly $\mu_{p-1}$.
\end{remark}

\subsubsection{The Galois action}

	The Galois group $\GG = \Gal(F_\infty/\Q)$ is naturally isomorphic to $\mathrm{Gal}(K_\infty/\Qp)$, as $p$ is totally ramified in $F_\infty$. Thus $\Gamma$ acts on $\sU_\infty$.

%The following proposition shows why we do not define the Coleman map to be the map $\mathrm{Col}'$ defined above. In particular, as defined it $\mathrm{Col}'$ is \emph{not} $\GG$-equivariant, and hence does not induce a map of $\Lambda(\GG)$-modules on $\sU_{\infty,1}$. 

\begin{notation}
	For $a \in \Zp^\times$, let $\sigma_a \in \GG$ be the corresponding element of $\GG$ with $\chi(\sigma_a) = a$, recalling that $\chi : \GG \isorightarrow \Zp^\times$ is the cyclotomic character from \eqref{eq:cyclo char}.
\end{notation}

\begin{proposition}
	The Coleman map $\mathrm{Col} : \sU_{\infty} \to \Lambda(\GG)$ is $\GG$-equivariant. 
\end{proposition}
\begin{proof}
	We must show that if $a \in \Zp^\times$, and $u \in \sU_{\infty}$, we have
	\[
	\mathrm{Col}(\sigma_a(u)) = \sigma_a(\mathrm{Col}(u)).
	\]
	This is easy to check if we understand how $\GG$ acts on each of the modules involved. If $u = (u_n)_{n \geq 1} \in \UU_\infty$, then 
	\[
	\sigma_a(u) = (\sigma_a(u_n))_{n \geq 1} \in \UU_\infty,
	\]
	and if $f(T) \in \zp\lsem T\rsem $, then 
	\[\sigma_a(f)(T) = f\big((1 + T)^a - 1\big).\]
	Then:
	\begin{itemize}
		\item We have
		\begin{align*}
			(\sigma_a f_u)(\pi_n) &= f_u((1 + \pi_n)^a - 1) \\
			&= f_u(\xi_{p^n}^a - 1) \\
			& = f_u(\sigma_a(\xi_{p^n} - 1)) \\
			&= \sigma_a(f_u(\xi_{p^n}-1)) = \sigma_a(u_n),
		\end{align*}
		so that $u \mapsto f_u(T)$ is $\GG$-equivariant.
		\item If $f(T) \in \zp\lsem T\rsem ^\times$, then an easy calculation on power series shows that
		\begin{equation}\label{eq:dlog equivariant}
		\partial\log(\sigma_a(f)) = a \sigma_a(\partial\log(f)).
		\end{equation}
		\item On measures, restriction to $\Zp^\times$ is $\GG$-equivariant since the action of $\sigma_a$ is by multiplying the variable by $a \in \Zp^\times$, which obviously stabilises both $\Zp^\times$ and $p\Zp$.
		\item 	As operations on $\Zp\lsem T\rsem ^{\psi = 0}$, we have 
		\begin{equation}\label{eq:delta equivariant}
		\partial^{-1} \circ \sigma_a = a^{-1} \sigma_a \circ \partial^{-1},
		\end{equation}
		as is easily checked on measures. Indeed, 
		\begin{align*}
			\int_{\Zp^\times} f(x) \cdot \partial^{-1}\sigma_a \mu &= \int_{\Zp^\times}\frac{f(x)}{x} \cdot \sigma_a \mu\\
			&= \int_{\Zp^\times}\frac{f(ax)}{ax} \cdot \mu \\
			&= a^{-1} \int_{\Zp^\times}f(ax) \cdot \partial^{-1}\mu\\
			&= a^{-1} \int_{\Zp^\times} f(x) \cdot \sigma_a\partial^{-1} \mu.
		\end{align*}
		\item By definition of the action, the inverse Mahler transform $\sA^{-1}$ is equivariant under $\sigma_a$.
	\end{itemize}
	Putting all that together, the result follows.
\end{proof}

Now, the $\GG$-action on $\UU_\infty$ fixes $1 \in \mu_{p-1}$, so it stabilises the subspace $\sU_{\infty,1}$. This action commutes with the $\Zp$-action on $\sU_{\infty,1}$. We deduce that $\sU_{\infty,1}$ is a $\Lambda(\GG)$-module. The results of \S\ref{sec:equivariance} can then be summarised as follows.

\begin{corollary}\label{cor:G-eq}
	The map $\mathrm{Col}$ restricts to a map $\sU_{\infty,1} \to \Lambda(\GG)$ of $\Lambda(\GG)$-modules.
\end{corollary}

\begin{remark}\label{rem:renormalise}
	In the construction of $\zeta_p$, we renormalised by `dividing by $x$' (in \S\ref{sec:dep on a}). This appears here via $\partial^{-1}$. We see from \eqref{eq:delta equivariant} that $\partial^{-1}$ really is essential for the Coleman map to be $\GG$-equivariant, motivating the appearance of $x^{-1}$ in \S\ref{sec:dep on a}.  Conceptually, $\zeta$ and $\zeta_p$ are the $L$-function and $p$-adic $L$-function of the trivial Galois representation $\Qp$, whilst the cyclotomic units in $\UU_\infty$ form an Euler system for its twist $\Qp(1)$; the $\partial^{-1}$ bridges between these two Galois representations.
\end{remark}

\subsection{The fundamental exact sequence}
Theorem \ref{thm:iwasawa} says that the Coleman map induces an isomorphism $\UU_{\infty,1}^+/\CC_{\infty,1}^+ \cong \Lambda(\GG^+)/I(\GG^+)\zeta_p$. To prove this, we must study the kernel and cokernel of the Coleman map. We do so here (in Theorem \ref{thm:fund exact seq}) via a careful study of each of its constituent maps.

\subsubsection{The logarithmic derivative}

\label{sec:logarithmic der}

We will now show that the logarithmic derivative translates norm-invariance into trace-invariance (recalling the trace operator $\psi$). The key result is Theorem \ref{thm:log der}. For convenience of notation, and consistency with \cite{CS06}, we make the following definition.

\begin{definition}
	For $f(T) \in \zp\lsem T\rsem ^\times$, define its logarithmic derivative as
	\begin{align*} 
		\Delta(f) &\defeq \partial\log f
		=\frac{\partial f(T)}{f(T)} = (1 + T) \frac{f'(T)}{f(T)}. 
	\end{align*}
\end{definition}

The main result of this section is the following.

\begin{theorem} \label{thm:log der}
The logarithmic derivative   induces a short exact sequence
\[ 0 \to \mu_{p-1} \to \big( \zp\lsem T\rsem ^\times \big)^{\cN=\mathrm{id}} \xrightarrow{\ \ \Delta\ \ } \zp\lsem T\rsem ^{\psi=\mathrm{id}} \to 0. \]
%a surjection
%	\[ \Delta : \big( \zp\lsem T\rsem ^\times \big)^{\cN=\mathrm{id}} \to \zp\lsem T\rsem ^{\psi=\mathrm{id}} \] with kernel $\mu_{p - 1}$.
\end{theorem}

%The difficulty in proving Theorem \ref{thm:log der} lies in the fact that the module $\zp\lsem T\rsem ^{\psi=\mathrm{id}}$ admits no simple description.
 We described the kernel of $\Delta$ in Remark \ref{rem:ker Delta} above, so it suffices to deal with its image. We first prove that this image is contained in $\zp\lsem T\rsem ^{\psi=\mathrm{id}}$  (Lemma \ref{lem:log der 1}). We then reduce the proof of surjectivity, via Lemma \ref{lem:log der red mod p}, to surjectivity modulo $p$. Finally, in Lemma \ref{lem:A mod p} and Lemma \ref{lem:B mod p} we calculate the reduction modulo $p$ of both spaces. 

For convenience, let $\WW \defeq \big( \zp\lsem T\rsem ^\times \big)^{\cN=\mathrm{id}}$.

\begin{lemma} \label{lem:log der 1}
	We have $\Delta(\WW) \subseteq \zp\lsem T\rsem ^{\psi=\mathrm{id}}$.% and the kernel of $\Delta$ on $\WW$ is $\mu_{p - 1}$.
\end{lemma}

\begin{proof}
	If $f \in \WW$, then
	\[ 
	\varphi(f) = (\varphi \circ \cN)(f) = \prod_{\eta \in \mu_p} f((1 + T)\eta - 1). 
	\]
	Applying $\Delta$ to the above equality and using the fact  that $\Delta \circ \varphi = p \, \varphi \circ \Delta$ (which is easy to see on power series from the definitions), we obtain
	\[ 
 (\varphi \circ \Delta)(f) = p^{-1} \sum_{\eta \in \mu_p} \Delta(f)((1 + T)\eta - 1) = (\varphi \circ \psi)(\Delta(f)).
 \] 
 By injectivity of $\varphi$, we deduce $\psi(\Delta(f)) = \Delta(f)$.
\end{proof}

We move now to the proof of surjectivity. In the following, let
\[ A = \overline{\Delta(\WW)} \subseteq \fp\lsem T\rsem ; \;\;\; B = \overline{\zp\lsem T\rsem ^{\psi=\mathrm{id}}} \subseteq \fp\lsem T\rsem  \]
be the reduction modulo $p$ of the modules we need to compare.

\begin{lemma} \label{lem:log der red mod p}
	If $A = B$, then $\Delta(\WW) = \zp\lsem T\rsem ^{\psi=\mathrm{id}}$.
\end{lemma}

\begin{proof}
	Let $f_0 \in \zp\lsem T\rsem ^{\psi=\mathrm{id}}$. By hypothesis, there exists a $g_1 \in \WW$ such that $\Delta(g_1) - f_0 = p f_1$ for some $f_1 \in \zp\lsem T\rsem $. Since $\Delta(\WW) \subseteq \zp\lsem T\rsem ^{\psi=\mathrm{id}}$ by Lemma \ref{lem:log der 1}, we see that $\psi$ fixes both $\Delta(g_1)$ and $f_0$ and hence, by additivity, $\psi$ fixes $f_1$; so again by hypothesis, there exists some $g_2 \in \WW$ such that $\Delta(g_2) - f_1 = p f_2$ for some $f_2 \in \zp\lsem T\rsem $. By induction, we induce the existence of $g_i \in \WW$ and $f_i \in \zp\lsem T\rsem ^{\psi=\mathrm{id}}$, $i \geq 1$, such that \[ \Delta(g_i) - f_{i - 1} = p f_i. \] 
	%Since $\Delta(a) = 0$ for any $a \in \zpe$ and since $\psi$ is linear, we can assume that $g_i(0) \equiv 1 \text{ (mod } p)$ for all $i \geq 1$.
	
    Now let 
	\[h_n = \prod_{k = 1}^{n}  g_k^{(-1)^{k-1}p^{k-1}} \in \WW,\]
	i.e. 
	\[
		h_1 = g_1, \ \ \ \ h_2 = \frac{g_1}{g_2^p}, \ \ \ \ h_3 = \frac{g_1\cdot g_3^{p^2}}{g_2^p}, \ \ \ \ h_4 = \frac{g_1 \cdot g_3^{p^2}}{g_2^p \cdot g_4^{p^3}}, 
	\]
	etc. As $\Delta$ transforms multiplication into addition, we have
	\begin{align*}
		\Delta(h_n) &= \Delta(g_1) - p\Delta(g_2) +  \cdots + (-1)^{n-1}p^{n-1}\Delta(g_n)\\
		&= (f_0 + pf_1) - (pf_1 + p^2f_2) + \cdots + (-1)^{n-1}(p^{n-1}f_{n-1} + p^nf_n)\\
		&= f_0 + (-1)^{n-1} p^{n} f_n.
	\end{align*}
	By compactness, the sequence $(h_n)_{n \geq 1}$ admits a convergent subsequence converging to an element $h \in \WW$ satisfying $\Delta(h) = f_0$, which shows the result.
\end{proof}

%The following lemma calculates the reduction modulo $p$ of $\WW$.

\begin{lemma} \label{lem:A mod p}
	We have $\overline{\WW}\defeq \WW\newmod{p} = \fp\lsem T\rsem ^\times$.
\end{lemma}

\begin{proof}
	The inclusion $\subset$ is obvious. Conversely, for any element $f \in \fp\lsem T\rsem ^\times$, lift it to an element $\tilde{f}_0 \in \zp\lsem T\rsem ^\times$. By points (ii) and (iv) of Lemma \ref{lem:norm continuity}, the sequence $\cN^k(\tilde{f}_0)$ converges to an element $\tilde{f}$ that is invariant under $\cN$ and whose reduction modulo $p$ is $f$. 
\end{proof}

The most delicate and technical part of the proof of Theorem \ref{thm:log der} is contained in the following two lemmas describing the reduction of $\zp\lsem T\rsem ^{\psi=\mathrm{id}}$ modulo $p$.

\begin{lemma} \label{lem:B mod p}
	We have $B = \Delta(\fp\lsem T\rsem ^\times)$.
\end{lemma}

\begin{proof}
	We have $\Delta(\WW) \subseteq \zp\lsem T\rsem ^{\psi=\mathrm{id}}$ by Lemma \ref{lem:log der 1}, so the inclusion $\Delta(\fp\lsem T\rsem ^\times) \subset B$ is clear using Lemma \ref{lem:A mod p}. For the other inclusion, take any $f \in B$ and use Lemma \ref{lem:B mod p 2} below to write
	\[ f = \Delta(a) + b \] 
	for some $a \in \fp\lsem T\rsem ^\times$ and $b = \sum_{m = 1}^{+\infty} d_m \frac{T + 1}{T} T^{pm}$. Since $\psi(f) = f$ and $\psi(\Delta(a)) = \Delta(a)$ (by a slight abuse of notation, as $f$ and $\Delta(a)$ are actually the reduction modulo $p$ of elements fixed by $\psi$), we deduce that $\psi(b) = b$. But we can explicitly calculate the action of $\psi$ on $b$. Using the identity\footnote{Again, this can be easily checked on measures.} $\psi(g\cdot\varphi(f)) = \psi(g) f$, the identity $T^{pm} = \varphi(T^m)$ in $\fp\lsem T\rsem $ and the fact that $\psi$ fixes $\frac{T + 1}{T}$ (cf.\ the proof of Lemma \ref{LemmaPsiInvariant}), we deduce that 
	\[ \psi(b) = \sum_{m = 1}^{+\infty} d_m \frac{T + 1}{T} T^m, \]
	which immediately implies $b = 0$ and concludes the proof.
\end{proof}

\begin{lemma} \label{lem:B mod p 2}
	We have
	\[ \fp\lsem T\rsem  = \Delta\Big(\fp\lsem T\rsem ^\times\Big) + \frac{T + 1}{T} C, \] 
	where $C = \big\{ \sum_{n = 1}^{+ \infty} a_n T^{pn} \big\} \subseteq \fp\lsem T\rsem $.
\end{lemma}

\begin{proof}
	One inclusion is clear. Take $g \in \fp\lsem T\rsem $ and write $\frac{T}{T + 1} g = \sum_{n = 1}^{+ \infty} a_n T^n$. Define
	\[ h = \sum_{\substack{m = 1 \\ (m, p) = 1}}^{+\infty} a_m \sum_{k = 0}^{+\infty} T^{m p^k}. \] Clearly $\frac{T}{T+1} g - h \in C$, so it suffices to show that $\frac{T + 1}{T} h \in \Delta(\fp\lsem T\rsem ^\times)$.
	Indeed, we will show by induction that, for every $m \geq 1$, there exists $\alpha_i \in \fp$ for $1 \leq i < m$ such that
	\[ h_m \defeq \frac{T + 1}{T} h - \bigg( \sum_{i = 1}^{m - 1} \Delta(1 - \alpha_i T^i) \bigg) \in T^{m-1} \fp\lsem T\rsem . \] The case $m = 1$ is empty. Suppose that the claim is true for $m$ and that $\alpha_1, \hdots, \alpha_{m - 1}$ have been chosen. Observe first that
	\[ \Delta(1 - \alpha_i T^i) =  - \frac{T + 1}{T} \sum_{k = 1}^{+\infty} i \alpha_i^k T^{ik}, \] so we can write 
	\[ h_m = \frac{T + 1}{T} \sum_{k = m}^{+ \infty} d_k T^k. \] Observe that, by construction of $h$ and $h_m$, we have $d_n = d_{np}$ for all $n$. If $d_m = 0$ then we set $\alpha_m = 0$. If $d_m \neq 0$ then, by what we have just remarked, $m$ must be prime to $p$, hence invertible in $\fp$, and we set $\alpha_m = - \frac{d_m}{m}$. One can then check that
	\[ g = \prod_{n = 1}^{+\infty} (1 - \alpha_n T^n) \in \fp\lsem T\rsem  \] satisfies $\Delta(g) = \frac{T + 1}{T} h$, which concludes the proof.
\end{proof}

We can now complete the proof of  Theorem \ref{thm:log der}. 

\begin{proof} [{Proof of Theorem \ref{thm:log der}}]
By Lemma \ref{lem:log der 1}, the map $\Delta$ is well-defined, and its kernel is $\mu_p$ by Remark \ref{rem:ker Delta}. It remains to prove surjectivity. By Lemma \ref{lem:log der red mod p}, it suffices to prove that $A = B$, which follows directly from Lemma \ref{lem:A mod p} and Lemma \ref{lem:B mod p}.
\end{proof}

\subsubsection{The fundamental exact sequence}

Finally, we study the fundamental exact sequence describing the kernel and cokernel of the Coleman map. The only remaining map to study is $1-\varphi\circ\psi$. By Theorem \ref{thm:log der}, it suffices to study this on $\Zp\lsem T\rsem^{\psi=\mathrm{id}}$. 

\begin{lemma} \label{lem:rest zp*}
	There is an exact sequence
	\[ 0 \to \zp \to \zp\lsem T\rsem ^{\psi=\mathrm{id}} \xrightarrow{\ 1 - \varphi\ } \zp\lsem T\rsem ^{\psi = 0} \to \zp \to 0, \]
	where the first map is the natural inclusion and the last map is evaluation at $T = 0$.
\end{lemma}

\begin{proof}
	Injectivity of the first map is trivial. To see surjectivity of the last map, note that $\psi(1 + T) = 0$, since 
	\[(\varphi \circ \psi)(1 + T) = p^{-1} \sum_{\eta \in \mu_p} \eta (1 + T) = 0.\]
	Thus $1+T \in \Zp\lsem T\rsem^{\psi=0}$ and is mapped to 1 under the last map.
 
 Let $f(T) \in \zp\lsem T\rsem ^{\psi = 0}$ be in the kernel of the last map, that is be such that $f(0) = 0$. Then $\varphi^n(f)$ goes to zero (in the weak/$(p,T)$-adic topology) and hence $\sum_{n \geq 0} \varphi^n(f)$ converges to an element $g(T)$ whose image under $(1 - \varphi)$ is $f(T)$. Since $\psi\circ \varphi = \mathrm{id}$, we also have
	\[ \psi(g) = \sum_{n\geq 0}\psi \circ \varphi^n(f) = \psi(f) + \sum_{n\geq 1} \varphi^{n-1}(f) = g, \]
	as $\psi(f) = 0$, which shows that 
	\[f \in (1-\varphi)\left(\Zp\lsem T\rsem ^{\psi=\mathrm{id}}\right)\]
	and hence that the sequence is exact at $\Zp\lsem T\rsem ^{\psi=0}.$ Finally, if $f(T) \in \zp\lsem T\rsem $ is not constant, then $f(T) = a_0 + a_r T^{r} + \hdots $ for some $a_r \neq 0$ and $\varphi(f)(T) = a_0 + p a_r T^r + \hdots \neq f(T)$, which shows that $\ker(1 - \varphi) = \zp$ and finishes the proof.
\end{proof}

\begin{definition}\label{def:Zp(1)}
	Let $\Zp(1) \defeq \varprojlim \mu_{p^n}$, the module $\Zp$ with an action of $\GG$ by $\sigma \cdot x = \chi(\sigma)x,$ recalling $\chi$ is the cyclotomic character. This is an integral version of $\Qp(1)$.
\end{definition}

\begin{theorem} \label{thm:fund exact seq}
	The Coleman map induces an exact sequence of $\GG$-modules
	\[ 0 \to \mu_{p - 1} \times \zp(1) \longrightarrow \mathscr{U}_\infty \xrightarrow{\mathrm{Col}} \Lambda(\GG) \longrightarrow \zp(1) \to 0, \] where the last map sends $\mu \in \Lambda(\GG)$ to $\int_{\GG} \chi \cdot \mu$. In particular, it induces an exact sequence
	\[
	0 \to \Zp(1) \longrightarrow \sU_{\infty,1} \xrightarrow{\mathrm{Col}} \Lambda(\GG) \longrightarrow \Zp(1) \to 0
	\]
	of $\Lambda(\GG)$-modules.
\end{theorem}

\begin{proof}
We first compute the kernel of the Coleman map, which we recall from Definition \ref{def:coleman map} is the composition 
\begin{align*}
\mathrm{Col} :	\sU_\infty \xrightarrow{\ u \mapsto f_u(T)\ } (\Zp\lsem T\rsem ^\times)^{\cN=\mathrm{id}} \xrightarrow {\ \ \Delta\ \ } &\Zp\lsem T\rsem ^{\psi=\mathrm{id}} \xrightarrow{\ 1-\varphi\circ\psi\ } \Zp\lsem T\rsem ^{\psi = 0}\\
&\xrightarrow{\ \ \partial^{-1}\ \ }\Zp\lsem T\rsem ^{\psi = 0} \xrightarrow{\ \ \sA^{-1}\ \ } \Lambda(\Zp^\times).
\end{align*}
In the third term, we have used Theorem \ref{thm:log der} to replace $\Zp\lsem T\rsem$ with $\Zp\lsem T \rsem^{\psi=\mathrm{id}}$.

The first map is an isomorphism by Theorem \ref{thm:coleman map 2}. By Theorem \ref{thm:log der}, the second map surjects with kernel $\mu_{p - 1}$. By Lemma \ref{lem:rest zp*} the third map has kernel $\zp$; this is the image of $\{(1 + T)^a : a \in \zp\}$ under $\Delta$. This is the power series interpolating the sequence $(\xi_{p^n}^a)_{n \geq 1}$. Accordingly, when we pull this back to $\UU_\infty$, we get the factor
	\[\Zp(1) = \{(\xi_{p^n}^a)_n : a \in \Zp\} \subset \UU_\infty.\]
Finally, the fourth and fifth maps are isomorphisms, so ultimately the kernel of $\mathrm{Col}$ is as claimed. 

We now compute the cokernel. The first two and last two maps in $\mathrm{Col}$ are surjective, and the third map has cokernel $\zp$ by Lemma \ref{lem:rest zp*}, showing the exactness of the sequence.
	
	Finally, we turn to the $\GG$-equivariance. The subspace $\mu_{p-1} \times \Zp(1) \subset \sU_\infty$ is preserved by $\GG$, so the first map is $\GG$-equivariant. That $\mathrm{Col}$ is $\GG$-equivariant was Corollary \ref{cor:G-eq}. The last map is $\GG$-equivariant since
	\begin{align*}
		\int_{\GG} \chi(x) \cdot \sigma \mu(x) &= \int_{\GG} \chi(\sigma x) \cdot \mu(x)
		= \chi(\sigma)\int_{\GG}\chi \cdot \mu,
	\end{align*}
	and $\GG$ acts on $\Zp(1)$ through the cyclotomic character $\chi$.
\end{proof}

\subsection{Generators for the global cyclotomic units}

 Recall the (global) module of cyclotomic units $\DD_n$ is the intersection of $\mathscr{O}_{F_n}^\times$ with the multiplicative subgroup of $F_n^\times$ generated by $\pm \xi_{p^n}$ and $\xi_{p^n}^a-1$ with $1 \leq a < p^n$, and $\DD_n^+ = \DD_n \cap F_n^+$. We now show this module is cyclic over the group ring $\Z[\Gamma_n^+]$. This is essential preparation for treating the local analogue $\CC_{n,1}^+$ in the next subsection.

  Recall that we defined
\[c_n(a) \defeq \frac{\xi_{p^n}^a - 1}{\xi_{p^n} - 1} \in \DD_n,\]
and note that
\[\gamma_{n,a} \defeq \xi_{p^n}^{(1-a)/2}c_n(a) = \frac{\xi_{p^n}^{a/2} - \xi_{p^n}^{-a/2}}{\xi_{p^n}^{1/2} - \xi_{p^n}^{-1/2}}\]
is fixed by conjugation $c \in \GG$, hence gives an element of $\DD_n^+$. In fact:

\begin{lemma} \label{lem:cyc units gen}
	Let $n \geq 1$. Then
	\begin{itemize}
		\item[(i)] The group $\mathscr{D}_n^+$ is generated by $-1$ and 
		\[ \bigg\{ \gamma_{n,a}  \; : \; 1 < a < \frac{p^n}{2}, (a,p) = 1 \bigg\}. \]
		\item[(ii)] The group $\mathscr{D}_n$ is generated by $\xi_{p^n}$ and $\mathscr{D}_n^+$.
	\end{itemize}
\end{lemma}

\begin{proof}
	We first show that we need only consider those elements $\xi_{p^n}^a - 1$ with $a$ prime to $p$. Indeed, this follows from the identity
	\[ \xi_{p^n}^{b p^m} - 1 = \prod_{j = 0}^{p^m - 1} (\xi_{p^n}^{b+jp^{n - m}} - 1), \] where $(b, p) = 1$ and $1 \leq m < n$, noting that $b+jp^{n - m}$ is prime to $p$. Also, since $\xi_{p^n}^a - 1 = -\xi_{p^n}^{a} (\xi_{p^n}^{-a} - 1)$, we can restrict to considering $1 \leq a < \frac{1}{2} p^n$ (recall here that $p$ is odd).
	
	Suppose that 
	\[
	 \gamma = \pm \xi_{p^n}^d \prod_{\substack{1 \leq a < \frac{1}{2}p^n \\ (a, p) = 1}} (\xi_{p^n}^a - 1)^{e_a} \in \mathscr{D}_n, 
	 \] for some integers $d$ and $e_a$. Since $v_p(\xi_{p^n}^d) = 0$ and all the $p$-adic valuations of $\xi_{p^n}^a - 1$ coincide (namely, $v_p(\xi_{p^n}^a - 1) = \frac{1}{(p - 1)p^{n - 1}}$), we deduce that $\sum_a e_a = 0$. Therefore we can write
	\[
	 \gamma = \pm \xi_{p^n}^d \prod_a \bigg( \frac{ \xi_{p^n}^a - 1}{\xi_{p^n} - 1} \bigg)^{e_a} = \pm \xi_{p^n}^e \prod_a \gamma_{n,a}^{e_a}, 
	 \] 
	 where $e = d + \frac{1}{2} \sum_a e_a (a - 1)$. This shows the second point; the first point follows by observing that every term $\gamma_{n,a}^{e_a}$ of the product is real, so $\gamma \in \mathscr{D}_n^+$ if and only if $e = 0$.
\end{proof}

\begin{corollary} \label{cor:cyc units gen 2}
If $a$ generates $(\Z / p^n \Z)^\times$, then $\gamma_{n,a}$ generates $\DD^+_{n}$ as a $\Z[\GG_n^+]$-module. 
\end{corollary}

\begin{proof}
	If $1 \leq b < p^n$ is prime to $p$, then $b \equiv a^r$ (mod $p^n$) for some $r \geq 0$, and hence
	 \[ \gamma_{n,b} = \frac{\xi_{p^n}^{a^r}-1}{\xi_{p^n}-1} = \prod_{i=0}^{r-1}\frac{\xi_{p^n}^{a^{i+1}}-1}{\xi_{p^n}^{a^i}-1} =  \prod_{i = 0}^{r - 1} (\gamma_{n,a})^{\sigma_a^i}. \qedhere
	 \]
\end{proof}

\subsection{Generators for the local cyclotomic units}

We now move to analogous results for the local cyclotomic units, that is, the $p$-adic closure of the global cyclotomic units.

A natural guess might be that `as $\DD_n^+$ is generated by $\gamma_{n,a}$ as a $\Z[\Gamma_n^+]$-module, then $\CC_n^+$ is generated by $\gamma_{n,a}$ as a $\Zp[\Gamma_n^+]$-module'. However, this is nonsense, as $\CC_n^+$ is not a $\Zp$-module; to parse this, we must pass to the principal cyclotomic units $\CC_{n,1}^+$. The following simple lemma then describes precisely the $p$-adic closure.

\begin{lemma}\label{lem:closure}
Let $g_1,\dots,g_r \in \UU_{n,1}$, and let $X = \langle g_1,\dots,g_r\rangle \subset \UU_{n,1}$ be the $\Z$-module they generate (multiplicatively\footnote{By which we mean: the $\Z$-module structure is given by exponentiation, i.e.\ $(a,g) \mapsto g^a$ for $a \in \Z$ (or $\zp$). As is standard, but perhaps confusingly, we sometimes use additive notation $ag$ for this module structure.}). Then the $p$-adic closure $\overline{X}$ of $X$ in $\UU_{n,1}$ is the $\Zp$-submodule of $\UU_{n,1}$ generated by $g_1,\dots,g_r$.
\end{lemma}

\begin{proof}
Let $a \in \Zp$, and $a_j$ be any sequence of integers tending to $a \in \Zp$. Then 
\begin{equation}\label{eq:converge Zp}
    g_i^{a_j} = \sum_{k \geq 0}\binomc{a_j}{k} (g_i-1)^k \longrightarrow \sum_{k\geq 0}\binomc{a}{k}(g_i-1)^k = g_i^a,
\end{equation}
 since $g_i - 1 \equiv 0 \newmod{\pri_n}$ (as $g_i \in \UU_{n,1}$). Thus $g_i^a$ lies in the $p$-adic closure. An identical argument shows more generally that the $\Zp$-span of the $g_i$'s is a subset of the $p$-adic closure.

To see the converse, let $g \in \overline{X}$. Then by definition there exists a sequence $(g_1^{a_{1,j}}\cdots g_r^{a_{r,j}})_j \subset X$ tending to $g$, with each $a_{i,j} \in \Z$. This gives a sequence $(a_{1,j},\dots,a_{r,j})_j \subset \Z^r \subset \Zp^r$. By compactness of $\Zp^r$, there exists a convergent subsequence  
\[
    (b_{1,k},\dots,b_{r,k})_k \to (b_1, \dots, b_r) \in \Zp^r.
\]
Then $(g_1^{b_{1, k}} \cdots g_r^{b_{r, k}})_k$ converges to $g_1^{b_1} \cdots g_r^{b_r}$ by \eqref{eq:converge Zp}; but by construction it also converges to $g$. Thus $g = g_1^{b_1}\cdots g_r^{b_r}$, showing that $g$ lies in the $\Zp$-span of the $g_i$'s, as required.
\end{proof}

We now put Corollary \ref{cor:cyc units gen 2} into a form useful for Lemma \ref{lem:closure}.

\begin{lemma}\label{lem:global generators 2}
Let $a \in \Z$ be a topological generator of $\zpe$, and $w \in \mu_{p-1} \subset \UU_n$ be such that $aw \equiv 1 \newmod{\mathfrak{p}_n}$. Then:
\begin{itemize}
\item[(i)] $w\gamma_{n,a} \in \UU_{n,1}$, 
\item[(ii)] $(w\gamma_{n,a})^{p-1} = \gamma_{n,a}^{p-1}$ lies in $\UU_{n,1}^+$, and generates the cyclic $\Z[\Gamma_n^+]$-module
\[
    \Z[\Gamma_n^+]\cdot (w\gamma_{n,a})^{p-1} = (p-1)\mathscr{D}_{n}^+ = \{\gamma^{p-1} : \gamma \in \mathscr{D}_{n}^+\} \subset \UU_{n,1}^+.
\]
\end{itemize}
\end{lemma}

\begin{proof}
(i) We first claim that $\gamma_{n,a} \equiv a \newmod{\mathfrak{p}_n}$. Indeed, by definition $\gamma_{n,a} = \xi_{p^n}^{a/2} c_n(a)$. Recall $\pi_n = \xi_{p^n} - 1$ is a uniformiser in $\mathfrak{p}_n$, so  $\xi_{p^n}^{a/2} \equiv 1 \newmod{\mathfrak{p}_n}$. It thus suffices to show 
\[
    c_n(a) \equiv a \newmod{\mathfrak{p}_n}.
\]
For this, recall that by construction of the Coleman power series attached to any unit $u = (u_n) \in \sU_\infty$, we have 
\[
    u_n = f_u(\pi_n) \equiv f_u(0) \newmod{\mathfrak{p}_n}.
\]
For $u = c(a)$, recall we have $f_{c(a)} = ((1+T)^a - 1)/T$; so $c_n(a) \equiv f_{c(a)}(0) = a \newmod{\mathfrak{p}_n}$, proving the claim. Thus $w$ is the unique element of $\mu_{p-1}$ such that $w\gamma_{n,a} \equiv 1 \newmod{\mathfrak{p}_n}$, and hence $w\gamma_{n,a} \in \mathscr{U}_{n,1}$, as required.

(ii) By Corollary \ref{cor:cyc units gen 2}, we know $\gamma_{n,a}$ generates $\DD_n^+$, and deduce $\gamma_{n,a}^{p-1}$ generates $(p-1)\DD_{n}^+$. In particular $\gamma_{n,a}^{p-1}$ lies in $\UU_{n,1}^+$. As $w^{p-1} = 1$, we have $\gamma_{n,a}^{p-1}= (w\gamma_{n,a})^{p-1}$, giving (ii).
\end{proof}

\begin{lemma} \label{LemmaGeneratorCinfty1}
Let $a \in \Z$ be a topological generator of $\zpe$, and $w \in \mu_{p-1} \subset \UU_n$ be such that $aw \equiv 1 \newmod{\mathfrak{p}_n}$. Then:
\begin{itemize}

\item[(i)] The module $\CC_{n,1}^+$ is a cyclic $\zp[\GG^+_n]$-module generated by $w\gamma_{n,a}$.

\item[(ii)] 	The module $\CC^+_{\infty, 1}$ is a cyclic $\Lambda(\GG^+)$-module generated by $(w \gamma_{n,a})_{n \geq 1}$.%, where $a \in \Z$ is a topological generator of $\zpe$ and $w \in \mu_{p - 1} \subseteq \Qp$ is such that $aw \equiv 1 \;(mod \; p)$.
 \end{itemize}
\end{lemma}

\begin{proof}
(i) By Lemma \ref{lem:global generators 2}(ii), 
\[
    (p-1)\mathscr{D}_{n}^+ = (p-1)\mathscr{D}_{n,1}^+ \subset \UU_{n,1}^+
\]
is generated as a $\Z[\Gamma^+_n]$-module by $(w\gamma_{n,a})^{p-1}$.  By Lemma \ref{lem:closure} the $p$-adic closure $(p-1)\CC_{n,1}^+$ of $(p-1)\mathscr{D}_{n,1}^+$ is generated as a $\Zp[\Gamma^+_n]$-module by $(w\gamma_{n,a})^{p-1}$. As $p-1$ is invertible in $\Zp$, we conclude $(p-1)\CC_{n,1}^+ = \CC_{n,1}^+$ is generated by $w\gamma_{n,a}$. We use here that $w\gamma_{n,a} \equiv 1 \newmod{\mathfrak{p}_n}$, so that $w\gamma_{n,a}$ is the unique $(p-1)$-th root of $(w\gamma_{n,a})^{p-1}$ lying in $\CC_{n,1}^+$.

(ii) Observe that
\[
    \CC_{\infty,1}^+ \cong \varprojlim \CC_{n,1}^+ = \varprojlim \Big(\Zp[\Gamma_n^+]\cdot w\gamma_{n,a}\Big)\cong \Lambda(\Gamma^+)\cdot (w\gamma_{n,a})_n,
\]
as required, with all maps as $\Lambda(\Gamma^+)$-modules and where the middle equality is (i).
\end{proof}

\subsection{End of the proof}

Finally we can prove Iwasawa's Theorem \ref{thm:iwasawa}.

\begin{theorem} \label{thm:iwasawa 2}
	The Coleman map induces:
	\begin{itemize}
 \item[(i)] A short exact sequence of $\Lambda(\GG)$-modules
		\[ 0 \to \mathscr{U}_{\infty, 1} / \mathscr{C}_{\infty, 1} \to \Lambda(\GG) / I(\GG) \zeta_p \to \zp(1) \to 0. \]
		\item[(ii)] An isomorphism of $\Lambda(\GG^+)$-modules
		\[ \mathscr{U}^+_{\infty, 1} / \mathscr{C}^+_{\infty, 1} \xrightarrow{\sim} \Lambda(\GG^+) / I(\GG^+) \zeta_p. \]
	\end{itemize}
\end{theorem}

\begin{proof}
	Theorem \ref{thm:fund exact seq} gave an exact sequence of $\Lambda(\GG)$-modules 
	\[ 
	0 \to \zp(1) \longrightarrow \mathscr{U}_{\infty, 1} \xrightarrow{\mathrm{Col}} \Lambda(\GG) \longrightarrow \zp(1) \to 0. 
	\]
	The theorem will follow by calculating the image of the modules $\CC_{\infty, 1}$ and $\CC^+_{\infty, 1}$ under the Coleman map. By Lemma \ref{LemmaGeneratorCinfty1}, it suffices to calculate the image under $\mathrm{Col}$ of an element $(\xi_{p^n}^b \gamma_{n,a})_{n \geq 1} \in \UU_{\infty, 1}$, for $a, b \in \zpe$. But this has already been done:  by Theorem \ref{thm:coleman to kl}, and the fact that $\xi_{p^n}^b$ lies in the kernel of the Coleman map, we know that 
	\[ 
	\mathrm{Col} \big( (\xi_{p^n}^b \gamma_{n,a})_{n \geq 1} \big) = \mathrm{Col} \left( \xi_{p^n}^{-(1-a)/2}(\gamma_{n,a})_{n \geq 1} \right) = \mathrm{Col}\big( c(a) \big) = ([\sigma_a] - [1]) \zeta_p, 
	\]
	where as usual $\sigma_a$ denotes an element of $\GG$ such that $\chi(\sigma_a) = a.$ Since $a \in \zpe$ was arbitrary, we conclude that the image of $\CC_{\infty, 1}$ (resp. $\CC^+_{\infty, 1}$) under $\mathrm{Col}$ is $I(\GG) \zeta_p$ (resp. $I(\GG^+) \zeta_p$). We deduce an exact sequence
	\[ 0 \to \UU_{\infty, 1} / \CC_{\infty, 1} \longrightarrow \Lambda(\GG) / I(\GG) \zeta_p \longrightarrow \zp(1) \to 0. \] 
	This shows (i). Since $p$ is odd, taking invariants under the group $\langle c \rangle \subset \GG$ of order two generated by complex conjugation is exact. As $c$ acts on $\Zp(1)$ by $-1$, we see that $\zp(1)^{\langle c \rangle} = 0$, which shows (ii) and concludes the proof of the theorem.
\end{proof}

%%%%%%%%%%%%%%%%%%%%%=============================================================
%%%%%%%%%%%%%%%%%%%%%=============================================================

\section{The Iwasawa Main Conjecture}\label{sec:IMC}

 We now move from arithmetic to algebra. To state the Iwasawa Main Conjecture, we use the structure theory of $\Lambda$-modules. We first summarise this theory. We then define modules from the Galois theory of abelian extensions that will be needed. These modules carry an action of the Galois group $\GG = \mathrm{Gal}(F_\infty/\Q) \cong \Zp^\times$, and hence obtain the structure of $\Lambda(\GG)$-modules. The Iwasawa Main Conjecture describes the characteristic ideal of one of these Galois modules in terms of the Kubota--Leopoldt $p$-adic $L$-function.
 
 In the interests of space, where necessary, we state without proof relevant auxiliary results.

\subsection{Structure theory for $\Lambda$-modules}\label{sec:structure theorems}

There is a rich structure theory of modules over Iwasawa algebras, which looks similar to that of modules over PIDs. Here we state (without proof) some basic yet fundamental results. 

Let $L$ be a finite extension of $\Qp$, $\mathscr{O}_L$ its ring of integers and let $\Lambda \defeq \Lambda(\zp) = \varprojlim \roi_L[\Zp/p^n\Zp] \cong \roi_L\lsem T\rsem$ be the Iwasawa algebra of $\Zp$ over $\roi_L$. Let $M, M'$ be two $\Lambda$-modules. We say that $M$ is pseudo-isomorphic to $M'$, and we write $M \sim M'$, if there exists a homomorphism $M \rightarrow M'$ with finite kernel and cokernel, i.e, if there is an exact sequence
 \[
  0 \to A \to M \to M' \to B \to 0, 
  \] 
 with $A$ and $B$ finite $\Lambda$-modules (just in case: $A$ and $B$ have finite cardinality!). We remark that $\sim$ is \emph{not} an equivalence relation (see \cite[Warning, \S13.2]{Washington2}) but it \emph{is} an equivalence relation between \emph{finitely generated, torsion} $\Lambda$-modules. The following is the main result concerning the structure theory of finitely generated $\Lambda$-modules.

\begin{theorem} \label{thm:structure theorem} \cite[Thm.\ 13.12]{Washington2}. Let $M$ be a finitely generated $\Lambda$-module. Then
	\[
	 M \sim \Lambda^r \oplus \bigg( \bigoplus_{i = 1}^s \Lambda / (p^{n_i}) \bigg) \oplus \bigg( \bigoplus_{j = 1}^t \Lambda / (f_j(T)^{m_j}) \bigg), 
	 \]
	for some $r, s, t \geq 0$, $n_i, m_j \geq 1$ and irreducible distinguished polynomials $f_j(T) \in \mathscr{O}[T]$.
\end{theorem}

	Here we call a polynomial $P(T) \in \roi_L[T]$ \emph{distinguished} if $P(T) = a_0 + a_1 T + \hdots + a_{n-1} T^{n - 1} + T^n$ with $a_i \in \mathfrak{p}$ for every $0 \leq i \leq n - 1$.

\begin{remark}
	We do \emph{not} have a similar result for the finite level group algebras $\roi_L[\Zp/p^n\Zp]$, only for the projective limit. This is another major example of the fundamental concept of Iwasawa theory, where it is profitable to study a whole tower of objects all in one go, rather than individually at finite level. %, and use results about the tower to deduce results at finite level., and a major reason for this is this structure theorem.
\end{remark}

\begin{definition}
	Suppose $M$ is a finitely generated torsion $\Lambda$-module. Then $r =0$ in the structure theorem. We define the \emph{characteristic ideal} of $M$ to be the ideal
	\[ 
	\mathrm{Ch}_{\Lambda}(M) = \left(p^{n}\right)\prod_{j=1}^t \left(f_j^{m_j}\right) \subset \Lambda, \qquad \text{where } n = \sum_{i = 1}^s n_i.
	\]
\end{definition}

We will apply this theory more generally. Suppose $\GG = H \times \Gamma'$, where $H$ is a finite commutative group of order prime to $p$ and $\Gamma' \cong \zp$.  A key example to have in mind is $\GG = \zpe$, $H = \mu_{p-1}$ and $\Gamma' = 1 + p \zp$. Then we have a decomposition
\[ 
\Lambda(\GG) \cong \roi_L[H] \otimes \Lambda. 
\] 
Let $M$ be a finitely generated torsion $\Lambda(\GG)$-module. Let $H^\wedge$ denote the group of characters of $H$ and define, for any $\omega \in H^\wedge$, the projector to the isotypic component
\[ e_\omega \defeq \frac{1}{|H|} \sum_{a \in H} \omega^{-1}(a) [a] \in \roi_L[H], \]
possibly after extending $L$ by adjoining the values of $\omega$. As the order of $H$ is prime to $p$, one can easily show the following result.

\begin{lemma}[{\cite[A.1]{CS06}}]
	The group $H$ acts on $M^{(\omega)} \defeq e_\omega M$ via multiplication by $\omega$ and we have a decomposition of $\Lambda(\GG)$-modules
	\[ M = \oplus_{\omega \in H^\wedge} M^{(\omega)}. \] Moreover, each $M^{(\omega)}$ is a finitely generated torsion $\Lambda$-module.
\end{lemma}

\begin{definition}
	Let $\GG = H \times \Zp$ be as above and let $M$ be a finitely generated torsion $\Lambda(\GG)$-module. We define the \emph{characteristic ideal} of $M$ to be the ideal
	\[ \mathrm{Ch}_{\Lambda(\GG)}(M) \defeq \oplus_{\omega \in H^\wedge} \mathrm{Ch}_\Lambda(M^{(\omega)}) \subseteq \Lambda(\GG). \] 
\end{definition}

\begin{lemma}[{\cite[A.1 Prop.\ 1]{CS06}}]
	The characteristic ideal is multiplicative in exact sequences.
\end{lemma}

\subsection{The $\Lambda$-modules arising from Galois theory} \label{SecLambdamods}

The following $\Lambda$-modules will be the protagonists of the Galois side of the Main Conjecture; we urge the reader to refer back to the following definitions as these objects appear in the text. Recall $F_n = \Q(\mu_{p^n})$, that $\pri_n$ is the unique prime above $p$ in $F_n$, and other similar notation from \S\ref{sec:coleman map notation}. Then define
\[ 
\mathscr{M}_n \defeq \text{maximal abelian $p$-extension of $F_n$ unramified outside $\pri_n$}, \]
\[ \mathscr{M}^+_n \defeq \text{maximal abelian $p$-extension of $F^+_n$ unramified outside $\pri_n^+$}, \]
\[ \mathscr{L}_n \defeq \text{maximal unramified abelian $p$-extension of $F_n$}, \]
\[ \mathscr{L}^+_n \defeq \text{maximal unramified abelian $p$-extension of $F^+_n$},
\]
and set
\[ 
\mathscr{M}_\infty \defeq \cup_{n \geq 1} \mathscr{M}_n = \text{maximal abelian pro-$p$-extension of $F_\infty$ unramified outside $\mathfrak{p}$}, \]
\[ \mathscr{M}^+_\infty \defeq \cup_{n \geq 1} \mathscr{M}^+_n = \text{maximal abelian pro-$p$-extension of $F^+_\infty$ unramified outside $\mathfrak{p}^+$}, \]
\[ \mathscr{L}_\infty \defeq \cup_{n \geq 1} \mathscr{L}_n = \text{maximal unramified abelian  pro-$p$-extension of $F_\infty$}, \]
\[ \mathscr{L}^+_\infty \defeq \cup_{n \geq 1} \mathscr{L}^{+}_n = \text{maximal unramified abelian pro-$p$-extension of $F^+_\infty$}.
 \]
Finally, define
\[
 \mathscr{X}_\infty \defeq \mathrm{Gal}(\mathscr{M}_\infty / F_\infty), \;\;\; \mathscr{X}^+_\infty = \mathrm{Gal}(\mathscr{M}^+_\infty / F^+_\infty), \]
\[ \mathscr{Y}_\infty \defeq \mathrm{Gal}(\mathscr{L}_\infty / F_\infty), \;\;\; \mathscr{Y}^+_\infty = \mathrm{Gal}(\mathscr{L}^+_\infty / F^+_\infty). 
\]
These modules fit into the following diagram of field extensions:
%% DIAGRAM OF FIELD EXTENSIONS

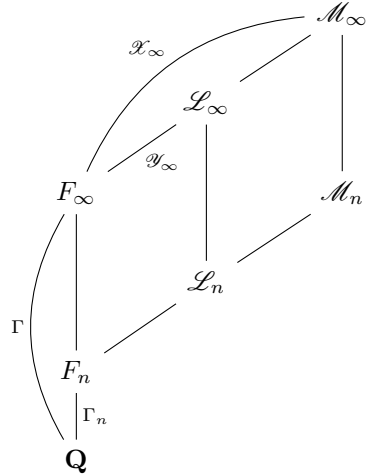
\[
\begin{tikzcd}[every arrow/.append style={dash}]
&                                & \mathscr{M}_\infty \\
& \mathscr{L}_\infty\arrow[ur]   & \\
F_\infty \arrow[ur,"\! \mathscr{Y}_\infty"'] \arrow[uurr, bend left, "\mathscr{X}_\infty"] & & \mathscr{M}_n\arrow[uu]\\
&\mathscr{L}_n \arrow[uu]\arrow[ur]                  &\\
F_n \arrow[uu]\arrow[ur]       &                                 & \\
\Q \arrow[u, "\GG_n"']\arrow[uuu, bend left, "\GG"]       &                          &
\end{tikzcd}
\]

There is an identical diagram for the totally real objects, with superscripts ${}^+$ everywhere.

\begin{rem}\label{rem:inner automorphisms}
In the limit, the module $\XX_\infty$ is a $\Lambda(\GG)$-module, and thus can be studied via Theorem \ref{thm:structure theorem}. To describe this, let $x \in \mathscr{X}_\infty$ and $\sigma \in \GG$, and choose any lifting $\tilde{\sigma} \in \mathrm{Gal}(\mathscr{M}_\infty / \Q)$ of $\sigma$; then 
\[
	\sigma \cdot x \defeq \tilde{\sigma} x \tilde{\sigma}^{-1}
\]
gives a well defined action of $\GG$ on $\mathscr{X}_\infty$. As $\mathscr{O}_L[\GG]$ is dense in $\Lambda(\GG)$, and the latter is Hausdorff, this action extends by linearity and continuity to an action of $\Lambda(\GG)$ on $\mathscr{X}_\infty$. In exactly the same way we define actions of $\Lambda(\GG)$ on $\mathscr{Y}_\infty$ and of $\Lambda(\GG^+)$ on $\mathscr{X}^+_\infty$ and $\mathscr{Y}^+_\infty$. 
\end{rem}

\subsection{The Main Conjecture}

Recall the ideal $I(\GG^+)\zeta_p \subset \Lambda(\GG^+)$, and that this encodes the zeros of $\zeta_p$. We already gave an arithmetic description of this ideal in Theorem \ref{thm:iwasawa} in terms of cyclotomic units. The Iwasawa Main Conjecture upgrades this to the following:

\begin{theorem} [Iwasawa Main Conjecture] \label{IMC} $\mathscr{X}^+_\infty$ is a finitely generated torsion $\Lambda(\GG^+)$-module, and 
\[ 
\mathrm{ch}_{\Lambda(\GG^+)}(\mathscr{X}^+_\infty) = I(\GG^+) \zeta_p. 
\]
\end{theorem}

\begin{rem} \leavevmode \label{rem:even v odd}
It is usual in the literature to formulate the Iwasawa Main Conjecture in terms of an even Dirichlet character of $\mathrm{Gal}(\Q(\mu_p) / \Q)$. As one can already observe from the behaviour of the Bernoulli numbers, there exists a certain dichotomy involving the parity of this character which makes the formulation of the Main Conjecture different in the even and odd cases. The above formulation takes into account every such even Dirichlet character. For a formulation of the Main Conjecture for odd Dirichlet characters, see \cite{HK}.
\end{rem}

\subsection{The Iwasawa Main Conjecture for Vandiver primes}\label{sec:vandiver proof}

\begin{definition}
	Let $h_n^+ \defeq \#\mathrm{Cl}(F_n^+)$ be the class number of $F_n^+$.  We say $p$ is a \emph{Vandiver prime} if $p \nmid h_1^+$.
\end{definition}

The rest of these notes are dedicated to the following theorem of Iwasawa.

\begin{theorem}\label{thm:vandiver}
If $p$ is a Vandiver prime, we have an isomorphism of $\Lambda(\GG^+)$-modules
\[ \mathscr{X}^+_\infty \cong \Lambda(\GG^+) / I(\GG^+) \zeta_p. \] In particular, Iwasawa's Main Conjecture holds.
\end{theorem}

The arguments of this section form the origins of Iwasawa's formulation of the Main Conjecture, and give further motivation for it. As our main goal is the study of $p$-adic $L$-functions, we omit the proofs of some more classical auxiliary results. Our approach follows that of \cite[\S4.5]{CS06}, which we suggest the reader consults for a more detailed exposition.

We first use class field theory to reinterpret Theorem \ref{thm:iwasawa 2} in terms of some modules arising from Galois theory.

\begin{definition} For any $n \geq 1$, define $\EE_n$ as the $p$-adic closure of the global units $\VV_n = \mathscr{O}_{F_n}^\times$ inside the local units $\UU_n$, let $\EE^+_n \defeq \EE_n \cap \UU_n^+$, and let
\[ \EE_{n, 1} \defeq \EE_n \cap \UU_{n, 1}, \;\;\; \EE^+_{n, 1} \defeq \EE^+_n \cap \UU_{n, 1}; \]
\[ \EE_{\infty, 1} \defeq \varprojlim_{n \geq 1} \EE_{n, 1}, \;\;\; \EE^+_{\infty, 1} \defeq \varprojlim_{n \geq 1} \EE^+_{n, 1}. \]
\end{definition}

 Leopoldt's conjecture (which is known in this case by a theorem of Brumer) states that for any $n \geq 1$, the group $\mathscr{E}_n$ is a finite $\zp$-module of rank $r_1 + r_2 - 1 = p^{n-1} (p-1)/2$. Here as usual $r_1$  (resp.\ $r_2$) denotes the number of real  (resp.\ half the number of complex) embeddings of $F_n$ into $\C$. In light of Lemma \ref{lem:closure}, the conjecture says that if global units are (multiplicatively) independent over $\Z$, then they are independent over $\Zp$.

%The following results connect units in the cyclotomic tower and modules coming from Galois theory.

\begin{proposition} \label{prop:CFTunits1}
There is an exact sequence of $\Lambda(\GG^+)$-modules
\[ 
0 \to \EE^+_{\infty, 1} \to \UU^+_{\infty, 1} \to \mathrm{Gal}(\MM^+_\infty / \LL^+_\infty) \to 0. 
\]
\end{proposition}

\begin{proof}
Global class field theory (see \cite[Cor.\ 13.6]{Washington2}) gives a short exact sequence
\begin{equation}\label{eq:cft} 
0 \to \EE^+_{n, 1} \to \UU^+_{n, 1} \to \mathrm{Gal}(\MM^+_n / \LL^+_n) \to 0. 
\end{equation}
Taking the inverse limit over $n$ gives the result. This is exact as all modules in the sequence above are finitely generated $\zp$-modules (so satisfy the Mittag-Leffler condition).
\end{proof}

We now rewrite the terms in this sequence, moving it closer to Theorem \ref{thm:vandiver}.  Motivated by Theorem \ref{thm:iwasawa}, we also introduce $\CC_{\infty,1}^+$ into the picture. Then:

\begin{corollary} \label{cor:CFTunits2} We have an exact sequence of $\Lambda(G)$-modules
\[ 0\to \EE^+_{\infty, 1} / \CC^+_{\infty, 1} \to \UU^+_{\infty, 1} / \CC^+_{\infty, 1} \to \mathscr{X}^+_\infty \to \mathscr{Y}^+_\infty \to 0.\]
\end{corollary}

\begin{proof}
The fundamental theorem of Galois theory yields a short exact sequence
\[
    0 \to \mathrm{Gal}(\MM_\infty^+/\LL_\infty^+) \to \XX_\infty^+ \to \YY_\infty^+ \to 0.
\]
The result now follows from Proposition \ref{prop:CFTunits1}, as
\[
    \Gal(\MM_\infty^+/\LL_\infty^+) \cong \EE_{\infty,1}^+/\UU_{\infty,1}^+ \cong \big(\EE_{\infty,1}^+/\CC_{\infty,1}^+\big)/\big(\UU_{\infty,1}^+/\CC_{\infty,1}^+\big),
\]
the last identification being the third isomorphism theorem.
\end{proof}

Key to the proof of Theorem \ref{thm:vandiver} is the following result from classical Iwasawa theory. For the sake of completeness we will give in Appendix \ref{sec:mu invariant} an introduction to this topic, including in particular a proof of the following result. Let 
\begin{equation}\label{eq:Y_n^+}
\YY_n^+ \defeq \Gal(\LL_n^+/F_n^+) \cong \mathrm{Cl}(F_n^+)\otimes_{\Z}\Zp.
\end{equation}

\begin{proposition}[{\cite[Prop. 13.22]{Washington2}}] \label{prop:coinvariants}
For all $n\geq 0$, we have
\[ (\mathscr{Y}^+_\infty)_{\GG^+_n} = \YY_n^+, \]
where  $\GG^+_n = \mathrm{Gal}(F^+_\infty / F^+_n)$ and the left-hand side is the module of coinvariants.
\end{proposition}

\begin{proof}
See Proposition \ref{prop:Iwmu2}.
\end{proof}

\begin{corollary} \label{cor:Iw1}
If $p$ is a Vandiver prime, then:
\begin{itemize}
\item[(i)] $\YY_\infty^+ = 0$;
\item[(ii)] $p \nmid h_n^+$ for any $n \geq 1$;
\item[(iii)] $\EE^+_{\infty, 1} / \CC^+_{\infty, 1} = 0$.
\end{itemize}
\end{corollary}

\begin{proof}
By \eqref{eq:Y_n^+}, we deduce that $p \nmid h_n^+$ if and only if $\YY_n^+ = 0$. %$\mathscr{L}^+_n = F_n^+$. 

(i)  By Proposition \ref{prop:coinvariants}, if $p \nmid h_1^+$, then $0 = \YY_1^+ = (\mathscr{Y}_\infty^+)_{\GG_0} = 0$. By Nakayama's lemma, this implies that $\mathscr{Y}_\infty^+ = 0$. 

(ii) Combining (i) with Proposition \ref{prop:coinvariants} shows $\YY_n^+ = 0$, hence the result.

(iii) In Theorem \ref{thm:cyclo units class number} we saw that $[\VV_n^+:\DD_n^+] = h_n^+$, which is prime to $p$ by (ii). We claim further that  
\begin{equation}\label{eq:prime to p}
    [\VV_{n,1}^+ : \DD_{n,1}^+] \text{ is prime to }p.
\end{equation}
Indeed,  note $\DD_{n,1}^+ = \VV_{n,1}^+ \cap \DD_{n}^+$ by definition; so applying the isomorphism theorem $S/(S\cap N) \cong SN/N \leq G/N$ to the subgroups $S = \VV_{n,1}^+$ and $N = \DD_n^+$ of $G = \VV_n^+$, we see that $\VV_{n,1}^+/\DD_{n,1}^+$ is isomorphic to a subgroup of $\VV_n^+/\DD_n^+$, hence has order dividing $h_n^+$, which is prime to $p$.  Hence there is an exact sequence
\[ 0 \to \DD^+_{n, 1} \to \VV^+_{n, 1} \to W_n \to 0, \]
where $W_n$ is a finite group of order prime to $p$. Applying $-\otimes_{\Z}\Zp$ to every term, we get
\[ \DD^+_{n, 1} \otimes_\Z \zp \cong \VV^+_{n, 1} \otimes_\Z \zp. \]
Recall now that $\EE^+_{n, 1}$ (resp. $\CC^+_{n, 1}$) is by definition the $p$-adic closure of $\VV^+_{n, 1}$ (resp. $\DD^+_{n, 1}$) inside $\UU^+_{n, 1}$, and that $\CC^+_{n, 1} \subseteq \EE^+_{n, 1}$. By Lemma \ref{lem:closure}, we have natural surjections $\DD^+_{n, 1} \otimes_{\Z} \zp \to \CC^+_{n, 1}$ and $\VV^+_{n, 1} \otimes_{\Z} \zp \to \EE^+_{n, 1}$, making the diagram
\[
\xymatrix{
\DD_{n,1}^+ \otimes_{\Z} \Zp \ar[r]^{\sim}\ar@{->>}[d] & \VV_{n,1}^+\otimes_{\Z} \Zp\ar@{->>}[d]\\
\CC_{n,1}^+ \ar[r] & \EE_{n,1}^+
}
\]
commute. Thus the inclusion $\CC^+_{n, 1} \to \EE^+_{n, 1}$ is surjective, so an isomorphism. Taking inverse limits yields $\CC_{\infty,1}^+ \cong \EE^+_{\infty,1}$, which finishes the proof.
\end{proof}

We can now easily finish the proof of Iwasawa Main Conjecture for Vandiver primes.

\begin{proof}
By Corollaries \ref{cor:CFTunits2} and \ref{cor:Iw1}(i,iii) (for the first isomorphism) and Theorem \ref{thm:iwasawa} (for the second), we have 
\[ \mathscr{X}^+_\infty \cong  \UU^+_{\infty, 1} / \CC^+_{\infty, 1} \cong \Lambda(\GG^+) / I(\GG^+) \zeta_p. \] 
In particular, 
\[ \mathrm{ch}_{\Lambda(\GG^+)}(\mathscr{X}_\infty^+) = \mathrm{ch}_{\Lambda(\GG^+)} \big( \Lambda(\GG^+) / I(\GG^+) \zeta_p \big) = I(\GG^+) \zeta_p.\qedhere \]
\end{proof}

\begin{remark}
Conjecturally, every prime is a Vandiver prime, and under this conjecture we have proved the full Iwasawa Main Conjecture. %(Contrast to the less strong condition of being a \emph{regular prime}, that is, a prime such that $p \nmid h_1$. Whilst irregular primes are somewhat rare, they do exist, and moreover there are infinitely many of them).
The conditional proof above was due to Iwasawa himself. The first full proof of the Iwasawa Main Conjecture was given by Mazur--Wiles \cite{MW84}. For a description of another proof, using Euler systems and due to Kolyvagin, Rubin, and Thaine, see \cite{CS06} and \cite{Lan90}.
\end{remark}

\subsection{Generalisations: Selmer groups, $p$-adic $L$-functions and Iwasawa--Greenberg Main Conjectures} \label{sec:greenberg selmer}

Our focus throughout has been on Iwasawa's original Main Conjecture. We conclude with a sketch of a formulation due to Greenberg \cite{Gre89} of a Main Conjecture for more general Galois representations, which, for the trivial Galois representation, recovers Theorem \ref{IMC}. In order to do this, we first need to introduce Selmer groups. These are fundamental objects in arithmetic, lying at the core of two very important conjectures concerning $L$-functions: Iwasawa Main Conjectures and Bloch--Kato conjectures.  The reader interested in this beautiful theory can learn more from the  references \cite{Gre89}, \cite{Bel09},  \cite{Rub00}, \cite{Kat04}, \cite{PerrinRiou}, and \cite{Ski18}.

\subsubsection{Selmer groups over $F_\infty^+$}\label{sec:greenberg 1}

Selmer groups are objects that generalise one of the sides of the Iwasawa Main Conjecture. Let us start with a general definition. For compatibility with our earlier study, we will work over the field $F_\infty^+ = \Q(\mu_{p^{\infty}})^+$, which we denote $\cF$ (to ease notation). This differs from \cite{Gre89}, who works over the cyclotomic $\Zp$-extension $\Q_\infty/\Q$ (cf.\ Remark \ref{rem:even v odd}). We continue to take coefficients in a fixed finite extension $L/\Qp$.

\begin{definition}
 \begin{enumerate}
 \item Let $M$ be a topological $\mathcal{O}_L$-module equipped with a continuous $\mathcal{O}_L$-linear action of $\mathscr{G}_{\cF}$, which is unramified outside a finite set of places. 
 \item A \emph{Selmer structure} $\mathcal{L} = (\mathcal{L}_v)_v$ for $M$ over $\cF$ is a choice of subspace $\mathcal{L}_v \subseteq H^1(\cF_v, M)$ for every finite place $v$ of $\cF$ such that, for almost all $v$, we have $\mathcal{L}_v = H^1_{\mathrm{ur}}(\cF_v, M)$ \footnote{Here, the \emph{unramified cohomology groups} are defined as $H^1_{\mathrm{ur}}(\cF_v, M) \defeq \mathrm{ker}(H^1(\cF_v ,M) \to H^1(I_{\cF_v}, M))$, where $I_{\cF_v}$ denotes the inertia group of $\cF_v$ and the map is restriction.}. 
 \item Given a Selmer structure $\mathcal{L} = (\mathcal{L}_v)_v$ for $M$, we define the \emph{Selmer group} to be
\[ H^1_\mathcal{L}(\cF, M) = \ker \left( H^1(\cF,M) \to \bigoplus_v H^1(\cF_v, M) / \mathcal{L}_v \right), \]
where $v$ runs over every finite place of $\cF$.
\end{enumerate}
\end{definition}

In what follows, we will only be interested in Selmer structures with unramified local conditions for all places $v$ not dividing $p$. Let $T$ be a finite free $\zp$-module equipped with an action of $\mathscr{G}_{\Q}$. We let $V \defeq T \otimes_{\zp} \qp$ and $W \defeq V / T$, which is a finite free $\qp / \zp$-representation of $\mathscr{G}_\Q$. We will assume that $V$ is one of the representations that appeared in \S \ref{subsec:classicalL}, i.e.\ a Galois representation attached to some arithmetic object, and take $M = W$.

Given the general definition of a Selmer structure and Selmer group, the next step is to give reasonable subspaces $\mathcal{L}_v$ for $v \in \Sigma$. There are two main approaches to this.

\begin{definition}\label{def:selmer at p}
\begin{enumerate}
\item (Greenberg's approach)  To define interesting Selmer structures at $p$, Greenberg's approach assumes that the representation $V$ is $p$-ordinary, in the sense that there exists a saturated, finite, $\mathscr{G}_{\Q}$-invariant filtration $\mathrm{Fil}^i$ of $V$ such that the inertia group $I_{\qp}$ of $\mathscr{G}_{\qp}$ acts on the $i$th graded piece $\mathrm{Fil}^i/\mathrm{Fil}^{i+1}$ as the $i$th power of the cyclotomic character (from \eqref{eq:cyclo char}). Let $\cF = \Q(\mu_{p^\infty})^+$ as before and let $v_p$ be the unique place of $\cF$ above $p$. Then one defines
\[ \mathcal{L}^{\mathrm{Gr}}_{v_p} \defeq \ker \Big( H^1(\mathcal{F}_{v_p}, W) \to H^1(I_{v_p}, W) \to H^1(I_{v_p}, W / \mathrm{Fil}^1 W) \Big), \]
where $\mathrm{Fil}^1 W$ denotes the image of $\mathrm{Fil}^1 V$ under the natural map $V \to W$.
\item (Bloch--Kato's approach) This approach uses technology coming from $p$-adic Hodge theory (see e.g. \cite{Ber02}, \cite{Ben22}), and we only mention it here for the sake of completeness. Letting $V$ be as before, we define
\[ H^1_f(\cF_{v_p}, V) \defeq \ker \left( H^1(\cF_{v_p}, V) \to H^1(\cF_{v_p}, V \otimes \mathbf{B}_{\rm cris}) \right), \]
where $\mathbf{B}_{\rm cris}$ denotes Fontaine's ring of crystalline periods. Then one defines
\[ \mathcal{L}^{\mathrm{BK}}_{v_p} \defeq \mathrm{Im} \left( H^1_f(\cF_{v_p}, V) \to H^1(\cF_{v_p}, W) \right), \]
where the map is induced by the natural map $V \to W$.
\end{enumerate}
\end{definition}

\begin{remark}
Greenberg Selmer groups and Bloch--Kato Selmer groups coincide in some fundamental examples. For a general relation between them, we refer the reader to \cite{kato98} and \cite[\S 2.4.7]{PerrinRiou}.
\end{remark}

\subsubsection{The module $\XX_\infty^+$ via Selmer groups}
 
We now reinterpret the space $\XX_\infty^+$ appearing in Theorem \ref{IMC} in terms of a Greenberg Selmer group. We follow \cite[\S1]{Gre89}. Recall $F_\infty = \Q(\mu_{p^\infty})$, and the cyclotomic character $\chi : \mathrm{Gal}(F_\infty/\Q) \to \Zp^\times$. Note that $\mathrm{Gal}(F_\infty/\Q)$ is the quotient of $\mathscr{G}_{\Q}$ by $\mathscr{G}_{F_\infty}$, so that we may consider $\chi$ as a character of $\mathscr{G}_{\Q}$ that is trivial on $\mathscr{G}_{F_\infty}$. For $n \in \Z$, we consider the representation $T_n \defeq \zp(n)$, that is, the space $\Zp$ upon which $\mathscr{G}_{\Q}$ acts by $\chi^n$ (cf.\ Definition \ref{def:Zp(1)}). Let $V_n \defeq \qp(n)$ and $W_n \defeq V_n/T_n =  (\qp / \zp)(n)$. 

By the above remark, $\mathscr{G}_{F_\infty}$ acts trivially on $W_n$, and we thus have 
\[
    H^1(F_\infty, W_n) = \mathrm{Hom}_{\mathrm{cts}}(\mathscr{G}_{F_\infty}, W_n).
\]
As $\Gal(F_\infty/F_\infty^+) = \{1,c\} \cong \{\pm 1\}$ (with the non-trivial element given by complex conjugation) and $p$ is odd, the inflation-restriction sequence gives
\[
    H^1(\cF,W_n) = H^1(\mathscr{G}_{F_\infty},W_n)^{\{\pm1\}} = \mathrm{Hom}_{\mathrm{cts}, \{\pm 1\}}(\mathscr{G}_{F_\infty}, W_n),
\]
recalling $\cF = F_\infty^+$ and where the subscript $\{\pm1\}$ denotes homomorphisms equivariant for $\Gal(F_\infty/F_\infty^+)$, acting on $\mathscr{G}_{F_\infty}$ as in Remark \ref{rem:inner automorphisms}. We can thus describe the Greenberg Selmer group over $\cF$ in the simpler language of group homomorphisms, rather than Galois cohomology classes.

We can refine this via local/global conditions. To do so, let $\sigma \in H^1(\cF,W_n)$. Then:

\begin{enumerate}

\item As $W_n$ is abelian and $p$-power torsion, we find 
\[
    H^1(\cF,W_n) = \mathrm{Hom}_{\mathrm{cts}, \{\pm 1\}}(\Gal(F_\infty^{\mathrm{ab},\mathrm{pro}-p}/F_\infty), W_n),
\]
where $F_\infty^{\mathrm{ab},\mathrm{pro}-p}$ is the maximal abelian pro-$p$ extension of $F_\infty$.

\item (Finite primes away from $p$). Since the local condition at finite places $v\nmid p$ is the unramified condition, we deduce that $\sigma$ is unramified at all $v \nmid p$. Then $\sigma$ descends to a homomorphism on $\XX_\infty = \mathrm{Gal}(\MM_\infty/F_\infty)$, for $\MM_\infty$ the maximal abelian pro-$p$ extension of $F_\infty$ unramified outside $\mathfrak{p}$ as in \S \ref{SecLambdamods}.

\item (The prime above $p$). At $p$, we have $\cF_{v_p} = K_\infty^+ = \Qp(\mu_{p^\infty})^+$. The filtration required by Greenberg is given by
\[
    \mathrm{Fil}^i \Qp(n) = \left\{\begin{array}{cl} \Qp(n) & \text{ if } i \leq n, \\
    0 & \text{ if } i > n.\end{array}\right.
\]
From the definition, we see that
\[
    \mathcal{L}^{\mathrm{Gr}}_{v_p} = \left\{\begin{array}{cl} H^1(K_\infty^+, W_n) & \text{ if } n \geq 1\\
    H^1_{\mathrm{ur}}(K_{\infty}^+,W_n) & \text{ if } n \leq 0.
    \end{array}\right.
\]
Thus if $n \geq 1$, the local condition at $p$ is empty, while for $n \geq 0$ the local condition at $p$ means $\sigma$ is also unramified at $p$, so it descends further to $\YY_\infty = \mathrm{Gal}(\LL_\infty/F_\infty)$ (for notation as in \S \ref{SecLambdamods}). We conclude that
\[
H^1_{\mathcal{L}^{\mathrm{Gr}}}(F_\infty, W_n) = \left\{\begin{array}{cl} \mathrm{Hom}_{\mathrm{cts}, \{\pm1\}}(\XX_\infty, W_n) & \text{ if } n \geq 1\\
    \mathrm{Hom}_{\mathrm{cts}, \{\pm1 \}}(\YY_\infty, W_n) & \text{ if } n \leq 0.
    \end{array}\right.
\]

\item (Equivariance for $\{\pm 1\}$). As $p$ is odd, the action of $\Gal(F_\infty/F_\infty^+) = \{1,c\}$ yields a decomposition $\XX_\infty = (\XX_\infty)^{c = 1} \oplus (\XX_\infty)^{c = -1}$. There is a similar decomposition for $\YY_\infty$. Note also that $c$ acts on $W_n$ as $(-1)^n$. We see that if $n > 0$, we have
\[
H^1_{\mathcal{L}^{\mathrm{Gr}}}(F_\infty, W_n) = \mathrm{Hom}_{\mathrm{cts}}((\XX_\infty)^{c = (-1)^n}, W_n),
\]
that is, we see the dual of $(\XX_\infty)^{c = 1}$ for even $n > 0$, and of $(\XX_\infty)^{c = -1}$ for odd $n > 0$. Similarly for $n \leq 0$ we have
\[
H^1_{\mathcal{L}^{\mathrm{Gr}}}(F_\infty, W_n) = \mathrm{Hom}_{\mathrm{cts}}((\YY_\infty)^{c=(-1)^n}, W_n).
\]
\end{enumerate}

As in \cite[p.6]{CS06}, the natural surjection $\XX_\infty \to \XX_\infty^+ = \Gal(\MM_\infty^+/F_\infty^+)$ induces an isomorphism $(\XX_\infty)^{c = 1} \cong \XX_\infty^+$ (and similarly $(\YY_\infty)^{c = 1} \cong \YY_\infty^+$).  Combining all of the above, we conclude that for even positive $n$, we have
\[
    H^1_{\mathcal{L}^{\mathrm{Gr}}}(F_\infty, W_n) = \mathrm{Hom}_{\mathrm{cts}}(\XX_\infty^+, W_n),
\]
(a twist of) the Pontryagin dual of the module $\XX_\infty^+$ appearing in Theorem \ref{IMC}.

\subsubsection{The Iwasawa--Greenberg Main Conjecture}

Let $T, V$ and $W$ be representations as in \S\ref{sec:greenberg 1}.  We also let $V^\vee = \mathrm{Hom}_{\mathrm{cts}}(V,\Qp(1))$ be the Tate dual of $V$. In particular $V$ is ordinary at $p$. Inspired by the Iwasawa Main Conjecture and by an analogous conjecture of Mazur for ordinary elliptic curves, in \cite{Gre89}, Greenberg described a Main Conjecture for $V$, which we now state.

Attached to $V$ is an $L$-function $L(V,s)$, and an Euler factor at infinity $L_\infty(V,s)$ (called the Gamma factor in \cite{Gre89}).  This involves a product of certain translates of the usual Gamma function $\Gamma(s)$ and  has no zeros. We let $r_V$ denote the order of the pole of $L_\infty(V,s)$ at $s=1$. For example, if $V = \Qp(n)$, then $L(V, s) = \zeta(s-n)$ and $L_\infty(V, s) = \pi^{-(s-n)/2}\Gamma((s-n)/ 2)$. The poles of the last function are $s = n, n-2, n-4, \ldots$.

Coates--Perrin-Riou conjectured in \cite{coatesperrinriou89} that there should exist a $p$-adic $L$-function for $V$, which in this context lies in the field of fractions of $\Lambda(\Gamma^+)$. We comment more on this conjecture in Appendix \ref{sec:appendix further generalisations}.

\begin{conj} [Greenberg]
The following assertions hold.
\begin{itemize}
\item[(i)] $H^1_{\mathcal{L}^{\mathrm{Gr}}}(\cF, W)$ has $\Lambda(\Gamma^+)$-corank equal to $r_V$. 
\item[(ii)] If $r_V = r_{V^\vee} = 0$, then the characteristic ideal of the Pontryagin dual of $H^1_{\mathcal{L}^{\mathrm{Gr}}}(\cF, W)$ (as a $\Lambda(\Gamma^+)$-module) coincides with the ideal generated by the $p$-adic $L$-function of $V$.
\end{itemize}
\end{conj}

\begin{example}\label{ex:greenberg general 2}
Suppose $n>0$ is even; then above we showed $H^1_{\mathcal{L}}(\cF,W_n) = \mathrm{Hom}_{\mathrm{cts}}(\XX_\infty^+,W_n)$. This has Pontryagin dual $\XX_\infty^+(-n)$, the space $\XX_\infty^+$ with $\Lambda(\Gamma^+)$-action twisted by $\chi^{-n}$. In this case, the $p$-adic $L$-function in Greenberg's conjecture is $\partial^n\zeta_p$, the $n$th twist of Kubota--Leopoldt, so we see that Greenberg's conjecture is essentially a (twist of) Theorem \ref{IMC}.
\end{example}

\begin{remark}\label{rem:greenberg general 2}
We described Theorem \ref{IMC} as `the Main Conjecture for $V = \Qp$'. Although they are equivalent, strictly speaking, the case $n = 0$ doesn't fit into Greenberg's picture. Indeed, the underlying assumption does not hold: here $\Qp^\vee = \Qp(1)$, so $L_\infty(\Qp^\vee,s) = L_\infty(\Qp,1-s)$, hence $r_{\Qp^\vee} = 1$. In Greenberg's terminology, $L(\Qp(n),1) = \zeta(1-n)$ and $L(\Qp(n)^\vee,1) = \zeta(n)$ are both critical only when $n$ is even and (strictly) positive, or $n$ is odd and negative, so the case $n = 0$ is not covered. %Again, this is an artifact of the pole of $\zeta(s)$ at $s=1$, making $\zeta(s)$ an exceptional case.
\end{remark}

\begin{remark}
There exist more general, even non-commutative, statements of the Main conjecture. We refer the interested reader to \cite{KatoIwasawa} and \cite{FukayaKato}.
\end{remark}

\vspace{15pt}
\newpage

\appendix
\begin{center}\scshape{{\LARGE Appendix}}\\
\end{center}
\addcontentsline{toc}{part}{Appendix}

\vspace{-10pt}

\section{Iwasawa's $\mu$-invariant}\label{sec:mu invariant}

We end these notes by giving a flavour of further topics in classical Iwasawa theory, introducing the $\mu$ and $\lambda$-invariants of a $\zp$-extension. In proving Iwasawa's theorem on the $\mu$ and $\lambda$-invariants, we develop techniques that can be used to show that the modules appearing in the exact sequence of Corollary \ref{cor:CFTunits2} are finitely generated torsion modules over the Iwasawa algebra, a part of the general Iwasawa Main Conjecture (beyond the Vandiver case that we have already proved). \\

The following results will hold for an arbitrary $\zp$-extension of number fields, although we will only prove them under some hypotheses that slightly simplify the proofs. 

\begin{definition}
Let $F$ be a number field. A $\zp$-extension $F_\infty$ of $F$ is a Galois extension such that $\mathrm{Gal}(F_\infty / F) \cong \zp$. If $F_\infty / F$ is a $\zp$-extension, we denote by $F_n$ the subextension such that $\Gal(F_n/F) \cong \Z/p^n\Z$.
\end{definition}

Recall first that any number field has at least one $\zp$-extension, the \emph{cyclotomic $\zp$-extension}. Indeed, by Galois theory $\mathrm{Gal}(F(\mu_{p^\infty}) / F)$ is an open subgroup of $\mathrm{Gal}(\Q(\mu_{p^\infty}) / \Q) \cong \Zp^\times$, and hence contains a maximal quotient isomorphic to $\zp$ (specifically, the quotient by the finite torsion subgroup $\mu_{p-1}$). The corresponding field (under the fundamental theorem of Galois theory) is the cyclotomic $\Zp$-extension.
%\begin{definition}
%Let $F_\infty/F$ be a $\Zp$-extension. For each $n$, let $F_n$ be the unique subextension of $F_\infty/F$ such that 
%\[\Gal(F_n/F) \cong \Z/p^n\Z.\]
%\end{definition}
\begin{example}
Let $F = \Q(\mu_p)$. Then $F_\infty = \Q(\mu_{p^\infty})$ is the cyclotomic $\Zp$-extension of $F$, and 
\[F_n = \Q(\mu_{p^{n+1}}),\]
noting that earlier we denoted this field $F_{n+1}$. The cyclotomic $\Zp$-extension of $\Q$ is the field $(F_\infty)^{\mu_{p-1}}$ in $F_\infty$ fixed by the torsion subgroup $\mu_{p-1} \subset \Gal(F_\infty/\Q)$.
\end{example}

\emph{Leopoldt's conjecture} states that the number of independent $\zp$-extensions of a number field $F$ is exactly $r_2 + 1$ , where $r_2$ is the number of complex embeddings of $F$. In particular, the conjecture predicts that any totally real number field possesses a unique $\zp$-extension (the cyclotomic one). Whilst the conjecture remains open for general number fields, it is known in the case that $F$ is an abelian extension of $\Q$ or an abelian extension of an imaginary quadratic field (see \cite[Theorem 10.3.16]{NSW99}).

\subsection{Iwasawa's theorem}

Let $F$ be a number field, $F_\infty / F$ a $\zp$-extension, $\Gamma = \Gamma_F = \mathrm{Gal}(F_\infty / F) \cong \zp$ and $\gamma_0$ a topological generator of $\Gamma_F$. Using this choice of $\gamma_0$, we identify $\Lambda(\Gamma)$ with $\Lambda \defeq \zp\lsem T\rsem$ by sending $\gamma_0$ to $T + 1$ (when $\gamma_0$ is sent to $1$ by the isomorphism $\Gamma \cong \Zp$, this is simply the Mahler transform, but this identification holds for any $\gamma_0$). Let $\mathscr{L}_n$ (resp. $\mathscr{L}_\infty$) be the maximal unramified abelian $p$-extension of $F_n$ (resp. pro-$p$ extension of $F_\infty$), and write 
\[
\mathscr{Y}_{F, n} = \YY_n \defeq \Gal(\mathscr{L}_n / F_n) = \mathrm{Cl}(F_n) \otimes \zp,
\] 
which is the $p$-Sylow subgroup of the ideal class group of $F_n$. Set 
\[
\YY_\infty = \YY_{F, \infty} \defeq \varprojlim_n \YY_{F, n}. 
\] 
Write $e_n = v_p(\# \YY_n  )$ for the exponent of $p$ in the class number of $F_n$. The following theorem is the main result we intend to show in this section.

\begin{theorem} [Iwasawa] \label{Iwmu}
There exist integers $\lambda \geq 0$, $\mu \geq 0$, $\nu \geq 0$, and an integer $n_0$, such that, for all $n \geq n_0$, we have
\[ e_n = \mu p^{n} + \lambda n + \nu. \]
\end{theorem}

\begin{rem} 
\begin{enumerate}
\item This is yet another typical example of the power of Iwasawa theory, in which we derive information at finite levels by considering all levels simultaneously. 

There are two basic steps in the proof of Theorem \ref{Iwmu}. We first show that the module $\mathscr{Y}_{F, \infty}$ is a finitely generated torsion $\Lambda(\Gamma)$-module. Using the structure theorem of $\Lambda(\Gamma)$-modules (Theorem \ref{thm:structure theorem}), we study the situation at infinite level, and then we transfer the result back to finite level to get the result.
\item We will only describe the proof for the case where the extension $F_\infty / F$ satisfies the following hypothesis: there is only one prime $\mathfrak{p}$ of $F$ above $p$, and it ramifies completely in $F_\infty$. The reduction of the general case to this case is not difficult, and is contained in \cite[\S13]{Washington2}. This assumption covers our cases of interest; in particular, it applies if $F = \Q(\mu_{p^m})$ or $F = \Q(\mu_{p^m})^+$ for some $m \geq 0$ and $F_\infty / F$ is the cyclotomic $\zp$-extension.
\end{enumerate}
\end{rem}

\subsubsection{First step}

The first step of the proof of Theorem \ref{Iwmu} consists in showing (Proposition \ref{prop:Iwmu3}) that the module $\YY_{\infty}$ is a finitely generated $\Lambda(\Gamma)$-module. Then Proposition \ref{prop:Iwmu2} will allow us to recover each $\YY_{n}$ from the whole tower $\YY_{\infty}$. We then use a variation of Nakayama's lemma to conclude.\\

Since $\mathfrak{p}$ is totally ramified in $F_\infty$, and $\Ln$ is unramified over $F_n$, we deduce that $F_{n + 1} \cap \Ln = F_n$ and hence  
\begin{align*}
\YY_n = \mathrm{Gal}(\Ln / F_n) &= \mathrm{Gal}(\Ln F_{n+1} / F_{n + 1})\\
&= \YY_{n + 1} / \mathrm{Gal}(\mathscr{L}_{n + 1} / \Ln F_{n+1}),
\end{align*}
showing that $\YY_{n+1}$ surjects onto $\YY_n$. The module $\YY_\infty$ carries the natural Galois action of $\Lambda = \Lambda(\Gamma)$, and under the identification $\Lambda \cong \Zp\lsem T\rsem$, the polynomial $1 + T \in \Lambda$ acts as $\gamma_0 \in \Gamma$.

Let $\tilde{\mathfrak{p}}$ be a prime of $\Linf$ above $\mathfrak{p}$, and write
 \[I \subseteq G \defeq \mathrm{Gal}(\Linf / F)\]
for its inertia group. Since $\Linf / F_\infty$ is unramified, all of the inertia occurs in the subextension $F_\infty/F$. Accordingly $I \cap \Yinf = 1$ and since $F_\infty / F$ is totally ramified at $\mathfrak{p}$, the inclusion $I \hookrightarrow G / \Yinf \cong \Gamma$ is surjective, and hence bijective. We deduce that
\[ G = I \Yinf = \Gamma \Yinf. \]
We've shown the following picture of extensions:
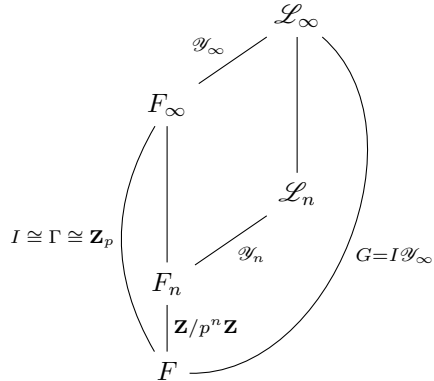
\[
\begin{tikzcd}[every arrow/.append style={dash}]
                     & \mathscr{L}_\infty   & \\
F_\infty \arrow[ur,"\mathscr{Y}_\infty"] & &\\
                              &\mathscr{L}_n \arrow[uu]                  \\
F_n \arrow[uu]\arrow[ur,"\mathscr{Y}_n"']      &                                  \\
F \arrow[u, "\Z/p^n\Z"']\arrow[uuu, bend left, "I \hspace{2pt}\cong \hspace{2pt}\Gamma \hspace{2pt}\cong \hspace{2pt}\Zp"] \arrow[uuuur, bend right = 70, "G = I\mathscr{Y}_\infty"']       &                          
\end{tikzcd}
\]

Let $\sigma \in I$ map to the topological generator $\gamma_0 \in \Gamma$ under the natural isomorphism $I \cong \Gamma$.

\begin{proposition}
Let $G'$ be the closure of the commutator of $G$. Then 
\[ G' = (\gamma_0 - 1) \cdot \Yinf = T \Yinf.\]
\end{proposition}

\begin{proof}
Recall that we have a decomposition $G = \Gamma \Yinf$. Let $a = \alpha x, b = \beta y \in G$, where $\alpha, \beta \in \Gamma$ and $x, y \in \Yinf$. Using the definition of the $\Lambda(\Gamma)$ structure of $\Yinf$, and the fact that $\Gamma$ and $\mathscr{Y}_\infty$ are abelian, we get that
\begin{align*}
ab a^{-1} b^{-1} &= \alpha x \beta y x^{-1} \alpha^{-1} y^{-1} \beta^{-1} \\
&= (\alpha x \alpha^{-1})(\alpha \beta y \beta^{-1} \alpha^{-1}) (\alpha \beta x^{-1} \beta^{-1} \alpha^{-1}) (\beta y^{-1} \beta^{-1}) \\
&= (x^\alpha)^{1 - \beta} (y^\beta)^{\alpha - 1}.
\end{align*}
Setting $\beta = 1$ and $\alpha = \gamma_0$, we deduce that $(\gamma_0 - 1) \Yinf \subseteq G'$. To see the other inclusion, write $\beta = \gamma_0^c$, where $c \in \zp$, so that $1 - \beta = - \sum_{n = 1}^{+ \infty} {c \choose n} (\gamma_0 - 1)^n = - \sum_{n = 1}^{+ \infty} {c \choose n} T^n \in T \Lambda$ and similarly for $\alpha - 1$, which allows us to conclude.
\end{proof}

Recall that the $n$th power of the Frobenius operator on $\Zp\lsem T\rsem$ is given by $\varphi^n(T) = (1 + T)^{p^n} - 1$. Let $\varphi^0(T) = T$.

\begin{proposition} \label{prop:Iwmu2}
We have \[ \Yn = \Yinf / \varphi^n(T). \]
\end{proposition}

\begin{proof}
We treat first the case $n = 0$. Since $\mathscr{L}_0$ is the maximal unramified abelian $p$-extension of $F$ and $\mathscr{L}_\infty / F$ is a $p$-extension, $\mathscr{L}_0 / F$ is the maximal unramified abelian subextension of $\Linf$. In particular, $\mathscr{Y}_0 = \mathrm{Gal}(\mathscr{L}_0 / F)$ is the quotient of $G$ by the subgroup generated by the commutator $G'$ and by the inertia group $I$ of $\mathfrak{p}$. By the above lemma and the decomposition $G = I \Yinf$, we conclude that 
\begin{align*} \mathscr{Y}_0 &= G / \langle G', I \rangle \\
&= \Yinf I / \langle (\gamma_0 - 1) \Yinf, I \rangle \\
&= \Yinf / (\gamma_0 - 1) \Yinf \\
&= \Yinf / T \Yinf. 
\end{align*}

For $n \geq 1$, we apply the arguments of the last paragraph, replacing $F$ by $F_n$ and $\gamma_0$ by $\gamma_0^{p^n}$, so that $\sigma_0$ becomes $\sigma_0^{p^n}$ and $(\gamma_0 - 1) \Yinf$ becomes 
\[(\gamma_0^{p^n} - 1) \Yinf = ((1 + T)^{p^n} - 1) \Yinf = \varphi^n(T) \Yinf,\]
which gives the result.
\end{proof}

We state next a variation of Nakayama's lemma for testing when a $\Lambda$-module is finitely generated, whose standard proof is left as an exercise.

\begin{lemma}[Nakayama's lemma; {\cite[Lemma 13.16]{Washington2}}] 
Let $\mathscr{Y}$ be a compact $\Lambda$-module. Then $\mathscr{Y}$ is finitely generated over $\Lambda$ if and only if $\mathscr{Y} / (p, T) \mathscr{Y}$ is finite. Moreover, if the image of $x_1, \hdots, x_m$ generates $\mathscr{Y} / (p, T) \mathscr{Y}$ over $\Z$, then $x_1, \hdots, x_n$ generate $\mathscr{Y}$ as a $\Lambda$-module. In particular, if $\mathscr{Y} / (p, T) \mathscr{Y} = 0$, then $\mathscr{Y} = 0$.
\label{lem:Nakayama}
\end{lemma}

Applying this in our particular situation we obtain the following result.
 
\begin{proposition} \label{prop:Iwmu3}
$\Yinf$ is a finitely generated $\Lambda$-module.
\end{proposition}

\begin{proof}
Since 
\[\varphi(T) = (1 + T)^p - 1 = \sum_{k = 1}^p {p \choose k} T^k \in (p, T),\]
the module $\Yinf / (p, T) \Yinf$ is a quotient of $\Yinf / \varphi(T) \Yinf = \mathscr{Y}_1 = \mathrm{Cl}(F_1) \otimes \zp$, the $p$-Sylow subgrop of $\mathrm{Cl}(F_1),$ which is finite. Therefore, applying Lemma \ref{lem:Nakayama}, we conclude that $\Yinf$ is a finitely generated $\Lambda$-module, as desired.
\end{proof}

\subsubsection{Second step}

Once we know that the module $\Yinf$ is a finitely generated $\Lambda$-module, we can invoke the structure theorem for these modules (Theorem \ref{thm:structure theorem}) to get an exact sequence
\[ 0 \to Q \to \Yinf \to \mathscr{A} \to R \to 0, \] where $Q$ and $R$ are finite modules and where
\[ \mathscr{A} = \Lambda^r \oplus \bigg( \bigoplus_{i = 1}^s \Lambda / (p^{m_i}) \bigg) \oplus \bigg( \bigoplus_{j = 1}^t \Lambda / (f_j(T)^{k_j}) \bigg). \]
for some integers $s, r, t \geq 0$, $m_i, k_j \geq 1$ and some distinguished polynomials $f_j(T) \in \Lambda$.\\

Recall that we want to calculate the size of $\Yn = \Yinf / \varphi^n(T)$. The following lemma reduces the problem  to calculating the size of $\mathscr{A} / \varphi^n(T)$.

\begin{lemme} \label{IwLemma0}
There exists a constant $c$ and an integer $n_0$ such that, for all $n \geq n_0$,
\[ | \Yinf / \varphi^n(T) | = p^c | \mathscr{A} / \varphi^n(T) |.\]
\end{lemme}

\begin{proof}
	The full proof of this lemma is given in \cite[Lem.\ 13.21]{Washington}; we give a sketch.
	
Consider the diagram

\[
\begin{tikzcd}
               0 \arrow[r] & \varphi^n(T) \Yinf \arrow[r] \arrow[d] & \Yinf \arrow[r] \arrow[d] & \Yinf / \varphi^n(T) \Yinf \arrow[r] \arrow[d] & 0 \\
               0 \arrow[r] & \varphi^n(T) \mathscr{A} \arrow[r] & \mathscr{A} \arrow[r] & \mathscr{A} / \varphi^n(T) \mathscr{A} \arrow[r] & 0
\end{tikzcd}
\]

%\[ \xymatrix{
%   0 \ar[r] & \varphi^n(T) \Xinf \ar[r] \ar[d]^\wr & \Xinf \ar[r] \ar[d]^\wr & \Xinf / \varphi^n(T) \ar[r] \ar[d]^\wr & 0 \\
%    0 \ar[r] & \varphi^n(T) \mathscr{A} \ar[r] & \mathscr{A} \ar[r] & \mathscr{A} / \varphi^n(T) \ar[r] & 0 } \]

By hypothesis, the kernel and cokernel of the middle vertical map are bounded. By elementary calculations and diagram chasing, one ends up showing that the kernel and the cokernel of the third vertical arrow stabilise for $n$ large enough, which is what is needed to conclude the proof.
\end{proof}

We now proceed to calculate the size of the module $\mathscr{A}$.

\begin{proposition} \label{prop:size of A}
Let
\[ \mathscr{A} = \Lambda^r \oplus \bigg( \bigoplus_{i = 1}^s \Lambda / (p^{m_i}) \bigg) \oplus \bigg( \bigoplus_{j = 1}^t \Lambda / (f_j(T)^{k_j}) \bigg), \]
for some integers $s, r, t \geq 0$ and $m_i, k_j \geq 1$ and some distinguished polynomials $f_j(T) \in \Lambda$, and write $m = \sum m_i$, $\ell = \sum k_j \, \mathrm{deg}(f_j)$. Suppose $\mathscr{A} / \varphi^n(T) \mathscr{A}$ is finite for all $n \geq 0$. Then $r = 0$ and there exist constants $n_0$ and $c$ such that, for all $n \geq n_0$, 
\[ | \mathscr{A} / \varphi^n(T)| = p^{m p^n + \ell n + c}.\]
\end{proposition}

\begin{proof}

\textit{Step 1.} First we show $r = 0$. Note that
\[
    \varphi^n(T) = (1+T)^{p^n}-1 = T^{p^n} + \sum_{k = 1}^{p^n - 1} {p^n \choose k} T^k
\]
is distinguished. We may thus apply the division algorithm from Weierstrass preparation (a $p$-adic analytic analogue of Euclid's algorithm, see \cite[\S5.2.1]{BGR}). This implies that any $f \in \Zp\lsem T\rsem$ can be written uniquely as $f(T) = q(T)\cdot ((1+T)^{p^n}-1) + r(T)$, where $r(T)$ is a polynomial of degree $\leq p^n-1$. We see 
\begin{equation}\label{eq:r=0}
    \Lambda/\varphi^n(T) = \Zp[\![T]\!]/((1+T)^{p^n}-1) \cong \{r(T) \in \Zp[T] : \mathrm{deg}(r) \leq p^n-1\}
\end{equation}
is infinite. Since $\mathscr{A} / \varphi^n(T)$ is assumed to be finite, we deduce that $r = 0$.

\medskip

\textit{Step 2.} We now deal with the second summand. By Equation \eqref{eq:r=0}, for any $k \geq 0$ we see $\Lambda/(p^k,\varphi^n(T))$ is the space of polynomials over $\Z/p^k\Z$ of degree at most $p^n-1$, whence
\[ | \Lambda / (p^k,\varphi^n(T)) |= p^{k p^n}. \]  
We deduce from this that
\[ \bigg|  \bigg(\bigoplus_{i = 1}^s \Lambda / (p^{m_i})\bigg)\Big/\varphi^n(T) \bigg| = p^{m p^n}, \] where $m = \sum_{i} m_i$.

\medskip

\textit{Step 3.} Finally, we deal with the last (and most involved) summand. Let $g(T) \in \Z_p[T]$ be a distinguished polynomial of degree $d$ (that we don't assume is irreducible, as we want to apply to $g=  f_j^{k_j}$). Letting $V = \Lambda / (g(T))$, we want to compute the order of 
\[
    V/\varphi^n(T) = \Lambda/(g,\varphi^n(T)).
\]
We will show inductively that for sufficiently large $n$, $|V/\varphi^{n+1}(T)| = p^{d}|V/\varphi^{n}(T)|$.

As $g$ is distinguished, $T^d \equiv p \cdot (\mathrm{poly}) \newmod{g}$,  where $(\mathrm{poly})$ denotes some polynomial in $\Zp[T]$. Further, $T^k \equiv p \cdot (\mathrm{poly}) \newmod{g}$ for all $k \geq d$. Let $n_0$ be such that $p^{n_0} \geq d$. We will use the following lemma.

\begin{lemma} \label{lemmatempIw} For any $n > n_0$, we have
\[
\varphi^{n+1}(T)\cdot V = p \varphi^n(T) \cdot V.
\]
\end{lemma}

\begin{proof}
First take $k \geq n_0$, allowing $k = n_0$. As $\varphi^k(T)$ is distinguished, we see that
\[
    \varphi^{k}(T) = T^{p^k} + p \, (\mathrm{poly})  \equiv p \cdot (\mathrm{poly}) \newmod{g},
\]
as $p^k \geq d$. Write $\varphi^k(T) = pQ_k(T) \newmod{g}$, for $Q_k(T) \in \Zp[T]$.

We use the identity
\[
    (X^{p^{k+1}} - 1) = (X^{p^k} - 1) \cdot (X^{p^k(p-1)} + X^{p^k(p-2)} + \cdots + X^{p^k} + 1).
\]
Applying with $X = 1+T$ yields
\begin{align}
    \varphi^{k+1}(T) &= (1+T)^{p^{k+1}} - 1\notag\\
    &= ((1+T)^{p^k} - 1) \cdot \Big[(1+T)^{p^k(p-1)} + \cdots + (1+T)^{p^k} + 1\Big]\notag\\
    &= \varphi^{k}(T) \cdot \Big[(\varphi^{p^k}(T) +1)^{p-1} + \cdots + (\varphi^{k}(T) +1) + 1\Big]\notag\\
    &\equiv \varphi^{p^k}(T) \cdot \Big[ \big(pQ_k(T) + 1\big)^{p-1}  + \cdots + \big(pQ_k(T)+1\big) + 1\Big] \newmod{g}. \label{eq:Q induction}
\end{align}
Expanding the binomials, every term is divisible by $p$ except the constant term 1 in each expression; but these sum to $p$. In particular, we deduce
\[
\varphi^{k+1}(T) \equiv p\varphi^k(T) \newmod{g}.
\]
Applying this with $k = n_0$, we see that 
\[
    pQ_{n_0+1}(T) \equiv \varphi^{n_0+1}(T) \equiv p\varphi^{n_0}(T) \equiv p^2Q_n(T) \newmod{g},
\]
so $Q_{n_0+1}(T) \equiv 0 \newmod{p}$. Inductively we see $Q_n(T) \equiv 0 \newmod{p}$ for all $n > n_0$.

Returning to the last line of \eqref{eq:Q induction}, now take $k = n > n_0$. As $Q_n(T) = 0 \newmod{p}$, every term in each binomial expansion is now divisible by $p^2$, again except the constant terms, which sum to $p$. In particular, we deduce
\begin{align*}
    \varphi^{n+1}(T) &\equiv \varphi^n(T) \cdot \big[p^2\cdot (\mathrm{poly}) + p\big] \\
    &\equiv p\varphi^n(T) \cdot \big[p \cdot (\mathrm{poly}) + 1\big] \newmod{g}.
\end{align*}
The term $p\cdot(\mathrm{poly}) + 1$ is a unit in $\Lambda$, so its reduction is a unit in $V = \Lambda/(g)$. We find 
\[
    \varphi^{n+1}(T) \newmod{g}\in p\varphi^n(T) \cdot V^\times,
\]
from which $\varphi^{n+1}(T) \cdot V = p\varphi^n(T) \cdot V$, completing the proof of the lemma.
\end{proof}

We now go back to the proof of the proposition. For any $n> n_0$, Lemma \ref{lemmatempIw} implies that $\varphi^{n+1}(T)V \subset pV$, and
\begin{align*} | V / \varphi^{n+1}(T) V | &= |V/pV| \cdot |pV/\varphi^{n+1}(T)V|\\ 
&= |V / p V | \cdot | p V / p \varphi^{n}(T) V |.
\end{align*}
Since $g(T)$ is distinguished of degree $d$, we have
\[ |V / p V | = | \Lambda / (p, g(T)) | = | \Lambda / (p, T^d) | = p^{d}. \]
Finally, we compute $| p V / p\varphi^{n + 1}(T) V |$. Since $(g(T), p) = 1$, multiplication by $p$ is injective on $V$ and hence 
\[
    | p V / p \varphi^{n + 1}(T) V | = | V / \varphi^{n + 1}(T) V |.
\]
Recall that $n_0$ is fixed with $p^{n_0} \geq d$. Inducting on Lemma \ref{lemmatempIw}, we find that, for any $n > n_0$, we have 
\[
    \varphi^{n + 1}(T)V = p^{n-n_0-1}\varphi^{n_0+1}(T)V.
\]
Again by Weierstrass division, we know $V$ is isomorphic to polynomials in $\Zp[T]$ of degree $\leq d-1$. This means
\[ | V / \varphi^{n + 1}(T) V | = p^{d (n - n_0 - 1)} | V / \varphi^{n_0 + 1}(T) V|. \]
Putting everything together, we deduce that
\[ | V / \varphi^{n}(T) V | = p^{nd + c},\] for some constant $c$ and all $n > n_0$. Applying this to the third summand of $\mathscr{A}$, we get
\[ 
\bigg| \bigg(\bigoplus_{j = 1}^t \Lambda / (f_j(T)^{k_j})\bigg)\Big / \varphi^n(T) \bigg| = p^{\ell n + c}, 
\] for $n \geq n_0$, where $n_0$ is such that $p^{n_0} \geq k_j\deg(f_j)$ for all $j$, $\ell = \sum_j k_j \deg(f_j)$, and $c$ is a constant. This finishes the proof of the proposition.
\end{proof}

Along the way, we have proven the following fact.

\begin{corollary} \label{Iwtor}
Let $\mathscr{Y}$ be a finitely generated $\Lambda$-module. If $\mathscr{Y} / \varphi^n(T) \mathscr{Y}$ is finite for all $n$, then $\mathscr{Y}$ is torsion.
\end{corollary}

\begin{proof}
If $\mathscr{A}$ is as in the statement of Proposition \ref{prop:size of A}, then we showed that $r = 0$ in the structure theorem for $\mathscr{Y}$. This implies that $\mathscr{A}$ is torsion; each element is annihilated by the characteristic ideal of $\mathscr{A}$. If $\mathscr{Y}$ is any finitely generated $\Lambda$-module, then $\mathscr{Y}$ is quasi-isomorphic to a module $\mathscr{A}$ as before, and as $\mathscr{A}$ is torsion, so is $\mathscr{Y}$.
\end{proof}

We can now complete the proof of Theorem \ref{Iwmu}.

\begin{proof} [Proof of Theorem \ref{Iwmu}]
Applying Lemma \ref{IwLemma0} and Proposition \ref{prop:size of A}, for $n \geq n_0$ we get
\[ | \Yn | = | \Yinf / \varphi^n(T)  \Yinf| = p^{c} | \mathscr{A} / (\varphi^n(T)) | = p^{\mu p^n + \lambda n + \nu}. \] This finishes the proof of the theorem.
\end{proof}

\subsection{Some consequences of Iwasawa's theorem}

We have already seen one application of Iwasawa's theorem (Corollary \ref{cor:Iw1}) during the statement of the main conjecture. This stated that if one class number in a $\zp$-extension is coprime to $p$, then so are all the others. We list here some further interesting applications.\\

Recall that if $A$ is a finite abelian group, then 
\[A[p] \defeq \{ x \in A :  p x = 0 \}\]
 denotes the subgroup of $p$-torsion elements and its $p$-rank $\mathrm{rk}_p(A)$ is defined to be 
\[\mathrm{rk}_p(A) = \dim_{\fp}(A / pA) = \mathrm{dim}_{\fp}(A[p]).\]
Equivalently, we can decompose $A$ uniquely as a direct sum of cyclic groups of prime power order; then the rank at $p$ is the number of direct summands of $p$-power order. 

\begin{corollary} \label{mu0prankbounded}
Let $F_\infty/F$ be a $\Zp$-extension. Then $\mu = 0$ if and only if $\mathrm{rk}_p(\mathrm{Cl}(F_n))$ is bounded independently of $n$.
\end{corollary}

\begin{proof}
 Recall that \[ \mathrm{Cl}(F_n) \otimes \zp = \Yn \defeq \Yinf / (\varphi^n(T)), \] that $\Yinf = \varprojlim \Yn$ is quasi-isomorphic to a $\Lambda$-module $\mathscr{A} = \big( \bigoplus_{i = 1}^s \Lambda / (p^{m_i}) \big) \oplus \big( \bigoplus_{j = 1}^t \Lambda / (g_j(T)) \big)$ for some integers $s, t \geq 0$, $m_i \geq 1$, and $g_i(T) \in \mathscr{O}_L[T]$ distinguished polynomials, and that we have (cf.\ the proof of Proposition \ref{prop:size of A}) an exact sequence
 \[ 0 \to C_n \to \Yn \to \mathscr{A}_n \to B_n \to 0, \] where $\mathscr{A}_n \defeq \mathscr{A} / \varphi^n(T)$, with $|B_n|$ and $|C_n|$ bounded independently of $n$. It then suffices to show that $\mu = 0$ if and only if $\dim_{\fp}(\mathscr{A}_n / p \mathscr{A}_n)$ is bounded independently of $n$.
 
 We have
 \[ \mathscr{A} / p \mathscr{A}_n = \mathscr{A} / (p, \varphi^n(T)) = \bigg( \bigoplus_{i = 1}^s \Lambda / (p, \varphi^n(T)) \bigg) \oplus \bigg( \bigoplus_{j = 1}^t \Lambda / (p, g_j(T), \varphi^n(T)) \bigg). \]
 Take $n$ large enough such that $p^n \geq \deg(g_j)$ for all $j$ and recall that $g_j$ and $\varphi^n(T)$ are distinguished polynomials (in the sense that all but their leading coefficients are divisible by $p$). The above formula then equals
 \[ \bigg( \bigoplus_{i = 1}^s \Lambda / (p, T^{p^n}) \bigg) \oplus \bigg( \bigoplus_{j = 1}^t \Lambda / (p, T^{\deg(g_j)})\bigg) = (\Z / p \Z)^{s p^n + t g}, \] 
where $g = \sum \deg(g_j)$. This shows that $\mathrm{rk}_p(\mathrm{Cl}(F_n))$ is bounded independently of $n$ if and only if $s = 0$, i.e.\ if and only if $\mu = 0$. This finishes the proof.
\end{proof}

Concerning Iwasawa's invariants, we have the following results:

\begin{theorem} [Ferrero-Washington; {\cite[\S 7.5]{Washington2}}]
If $F$ is an abelian number field and $F_\infty / F$ is the cyclotomic $\zp$-extension of $F$, then $\mu = 0$.
\end{theorem}

Finally, the following is an open conjecture of Greenberg (see \cite{Gre76}).
\begin{conj} [Greenberg]
For any totally real field $F$, and any $\Zp$-extension $F_\infty/F$, we have $\mu = \lambda = 0$. In other words, the values $\# \mathrm{Cl}(F_n)$ are bounded as $n \to +\infty$.
\end{conj}

\section{Iwasawa theory for modular forms}\label{sec:modular forms}

An interesting source of $L$-functions are those arising from automorphic forms, analytic functions on adelic groups that are symmetric for certain group actions. Dirichlet characters are algebraic automorphic forms for $\GL(1)$, so Parts I and II describe `Iwasawa theory for $\GL(1)$'.  The next natural case, of $\GL(2)$, is that of \emph{modular forms} (as explained in \cite{Wei71}). It is natural to ask how much of the theory above has an analogue for modular forms. The short answer is \textit{all of it}, though our understanding of this case is not fully complete.

\subsection{Recapping $\mathrm{GL}(1)$}

In these notes we have described three different constructions of the Kubota--Leopoldt $p$-adic $L$-function $\zeta_p$. Recall  $\Gamma^+  \defeq \mathrm{Gal}(\Q(\mu_{p^\infty})^+ / \Q) \cong \zpe / \{\pm 1\}$, and that $\Lambda(\Gamma^+)$ is its Iwasawa algebra, with ring of fractions $Q(\Gamma^+)$.

\begin{enumerate}\setlength{\itemsep}{5pt}
	\item In Part I, we gave an \emph{analytic} construction, a $p$-adic pseudo-measure $\zeta_p^{\mathrm{an}} \in Q(\Gamma^+)$ interpolating special values of the Riemann zeta function.
 
	\item In \S\ref{sec:coleman map}, we gave an \emph{arithmetic} construction, defining $\zeta_p^{\mathrm{arith}}$ via the image under $\mathrm{Col}$ of the family of cyclotomic units.
 
	\item In \S\ref{sec:IMC} we gave an \emph{algebraic} construction. We described a torsion $\Lambda(\Gamma^+)$-module $\sX_\infty^+$, with characteristic ideal $\zeta_p^{\mathrm{alg}} \defeq \mathrm{Char}_{\Lambda(\Gamma^+)}(\XX_\infty^+) \subset \Lambda(\Gamma^+)$.
\end{enumerate}

Theorem \ref{thm:coleman to kl} says $\zeta_p^{\mathrm{an}} = \zeta_p^{\mathrm{arith}}$. The Iwasawa Main Conjecture says $\zeta_p^{\mathrm{alg}} = I(\Gamma^+) \zeta_p^{\mathrm{an}}$.

\subsection{Analogues for $\mathrm{GL}(2)$}
Ultimately, versions of all of the above theory are known for sufficiently nice modular forms. Let $f$ be a cuspidal Hecke eigenform of weight $k+2$ and level $\Gamma_0(N)$, with $p|N$, and let $L(f,s)$ be its attached $L$-function. There are three ways of associating a $p$-adic $L$-function to $f$.

\subsubsection{Analytic} In the $\GL(1)$ story, the Kubota--Leopoldt $p$-adic $L$-function interpolated zeta values $L(\chi,-k)$ for $k \geq 0$ with $\chi(-1)(-1)^k = -1$. Such values are called `critical'. For a more general $L$-function $L(s)$, Deligne \cite[Def.\ 1.3]{Del79} gave an arithmetic characterisation\footnote{Precisely, Deligne asks that given a `motivic' $L$-function $L(M,s)$, then if $L_\infty(M,s)$ denotes the `Euler factor at infinity', then $s= j$ is critical for $L(M,s)$ if neither $L_\infty(M,s)$ nor $L_\infty(M,1-s)$ has a pole at $j$. For $\zeta(s)$, the factor at infinity is $\pi^{-s/2}\Gamma(s/2)$, and the $\Gamma$-function has poles at negative even integers. We deduce the critical values of $\zeta(s)$ are exactly the negative odd integers (as seen by Kubota--Leopoldt) and positive even integers, which relate to the negative odd integers through the functional equation for $\zeta(s)$.} of which values $s$ should be critical for $L(s)$. For the modular form $f$, his criterion says the critical values of $L(f,s)$ are $L(f,\chi,j+1)$ for $\chi$ any Dirichlet character and $0 \leq j \leq k$. 

The analytic $p$-adic $L$-function is an element $L_p^\mathrm{an}(f)$ in space of $p$-adic distributions $\ccD(\Zp^\times)$ which interpolates these critical values. In particular, we have the following.

\begin{theorem}\label{thm:intro 1}
	Let $\alpha_p$ denote the $U_p$ eigenvalue of $f$. If $v_p(\alpha_p) < k+1$, then there exists a unique locally analytic distribution $L_p^{\mathrm{an}}(f) \in \mathscr{D}^{\rm la}(\Zp^\times)$ on $\Zp^\times$ such that
	
	\begin{itemize}\setlength{\itemsep}{5pt}
		\item $L_p^{\mathrm{an}}(f)$ has growth of order $v_p(\alpha_p)$.
		\item For all Dirichlet characters $\chi$ of conductor $p^n$, and for all $0 \leq j \leq k$, we have
		\begin{align*}
			L_p^{\mathrm{an}}(f,\overline{\chi},j+1) &= \int_{\Zp^\times} \chi(x)x^j \cdot L_p^{\mathrm{an}}(f)\\
			&= -\alpha_p^{-n}\cdot \left(1-\chi(p)\frac{p^j}{\alpha_p}\right)\cdot\frac{G(\chi)\cdot j!\cdot p^{nj}}{(2\pi i)^{j+1}}\cdot \frac{ L(f,\overline{\chi},j+1)}{\Omega_f^\pm}.
		\end{align*}
	\end{itemize}
\end{theorem}

We see that $L_p^{\mathrm{an}}(f)$ is, in this generality, only a locally analytic distribution (in the sense of \S\ref{sec:locally analytic}). We note, however, that if $f$ is $p$-ordinary (i.e.\ if $v_p(\alpha_p) = 0$) then the growth condition implies that $L_p^{\mathrm{an}}(f)$ lies in the subspace of $p$-adic measures on $\Zp^\times$. This theorem was first proved in \cite{MSD74, AV75, Vis76}.

If $\alpha_p \neq 0$, then we know that $v_p(\alpha_p) \leq k+1$, but the above theorem does not handle the case $v_p(\alpha) = k+1$. Subsequent work of Pollack--Stevens \cite{PS12} and Bella\"iche \cite{Bel12} means we have $p$-adic $L$-functions in this case too, though the statement is slightly different.

The case of $\alpha_p = 0$ is known as the \emph{infinite slope} case. For certain particular cases, a construction of a $p$-adic $L$-function attached to elliptic curves of bad additive reduction (where $a_p = 0$) can be found in \cite{Delbourgo}. For arbitrary modular forms, partial $p$-adic $L$-functions with good interpolation properties have been constructed in \cite{RodriguesPhiGamma}, using Kato's Euler system and extending Perrin-Riou's big logarithm maps (as discussed in \S \ref{sec:pr log}).

\subsubsection{Arithmetic}
As we sketched in \S\ref{sec:pr log}, the appropriate generalisation of the arithmetic construction goes through Galois representations and Euler systems. Attached to a modular form $f$, we have a Galois representation $V_f$, constructed by Deligne inside the \'etale cohomology of the modular curve, in which we can pick a Galois-stable integral lattice $T_f$. The arithmetic $p$-adic $L$-function is then given by the following deep theorem of Kato, proved in his magisterial paper \cite{Kat04}.

\begin{theorem}[Kato]
	There exists an Euler system $\mathbf{z}_{\mathrm{Kato}}(f)$ attached to $T_f$.
%	\[	
%	\Big\{\mathbf{z}_m(f) \in \mathrm{H}^1(\Q(\mu_m), T_f(2))\Big\}
%	\]
%	attached to $T_f$.
\end{theorem} 

In the yoga described in \S\ref{sec:pr log}, we then consider the localisation of Kato's Euler system at $p$, which we still denote by the same name, and obtain a class
\[ \mathbf{z}_{\rm Kato}(f) \in H^1_{\rm Iw}(\qp, V_f). \]
%the system of classes over $\Q(\mu_{p^n})$, localise to get classes over $\Qp(\mu_{p^n})$, and then take the inverse limit to obtain an attached element of the Iwasawa cohomology
%\[
%\mathbf{z}_{\mathrm{Iw}}(f) \in \mathrm{H}^1_{\mathrm{Iw}}(\Qp,V_f).
%\]
The arithmetic $p$-adic $L$-function is then the image of $\mathbf{z}_{\mathrm{Kato}}(f)$ under the Perrin-Riou big logarithm map :
\begin{align*}
	\mathrm{Log}_{V_f} : \mathrm{H}^1_{\mathrm{Iw}}(\Qp, V_f) &\longrightarrow \mathscr{D}^{\rm la}(\Zp^\times)\\
	\mathbf{z}_{\mathrm{Kato}}(f) &\longmapsto L_p^{\mathrm{arith}}(f).
\end{align*}
The second deep theorem of \cite{Kat04} is the following \emph{explicit reciprocity law}.

\begin{theorem}
	There is an equality 
	\[
	L_p^{\mathrm{an}}(f) = L_p^{\mathrm{arith}}(f) \in \mathscr{D}^{\rm la}(\Zp^\times).
	\]
\end{theorem}

\subsubsection{Algebraic}

Recall that, in \S\ref{sec:greenberg selmer}, we described how the Iwasawa Main Conjecture for $\GL(1)$ (Theorem \ref{IMC}) can be generalised via Selmer groups. If $f$ is a $p$-\emph{ordinary} modular form, that is, when the eigenvalue $\alpha_p$ has $v_p(\alpha_p) = 0$, then its associated Galois representation $V_f$ is $p$-ordinary in the sense of Definition \ref{def:selmer at p}; analogously to that section, we obtain a Selmer group $H^1(\Q(\mu_{p^\infty}),V_f)$ attached to $f$ over the Iwasawa algebra $\Lambda(\Gamma)$ of $\Gamma \defeq \Gal(\Q(\mu_{p^\infty})/\Q) \cong \Zp^\times$. In \cite{Kat04}, Kato proved that this is a torsion $\Lambda$-module, and thus has a characteristic ideal 
\[
L_p^{\mathrm{alg}}(f) \defeq \mathrm{ch}_{\Lambda(\Zp^\times)}(\sX_{p^\infty}(V_f)),
\]
the \emph{algebraic $p$-adic $L$-function of $f$.} When $f$ is $p$-ordinary, the analytic/arithmetic $p$-adic $L$-function is actually a measure on $\Zp^\times$, and hence lives in the subspace $\Lambda(\Zp^\times) \subset \cD(\Zp^\times)$. 

\begin{theorem}[Iwasawa Main Conjecture for $f$]
	Under some mild additional technical hypotheses, we have
	\[
	L_p^{\mathrm{alg}}(f) = (L_p^{\mathrm{an}}(f)) \subset \Lambda(\Zp^\times).
	\]
\end{theorem}
This is a theorem of Kato \cite{Kat04} and Skinner--Urban \cite{SU14}.
There has since been much further work weakening the required hypotheses, including analogues for non-ordinary modular forms. For a description of up-to-date developments in this direction, see \cite{FouquetWan}.

\begin{remark}
The Iwasawa theory of modular forms has important applications to elliptic curves, and in particular to the BSD conjecture. An introduction to this is contained in Skinner's 2018 Arizona Winter School lectures \cite{Ski18}. 
\end{remark}

\begin{remark}
The topics described above comprise the \emph{cyclotomic} Iwasawa theory of modular forms.  There is also a rich \emph{anticyclotomic} Iwasawa theory, working over an auxiliary imaginary quadratic field, with similarly spectacular applications to BSD. This is described in \cite{BD05}, and for a more recent overview, see \cite{CGLS22}.
\end{remark}

\subsection{Further results}\label{sec:appendix further generalisations}

The three constructions above, and the equalities between them, are expected to go through in very wide generality, but there are very few cases in which the whole picture has been completed. We sketch this here. Suppose $V$ is a Galois representation, arising from a motive $M$, and corresponding (at least conjecturally) under the Langlands correspondence to an automorphic representation $\pi$. 

\begin{enumerate}\setlength{\itemsep}{5pt}
	\item \textbf{(Analytic)}. There should be a locally analytic $p$-adic distribution $L_p^\mathrm{an}(V)$ which interpolates critical values of $L(V,s)$. 
 This analytic $p$-adic $L$-function is subject to precise conjectures of Coates--Perrin-Riou and Panchishkin \cite{coatesperrinriou89, coates89, Pan94}.
	
All current techniques for proving such conjectures are automorphic; but this is already difficult, even assuming $p$-ordinarity. We illustrate this in the case of (regular algebraic, cuspidal, $p$-ordinary) automorphic representations of $\GL_n(\A_{\Q})$. 
 \begin{itemize}
\item The cases of $\GL(1)$ and $\GL(2)$ were described above.
\item The case of $\GL(3)$ was only recently handled in \cite{LW-GL3}. Constructions in the special case where $\pi$ is a symmetric square lift were given (decades earlier) in \cite{Sch88, Hid90}.
\item No general construction is known for any $n\geq4$. The best known results are in further `degenerate' cases where $\pi$ really comes from a different group (e.g.\ \cite{AG94}). 
 \end{itemize}

 For other groups, there are also many results; e.g.\ \cite{EHLS16} for unitary groups, \cite{LPSZ19, Liu16} for Siegel modular forms, and \cite{KMS00, Jan-Non-Abelian} for $\GL_{n+1}\times\GL_n$. The general picture remains, however, very fragmented.

We do not claim to give anything approaching a comprehensive list here. Indeed, we have only scratched the surface; there are also constructions for many other groups, with more general base fields, and without assuming ordinarity. We highlight mainly that there remain a vast number of open questions in the construction of analytic $p$-adic $L$-functions.

If one drops the cuspidality assumption, we know even less, with good results only for $\GL(2)$. Without the regular algebraic assumption, we know essentially nothing at all.

\item \textbf{(Arithmetic)}.  We also expect Euler systems, in the sense of \cite{Rub00}, to exist in great generality, but known examples are scarcer still. Until relatively recently, Kato's Euler system and the cyclotomic units were two of only three examples of Euler systems, the other being the system of \emph{elliptic units} (though the system of \emph{Heegner points} is closely related). The last decade, though, has seen an explosion of activity in the area. Recent important examples of Euler systems include Euler systems for products of two modular forms \cite{LLZ14}, the \emph{diagonal cycles} attached to triple products of modular forms \cite{DR14}, and Euler systems for $\mathrm{GSp}_4$ \cite{LSZ17}.
	
    Where an Euler system exists, one can apply a Perrin-Riou logarithm map and extract an arithmetic $p$-adic $L$-function; but proving an explicit reciprocity law is harder still. Such reciprocity laws were studied in the Rankin--Selberg setting in \cite{KLZ17}, for diagonal cycles in \cite{DiagonalAsterisque}, and for $\mathrm{GSp}_4$ in \cite{LZ20}. For a precise summary of the double- and triple-product settings, see \cite[\S B]{LoefflerRiveroEisenstein}.

	\item \textbf{(Algebraic)}. There are Iwasawa Main Conjectures in wide generality, at least in ordinary settings, and there are many partial results towards these too. Whenever one has an Euler system with the equality $L_p^{\mathrm{an}} = L_p^{\mathrm{arith}}$, for example, one has that the corresponding Selmer group is torsion and the divisibility $L_p^{\mathrm{alg}} | (L_p^{\mathrm{an}})$. 
	
\end{enumerate}

\small
\addtocontents{toc}{\protect\vspace{10pt}}
\bibliography{master_references}{}
\bibliographystyle{alpha}

\end{document}